\documentclass[11pt]{amsart}
\usepackage[margin=0.9in]{geometry}

\usepackage[centertableaux]{ytableau}
\usepackage{graphicx, wrapfig, eso-pic, mathtools, float, enumerate}
\usepackage{amsmath, amsthm, amsfonts, amssymb, amscd, bm, accents}

\usepackage{tikz-cd} 
\usepackage{subfigure}

\usepackage{xcolor}
\definecolor{OxBlue}{HTML}{002147}

\usepackage[colorlinks]{hyperref}
\hypersetup{colorlinks = true, citecolor={OxBlue}, linkcolor={OxBlue}}

\usepackage[normalem]{ulem}
\usepackage{stmaryrd}
\usepackage[capitalise, noabbrev]{cleveref}
\usepackage[all]{xy}

\usepackage[utf8]{inputenc}
\usepackage[T1]{fontenc}

\usepackage{epigraph}

\usepackage{etoolbox}

\makeatletter
\patchcmd{\epigraph}{\@epitext{#1}}{\itshape\@epitext{#1}}{}{}
\makeatother

\newtheorem{thm}{Theorem}[section]
\newtheorem{cor}[thm]{Corollary}

\newtheorem{prop}[thm]{Proposition}
\newtheorem{lm}[thm]{Lemma}

\newtheorem{que}[thm]{Question}

\theoremstyle{definition}
\newtheorem{de}[thm]{Definition}
\newtheorem{ex}[thm]{Example}

\theoremstyle{remark}
\newtheorem{rmk}[thm]{Remark}

\def\im{\text{im}}
\def\det{\text{det}}
\def\rk{\emph{rk}}
\def\sign{\text{sign}}
\def\Hom{\text{Hom}}
\def\codim{\text{codim}}
\def\ol{\overline}

\def\innt{\text{int}}
\def\th{\theta}
\def\z{\zeta}
\def\la{\langle}
\def\ra{\rangle}
\def\K{\mathbb{K}}

\def\R{\mathbb R}
\def\C{\mathbb C}
\def\N{\mathbb N}
\def\Z{\mathbb Z}

\def\H{\mathbb H}
\def\KH{K\"{a}hler }
\def\Kh{K\"{a}hler}
\def\P{{\mathbb P}}
\def\HKL{hyperk\"{a}hler}
\def\HK{hyperk\"{a}hler }

\def\AHK{almost hyperk\"{a}hler }
\def\eps{\varepsilon}
\def\a{\alpha}
\def\b{\beta}

\def\Om{\Omega}

\def\Fi{\varphi}

\def\fun{\rightarrow}
\newcommand{\fja}[1]{\xrightarrow{#1}}

\def\dejstvo{\curvearrowright}

\def\BB{Bia\l{}ynicki-Birula }

\def\M{\mathfrak{M}}
\def\Aff{\text{Aff}}
\def\CM0{\C[\M_0]}
\def\omC{\om_\C}
\def\L{\mathfrak{L}}
\def\D{\mathfrak{D}}

\def\F{\mathfrak{F}}
\def\Fmin{\F_\Fi}
\def\Fminn{\F_{\Fi^1}}
\def\Fminnn{\F_{\Fi^2}}

\def\Con1{\mathrm{Con}_1(\M,\Fi)}

\def\g{\mathfrak{g}}

\def\B{\mathcal{B}}
\def\O{\mathcal{O}}

\def\MM0{\mathfrak{M}_{\tiny{(\zeta_{\mathbb{R}},0})}(Q,{\normalfont\textbf{v}},{\normalfont\textbf{w}})}

\def\MG0{\mathcal{M}_{0,\z_\C}(Q,{\normalfont\textbf{v}},{\normalfont\textbf{w}})}

\def\iso{\cong}

\def\Sym{\text{Sym}}%^n(\C^2)}
\def\Hilb{\text{Hilb}}

\def\om{\omega}

\def\k{\mathbb{K}}
\begin{document}

%TITLE

\title
[Exact Lagrangians from contracting $\C^*$-actions]
{Exact Lagrangians from contracting $\C^*$-actions}
\author{Filip Živanović}
\address{F. T. Živanović, School of Mathematics, University of Edinburgh, EH9 3FD, U.K.}
\email{f.zivanovic@ed.ac.uk}
%\date{\today}
\thanks{Author is supported by ERC Starting Grant 850713 – HMS} %EPSRC Fellowship EP/N01815X/1.

%DEDICATION/QUOTE

%\centerline{\it {Dedicated to the memory of ?? ??}}
%\centerline{\it {To the memory of the late Professor ?? ???}}
%\centerline{\it {Dedicated to ??? ??? on the occasion of his 60th birthday}}

%\epigraph{All we have to decide is what to do with the time that is given us.}{\emph{J.R.R. Tolkien}, Lord of the Rings}

\begin{abstract}

We obtain families of non-isotopic closed exact Lagrangian submanifolds in quasi-projective holomorphic symplectic manifolds that admit contracting $\C^*$-actions. We show that the Floer cohomologies of these Lagrangians are topological in nature, recovering the ordinary cohomologies of their intersection.
Moreover, by using these Lagrangians and a version of Carrell-Goresky's integral decomposition theorem, we obtain degree-wise lower bounds on the symplectic cohomology of these spaces. %the ambient space.
\end{abstract}

\maketitle

%set the number of sectioning levels that get number and appear in the contents
\setcounter{secnumdepth}{3}
\setcounter{tocdepth}{1}

\tableofcontents            % generate and include a table of contents

\section{Introduction}\label{Intro}

\subsection{Motivation and Summary}
In this paper we consider quasi-projective holomorphic symplectic manifolds that admit contracting $\C^*$-actions.
This means that the whole manifold contracts to a compact subset, when acting by the group element $t\in \C^*$ which goes to zero. 
Moreover, we impose the condition that the $\C^*$-action acts on the holomorphic symplectic form by a positive weight. Such spaces we call Semiprojective Holomorphic Symplectic (SHS) manifolds.

The motivation to study these spaces is that they naturally occur in geometric representation theory, mathematical physics, differential and algebraic geometry.
To name a few, among such spaces are Nakajima quiver varieties, Springer resolutions, Hypertoric varieties, Moduli spaces of Higgs bundles, Hilbert schemes of points on cotangent bundles.
Moreover, the principal family of these spaces, namely Conical Symplectic Resolutions, are expected to have their partially wrapped Fukaya categories equivalent 
(at a derived level) to certain categories of quantisation modules, defined in \cite{BPW16,BLPW16}. 
In the case of Springer resolutions, the latter categories generalise the 
celebrated BGG category $\mathcal{O}$ of representations of semisimple Lie algebras.
An example of the aforementioned equivalence was given in \cite{MaS19}.

Altogether, the symplectic topology of these spaces is expected to have interesting features and relations to representation theory, which is our motivation to study it. 
A fundamental task of symplectic topology is finding Lagrangian submanifolds in a given symplectic manifold, and describing their Floer invariants. 
In this paper we do some first steps in that direction. Briefly speaking, we:
\begin{itemize}
	\item Provide a family of non isotopic exact Lagrangian submanifolds in SHS manifolds.
	\item Compute their mutual Floer cohomology groups, as graded vector spaces.
	\item Using these Lagrangian submanifolds, we infer lower bounds on the symplectic cohomology.% of SHS manifolds.
\end{itemize}

\subsection{Semiprojective varieties}

We consider first the ordinary quasi-projective varieties that have contracting $\C^*$-actions, and we prove statements about them that we will use later.
These varieties were already considered by Hausel and Rodriguez-Villegas \cite{HaR-V15}, who coined the term for them:

\begin{de}
	A \textbf{semiprojective variety} $(\M,\Fi)$ is a smooth quasi-projective variety $\M$ with an algebraic $\C^*$-action $\Fi$ satisfying two properties:
	\begin{enumerate}
		\item The set of fixed points $\M^{\Fi}$ is compact.
		\item Every point $x\in\M$ has a limit point $\displaystyle \lim_{\C^*\ni t\fun 0} t\cdot x.$ %that is automatically a fixed point.
	\end{enumerate}
\end{de}
A semiprojective variety $(\M,\Fi)$ has a distinguished subvariety called the \textbf{core}:
$$\L_\Fi:=\{x \in \M \mid \lim_{\C^*\ni t\fun \infty} t \cdot x \text{ exists} \}.$$
The important feature of the core, proved in \cite{HaR-V15}, is that it is a compact %\footnote{Meaning: compact in analytic topology} 
variety and a deformation retract of $\M.$ %In other words, unlike $\M,$ the core is compact, and has "the same" topology. %of $\M$unlike $\M$. 
In principle, two different contracting actions $\Fi_1$ and $\Fi_2,$ can have different cores. However, we prove that this is not the case, when $\Fi_1$ and $\Fi_2$ commute.
Apart from retracting to the core, the potential non-compactness of a semiprojective variety $(\M,\Fi)$ can also be solved by a completion: 

\begin{lm}\label{IntroCompletionLemma}
	Any semiprojective variety $\M$ has a smooth equivariant projective completion, that is, a smooth projective variety $Y$ with a $\C^*$-action
	which has a Zariski-open $\C^*$-invariant subvariety isomorphic to $\M.$
\end{lm}
%is that it has a smooth equivariant completion, that is, a smooth projective variety $Y$ with a $\C^*$-action, such that $\M \subset Y$ is Zariski open and $\C^*$-invariant subset whose $\C^*$-restriction is equal to $\Fi.$ %In particular, this completion shows the version of \BB decomposition for these varieties. 
This completion is not unique, and we will need its existence rather than its precise description. In particular, we do not know whether the
complement $Y \setminus \M$ can be made to be normal crossing divisor. 
Given the fixed locus decomposition into connected components $\M^{\Fi}=\bigsqcup_i \F_i,$ consider the $(t\fun \infty)$-attraction sets $\L_i:=\{x\in \M \bigm| \lim_{t\fun \infty} t\cdot x \in\F_i \}.$
The \BB decomposition theorem \cite{BB73} applied to the completion $Y$ says that the pieces $\L_i$ of the decomposition $\L=\bigsqcup_i \L_i$ are locally closed subsets 
and that the natural morphisms 
\begin{equation}\label{IntroBBbundles}
	p_i:\L_i\fun \F_i, \ x\mapsto \lim_{t\fun \infty} t\cdot x,
\end{equation}
%$$q_i:\D_i\fun \F_i, \ x\mapsto \lim_{t\fun 0} t\cdot x,$$
are %isomorphic to 
bundles with fibres being affine spaces. This allows us to prove that the core $\L$ inside the completion $Y$ satisfies the conditions of the Carrell-Goresky integral homology decomposition theorem \cite[Thm 1']{CaGo83}.

\begin{lm}\label{IntroHomologyDecompositionOfTheCore}
	Given a smooth semiprojective variety $(\M,\Fi),$ there is a decomposition of the homology of its core, i.e. an isomorphism
	$$\Phi= \oplus_i \eta_i: \bigoplus_i H_*(\F_i)[-\mu_i] \fun H_*(\L).$$ Here, $\mu_i$ are real dimensions of the fibres of attracting sets $p_i:\L_i\fun\F_i,$ and
	$\eta_i([C])=[\ol{p_i^{-1}(C)}],$ for a generic cycle $C,$ where the closure is taken in the core  $\L.$
%	In particular, this gives us the relation between the Betti numbers of $\M$ and the fixed loci:
%	\begin{equation}\label{BettiNumbers} 
%		b_k(\M)=b_k(\L) = \sum_{i} b_{k-\mu_i}(\F_i).
%	\end{equation}
\end{lm}
%We remark that the Betti numbers relation has appeared already in \cite{HaR-V15}.
This decomposition will be crucial for getting our estimates on the symplectic cohomology.
%but we will use the Betti number relation while studying the Lagrangian submanifolds beforehand. We remark that 
Another use of the completion from Lemma \ref{IntroCompletionLemma} is that it gives us a $S^1$-invariant \KH structure on a semiprojective variety, with a moment map:

\begin{lm} \label{IntrothereisaNonExactstructure} Any semiprojective variety $(\M,\Fi)$ has a compatible $S^1$-invariant %Calabi-Yau 
	\KH structure $\om_I$ that admits a moment map $H:\M \fun \R.$
	%Then the vector field of the $\R_+$-part of $\Fi$ is $X_{\R_+}=\nabla H$ and the vector field of the $S^1$-part is $X_{S^1}=X_H.$  %In particular, $X_H=I \nabla H.$
\end{lm}

We are going to use this \KH structure in proofs of our statements below, however we underline that its symplectic structure will \textbf{not} be the one which we consider while searching for exact Lagrangians. In the first place, the structure from Lemma \ref{IntrothereisaNonExactstructure} is non-exact, unless $\M$ is equal to $\C^{2n},$ for some $n\geq 0.$
In the forthcoming work with Alexander Ritter \cite{RZ22}, we construct the symplectic cohomology for these \KH structures (in the case when the moment map $H$ is proper), and prove its vanishing, which should imply the non-existence of unobstructed Lagrangians with respect to these structures, as it does in the exact case (\cite[Sec. 5]{Sei08}).
%\footnote{One cannot exclude \textit{all} Lagrangians, as we always have e.g. Lagrangian tori coming from local Darboux charts.}

\subsection{Semiprojective holomorphic symplectic manifolds}

Finally, we introduce the spaces that are the main subject of this paper:
\begin{de}\label{IntroDefSHS}
	A \textbf{Semiprojective Holomorphic Symplectic (SHS)} manifold $(\M,I,\om_\C,\Fi)$ is a smooth connected semiprojective variety $(\M,I,\Fi)$\footnote{Here the complex structure $I$ comes from the 
	structure of a smooth algebraic variety on $\M.$} with an $I$-holomorphic symplectic form $\om_\C$
	that has a positive weight with respect to the $\C^*$-action $\Fi,$ i.e. there is an integer $s>0$ such that 
	\begin{equation}\label{IntroSymplWeight}
		t \cdot \om_\C = t^s \om_\C, \forall t \in \C^*.
	\end{equation}
	Such $\C^*$-actions we will call \textbf{conical}. The integer $s$ we call the \textbf{weight} of an SHS $(\M,\om_\C,\Fi).$
	Therefore, by a \textbf{weight-1 SHS} we will mean an SHS manifold that admits a weight-1 conical action.
%	\noindent \textbf{ Notation.} In the further text we will often refer to an SHS as any of its substructures (e.g. $(\M,\Fi),$ $(\M,I),$ $(\M,\Fi,I)$ or $(\M,\om_\C)$) for brevity, 
%	depending on which part of its structure we want to emphasize.
\end{de}

Although these conditions give a lot of restrictions to geometry of $\M,$ there are many examples of SHS manifolds. The principal family amongst them are Conical Symplectic Resolutions (CSRs). Those are symplectic
$\C^*$-equivariant projective resolutions $\pi:\M\fun \M_0,$ where $\M_0$ is a Poisson normal affine variety with a $\C^*$-action. The known examples of CSRs are Nakajima quiver varieties, Springer resolutions, Hypertoric varieties, Slices of affine Grassmannian, and moreover they all belong to two big classes of spaces arising from Supersymmetric Gauge Theories, called Higgs and Coulomb branch of 3-dimensional supersymmetric gauge theories. Apart from CSRs, further examples of SHS manifolds are celebrated Moduli spaces of Higgs bundles and Hilbert schemes of points on cotangent bundles of curves. 
In addition, there may be many more SHS manifolds, as we discuss further at the end of Section \ref{CSRs}. For now, we give a working example for this introduction:

\begin{ex}\label{IntroDuVal_TypeA}
	Consider the minimal resolution of Du Val singularity of type $A_{n}.$ It is given as the minimal resolution of the quotient singularity
	$$\pi: X_{\Z/(n+1)} \fun \C^2/\Z/(n+1).$$
	Here, $\Z/(n+1)$ acts on $\C^2$ via $\eps \cdot (z_1,z_2)= (\eps z_1, \eps^{-1} z_2),$ where $\eps$ is a primitive element. 
	Thus, we have that $\C^2/\Z/(n+1)$ is isomorphic to the subvariety $V(XY-Z^{n+1})$ of $\C^3,$ given by the categorical quotient map
	$\C^2 \fun \C^3, (z_1,z_2)\mapsto (z_1^{n+1},z_2^{n+1}, z_1z_2).$
	The natural holomorphic-symplectic structure on $X_{\Z/(n+1)}$ comes as the pull-back 
	%$\om_\C:={\pi_\Gamma}^*\om$ 
	of the quotient of the standard symplectic structure on $\C^2,$ defined on $(\C^2-\{0\})/\Gamma.$ Thus, $\om_\C$ is defined away from the central fibre but it can be proved that it extends symplectically on it.
	Dilation action in $\C^2$ passes to the quotient and lifts to a conical weight-2 action on $X_{\Z/(n+1)},$ making it an SHS manifold.
	%in these coordinates is equal to $$t\cdot (X,Y,Z)=(t^{n+1} X, t^{n+1} Y, t^2 Z).$$
%	It is a weight-2 action.
	The core of this action is exactly the central fibre $\pi^{-1}(0),$ and it is classically known that it consists of 2-spheres whose graph of intersections forms an $A_n$ Dynkin tree.
	It is not hard to see that the fixed loci $\F_i$ and their corresponding attracting sets $\L_i$ are:
	\begin{enumerate}
		\item When $2\nmid n$: The fixed loci $\F_i$ are the central sphere in the $A_n$ graph, the intersection points of adjacent spheres, and their opposite points on the 
		first and the last sphere in the $A_n$-chain.
		The corresponding attracting sets $\L_i$ are the central sphere itself, isomorphic to $\C P^1,$ and the flowlines going ``from the centre towards the ends'', all isomorphic to $\C.$ 
		\item When $2\mid n$: The fixed loci $\F_i$ are the central intersection point, other intersection points of adjacent spheres, and their opposite points on the first and the last sphere in the $A_n$-chain.
		The corresponding attracting sets are the central intersection point itself, and the flowlines going ``from the centre towards the ends,'' all isomorphic to $\C.$
	\end{enumerate}
	In particular, one can verify that the equality between the ranks of homologies of the total space and of the fixed loci (appropriately shifted), coming from Lemma \ref{IntroHomologyDecompositionOfTheCore}, is satisfied in this example. 
\end{ex}

Having the holomorphic symplectic structure $\om_\C$ and a conical $\C^*$-action, %with the compatibility \eqref{IntroSymplWeight}, 
we can construct the Calabi-Yau Liouville structure on these spaces. Recall
that a real-symplectic manifold $(M,\om)$ has a \textbf{Liouville structure} $(M,\th)$ if the symplectic form $\om=d \th$ is exact and moreover, there is a 
compact submanifold $K\subset M$ with boundary, such that there is a symplectomorphism $(M\setminus int(K),\om)\iso (\Sigma \times [1,+\infty),d(R\a)),$
where $\Sigma = \partial K,\ \alpha=\th|_\Sigma,$ $R$ is the coordinate on $[1,\infty],$ and $R\alpha$ pulls back to $\theta$ via the symplectomorphism. 
The Liouville vector field $Z$ is defined by $i_Z \om=\th,$ and the Liouville skeleton is the set of points in $M$ which do not escape every compact set under the Liouville flow. 
Denoting the by $\om_J:=\mathbb{R}e(\om_\C)$ and $\om_K=\mathbb{I}m(\om_\C),$ we have:

\begin{prop} \label{IntrocanonicalLiouville} Any SHS manifold $(\M,\Fi,\om_\C)$ has a canonical family of isomorphic Calabi-Yau Liouville structures $(\M,a\th_J+b\th_K),$ 
	parametrised by $(a,b)\in \R^2\setminus \{0\}$ such that
		\begin{enumerate}
			\item The Liouville vector field $Z$ is the $1/s$-multiple of the vector field of the $\R_+$-action, where $s$ is the $\om_\C$-weight of $\Fi.$
			\item The Liouville 1-forms are $\th_J=i_Z\om_J, \th_K=i_Z \om_K.$ %thus primitives of the real and imaginary part of $\om_\C.$ 
			\item \label{IntroCoreIsSkeleton} The Liouville skeleton is the core $\L$ itself.
		\end{enumerate}
\end{prop}

Here the isomorphism between these Liouville structures is ensured by the standard Moser-type deformation argument. In particular, we get the monodromy map in $\text{Symp}(\M,\om_J)$ 
obtained by going around the origin in $\R^2\setminus \{0\}.$ Studying these monodromies for some families of SHS manifolds might be an interesting further avenue of research.
We get the Calabi-Yau condition in Proposition \ref{IntrocanonicalLiouville} by enhancing the holomorphic symplectic structure $(\om_\C,I)$ to a \textbf{compatible} almost \HK structure $(g,I,J,K),$
such that $\om_\C=\om_J+i \om_K,$ and $\om_S(\cdot,\cdot):=-g(\cdot,S\cdot),$ for $S=J,K.$ 
That gives us a smooth $S^2$-family of complex structures
$I_u=\{u_I I + u_J J + u_K K \mid (u_I,u_J,u_K)\in S^2\},$ and as an $I$-holomorphic volume form $\om_\C^{\wedge n}$ trivialises the canonical bundle, thus $c_1(T\M,I)=0,$ 
by deformation we have $c_1(T\M,I_u)=0$ as well, which in particular gives us the Calabi-Yau condition for the forms $\om_{J,K}:=\{a\om_J + b \om_K \mid (a,b)\in \R^2\setminus \{0\} \}.$

The importance of having a Liouville structure on SHS manifold $\M,$ is that we can do symplectic topology on it without the issues regarding its non-compactness of $\M.$ 
In particular, it has a well-defined symplectic cohomology $SH^*(\M):=SH^*(\M,\om_{J,K}),$ (partially) wrapped Fukaya categories, etc.
By construction, this Liouville structure in principle depends on $\Fi,$ but, at least for the case of Conical Symplectic Resolutions, we prove
that two commuting conical actions yield isomorphic Liouville structures.
\subsection[Exact Lagrangians]{Exact Lagrangians} \label{smoothcomp}

The compatibility \eqref{IntroSymplWeight} in Definition \ref{IntroDefSHS} gives the core of an SHS manifold more of importance comparing to the case of ordinary semiprojective varieties. Namely, we have the following lemma, whose proof goes back to \cite[Thm 5.8]{Nak94a}.

\begin{lm}\label{IntroCoreIsIsotropic}
	The core of an SHS manifold $(\M,\Fi,\om_\C)$ is $\om_\C$-isotropic. Furthermore, the core of a weight-1 SHS is a Lagrangian subvariety.
\end{lm}

Therefore, the core of a weight-1 SHS manifold is a $\om_\C$-Lagrangian variety, so its irreducible components are $\om_{J,K}$-Lagrangian submanifolds, when they are smooth. 
We prove that there is indeed a family of smooth ones, labelled by different conical actions:

\begin{thm}\label{IntroMinimalComponentsTheorem}
	Given a weight-1 SHS $(\M,\Fi),$ there are at least $N$ smooth irreducible components of its core, 
	where $N$ is the maximal number of commuting weight-1 conical actions on $\M.$
\end{thm}

Given a weight-1 conical $\C^*$-action $\Fi,$ we get the smooth component $\F_{\Fi}$ as the minimum of the moment map of its $S^1$-part, where we take the the \KH structure and the moment map from 
Lemma \ref{IntrothereisaNonExactstructure}. There is a unique such minimum, due to Betti number relation coming from Lemma \ref{IntroHomologyDecompositionOfTheCore}. Moreover, we prove that two different commuting actions yield different minima, concluding the theorem. The core components obtained with this method we call \textbf{minimal components}. 

Given a weight-2 Conical Symplectic Resolution, we describe a method of getting a family of commuting weight-1 conical actions in practice.
In particular, applying this method and Theorem \ref{IntroMinimalComponentsTheorem} to Nakajima Quiver Varieties of type A, we get a combinatorially-computable formula for a lower bound for the number of smooth components of its core. Together with Maffei's isomorphism \cite{Maf05} which gives a correspondence between cores of Quiver Varieties of type A and Springer fibres of type A, 
this yields a family of $N(\lambda,p)>0$ smooth components in an arbitrary generalised Springer fibre $\B_p^{\lambda}$ of type A, where the number $N(\lambda,p)$ is readily computable, given the partition $\lambda$ and composition $p.$ This is written in our forthcoming paper \cite{FZ2}. This application is interesting as it was not known whether a generalised Springer fibre has even a single smooth component. The current literature regarding smoothness of components of Springer fibres deals with the ordinary cases only, when $p=(1,\dots,1)$, and for those a lot of partial results were given in the last two decades \cite{PaRe06, Fr09b, FrMe10, FrMe11, Fr11, FrMeS-O15, BaZi08, GrZi11}.

\begin{ex}\label{IntroDuVal_Weight1actions_TypeA}
	Consider again the Du Val minimal resolution $\pi: X_{\Z/(n+1)} \fun \C^2/\Z/(n+1).$
	As mentioned, the core $\pi^{-1}(0)$ consists of Dynkin $A_n$-tree of 2-spheres.
	We are going to ``find'' all these spheres by the method from Theorem \ref{IntroMinimalComponentsTheorem}. 
	Namely, one can easily find $n$ different conical actions 
	$$t\cdot_k (X,Y,Z):=(t^{k} X, t^{n+1-k} Y, t Z), \ k=1,\dots,n $$ 
	on the quotient $\C^2/\Z/(n+1)\iso V(XY-Z^{n+1}).$
	These actions are all weight-1. To give the idea of proof, recall that $X=z_1^{n+1}$ and $Y= z_2^{n+1},$ hence 
	$t \cdot z_1=``{t^{{k} \over {n+1}}}" z_1, t \cdot z_2=``{t^{{n+1-k} \over {n+1}}}" z_1,$
	hence plugging this in $\om_\C=d z_1 \wedge d z_2$ gives weight 1.
	Moreover, as $\C^2/\Z/(n+1)$ is a conical symplectic singularity, and resolution $\pi$ is symplectic, these $\C^*$-actions lift to $X_{\Z/(n+1)},$ \cite[Lemma 5.3]{GiKalPoissonDeformation}. 
	They obviously all commute and are different, hence produce exactly $n$ different core components on $X_{\Z/(n+1)}.$
	As there are precisely $n$ core components, we see that this method exhausts \textbf{all} core components in this case.
\end{ex}

Getting to the symplectic side, recall that a smooth core component is a $\om_\C$-Lagrangian, hence a $\om_{J,K}$-Lagrangian submanifold.
Moreover, we have that it is exact, with respect to the Liouville 1-forms $\th_{J,K}:= i_{X_{\R_+}} \om_{J,K}$ (where $X_{\R+}$ is the vector field of the $\R_+$-action).
In particular, smooth core components are well-defined objects of Fukaya categories of $(\M,\om_{J,K}).$ As their classes form a basis for the middle-dimensional homology of $\M$, they are also
non-isotopic, hence we have the following corollary of Theorem \ref{IntroMinimalComponentsTheorem}:

\begin{cor}
	Given a weight-1 SHS $(\M,\Fi),$ there are at least $N$ smooth closed exact non-isotopic Lagrangian submanifolds of $(\M,\om_{J,K}),$ where $N$ is the maximal number of commuting weight-1 conical actions.
\end{cor}

\begin{rmk}
We remark that the weight-1 condition is \textit{essential} here and cannot be lifted. Looking at examples of the Higgs branch of gauge theories, such as quiver varieties and hypertoric varieties, 
in the case that one does not have a weight-1 conical action (for quiver varieties this happens iff the quiver has a loop edge), one can show that they are subcritical Stein manifolds, hence by Cieliebak \cite[p. 121]{Cie02} the symplectic cohomology vanishes, thus there are no exact Lagrangians. In particular, in these cases the core is less than half-dimensional so cannot contain Lagrangian submanifolds.
\end{rmk}
We now describe the Floer cohomologies of smooth core components, after giving remarks about grading and coefficients on their Floer cohomologies.

Recall that an SHS has a compatible almost \HK structure $(g,I,J,K).$ Now, as any $\Theta \in \{aJ+bK \mid a^2+b^2=1\}$ has a natural $\Theta$-complex volume form $\Om_{\Theta}:= {1 \over {(n/2)!}}(\om_I - i \om_{I \Theta})^{n/2},$ this gives a notion of $\Z$-graded $\om_{\Theta}$-Lagrangian submanifolds. Moreover, as smooth core components are $\om_\C$-holomorphic Lagrangians, they are also  $\Om_{\Theta}$-special Lagrangians with respect to any $\Theta.$ Altogether, they are canonically $\Z$-graded, hence their mutual Floer cohomologies are so as well. We will denote the
graded core component $L$ by $\widetilde{L}.$ Strictly speaking, this depends on the choice of \AHK metric, however most examples of SHS manifolds come with a \HK structure, which gives a canonical choice.

With regards to the coefficients, as we are in exact setup there is no need for Novikov field and we work over an ordinary field $\K.$ In general, we have to assume that $\text{char}(\K)=2,$ as the Lagrangians that we consider may not be spin.\footnote{E.g. $\C P^2$ in $T^*\C P^2.$} However, for cases of the spaces belonging to Higgs branch we can bypass this issue and still assume arbitrary coefficients, as explained in Remark \ref{OnCoefficients}. We prove:

\begin{thm}\label{IntroLagrFloerMinComps}
	%Given a weight-1 SHS manifold $\M,$ its smooth core components are non-isotopic exact Lagrangian submanifolds of $(\M,\om_{J,K}).$ %hence $HF^*(L,L)\iso H^*(L)$
	Given a minimal $\F_{\Fi}$ and a smooth core component $L$ of a weight-1 SHS $\M,$ their intersection is clean and connected, hence
	\begin{equation*}\label{IntroLagrFLoerOfTwoMinima}
		HF^*(\widetilde{\F}_{\Fi},\widetilde{L}) \iso H^{*-\mu}({\F}_{\Fi} \cap  L), %[dim_\L-\dim ], {\FZ FINISH}
	\end{equation*}
	for certain shift $\mu.$ %=\mu(\Fi^1,\Fi^2).$
	Moreover, when the $S^1$-part of $\Fi$ 
	is isometric with respect to a compatible almost \HK metric we have that
%	\begin{equation*}\label{gradingshiftcleanLagrMinimal}
		%\mu=\codim_\C(\Fminn \cap \Fminnn \subset \Fminn),
		$\mu=\codim_\C({\F}_{\Fi} \cap L \subset {\F}_{\Fi}),$
%	\end{equation*}
	hence the grading is symmetric, giving
	$$HF^*(\widetilde{\F}_{\Fi},\widetilde{L})\iso HF^*(\widetilde{L},\widetilde{\F}_{\Fi})$$ as graded vector spaces.
	In particular, when $\M$ is also \HKL, and its conical actions have isometric $S^1$-parts, we have a graded isomorphism 
	\begin{equation*}\label{IntroSHKcaseFloerAlgebra}
		\bigoplus_{L_i,L_j \text{minimal}} HF^*(L_i,L_j) \iso \bigoplus_{L_i,L_j \text{minimal}} H^*(L_i \cap L_j)[-d_{i,j}],
	\end{equation*}
	where $d_{i,j}=codim_\C ((L_i \cap L_j) \subset L_i),$ 
	and the summands correspond to each other.
\end{thm}

\begin{rmk}
We remark that these symmetric gradings on Floer cohomologies that we obtain %using the canonical gradings on core components compares well
compares well with the gradings on the Lagrangians used in the work on Symplectic Khovanov Homology \cite[Sec. 4.4]{AbS19}, chosen in such way 
in order to obtain the purity of the obtained Floer $A_\infty$-algebra.\footnote{A graded $A_\infty$-algebra $\mathcal{A}$ is called \textit{pure} if it admits 
a \textit{pure \textbf{nc}-vector field}, i.e. a Hochschild cocycle $b\in CC^1(\mathcal{A},\mathcal{A})$ having $b^0=0$ and whose endomorphism $b^1:\mathcal{A} \fun \mathcal{A}$ 
on the homology level $A=H(\mathcal{A})$ agrees with the Euler vector field, which multiplies the graded piece $A^i \subset A$ of $A$ by $i.$ 
Purity is a one way to ensure the formality of an $A_\infty$ algebra.}
The Lagrangians $L_\rho$ used there are labelled by so-called crossingless matchings $\rho,$ and they are topologically same as the Lagrangians of the core of a certain SHS $\M^{n,n}.$ 
Their ambient manifold $\mathcal{Y}_n$ is the deformation of a singular Slodowy variety $\M_0^{n,n}$ whereas $\M^{n,n}$ is its resolution. However, these are known to be symplectomorphic.
Moreover, the only minimal Lagrangian in $\M^{n,n}$ is precisely the Lagrangian which corresponds to their Lagrangian $L_{\rho_{plait}}$ obtained from certain ``plait'' matching. 
They choose the gradings of other Lagrangians $L_\rho$ such that obtained Floer cohomologies $HF^*(L_{\rho_{plait}},L_{\rho})$ have symmetric gradings, whereas we 
prove this symmetry by using the existing \HK structure on $\M^{n,n}$ and the canonical gradings given to core components being special Lagrangians.
\end{rmk}
Regarding the Floer product and higher operations on the smooth core components of SHS manifolds, we prove the following encouraging statement (recall the almost \HK notation from above):

\begin{prop}\label{IntroFloerSolnsAreConstant}
	%\textbf{(Holomorphic maps with boundary in the core are constant)}
	Consider a weight-1 SHS $\M$ and pick \textit{any} choice of complex structure $\Theta \in \{aJ+bK \mid a^2+b^2=1\}.$ %and its compatible symplectic form $\om_{\Theta}.$ 
	Given any smooth core components $L_1,\dots,L_n,$ and a Riemann surface $\Sigma$ of type-$n,$ any $\Theta$-holomorphic map $u:\Sigma\fun \M,\; u( \partial \Sigma_i) \subset L_i$ is constant.
\end{prop}
	
Here, we assume that we have chosen a compatible almost \HK structure, giving almost complex structures $J$ and $K.$ Also, we say that a Riemann surface $\Sigma$ is of \textbf{type-$n$} if its boundary $\partial \Sigma$ has punctures, which decompose it into $n$ connected pieces, denoted by $\partial \Sigma_i, i=1,\dots,n.$ This proposition tells us that operations in the Fukaya category of an SHS $\M$ that involve smooth components of the core come only from considering constant solutions. This makes
believable that the Floer product of smooth components becomes just the convolution product on the ordinary cohomology, under the isomorphisms from Theorem \ref{IntroLagrFloerMinComps}. 
Considering the examples of resolutions of two-row Slodowy varieties, this compares well with combining the work of Mak--Smith \cite{MaS19} and Stroppel--Webster \cite{SW12}, on which we comment in detail at the end of Section \ref{FloerProductsOnSHS}.

\subsection{Estimates on symplectic cohomology}
We show that existence of minimal components yields lower bounds on the ranks of the symplectic cohomology $SH^*(\M)=SH^*(\M,\om_{J,K})$ of a weight-1 SHS $\M.$ 
Symplectic cohomology in general is notoriously hard to compute explicitly, so we usually have to make do with partial information. 
We will obtain the lower bounds of its \textbf{degree-wise} ranks, hence it is important to incorporate the $\Z$-grading on $SH^*(\M),$ coming from a compatible \AHK structure.
Thus, strictly speaking it depends on the choice of \AHK metric, however as previously mentioned, most examples of SHS manifolds come with a \HK structure, which then gives the canonical choice of grading.

In order to get our lower bounds, we start with the observation by Seidel \cite[Sec. (5a)]{Sei08}, saying that 
given an exact Lagrangian submanifold $i:L \hookrightarrow \M$ inside a Liouville $\M,$ one has a commuting triangle 
\begin{equation} \label{IntroSeidelTriangle}
	\begin{tikzcd}[column sep=small]
		H^*(\M) \arrow{r}{c^*}  \arrow{rd}{i^*} 
		& SH^*(\M) \arrow{d}{\mathcal{CO}^0} \\
		& HF^*(L,L)
	\end{tikzcd}
\end{equation}
where the diagonal map becomes the restriction map $i^*:H^*(\M)\fun H^*(L),$ under the isomorphism $H^*(L,L)\iso H^*(L).$ Here the $c^*$ map is a version of the PSS map for symplectic cohomology and 
$CO^0$ is the closed-open string map that basically counts cylinders satisfying the Floer equation near one boundary and being pseudoholomorphic with boundary on $L$ on the other. 
A corollary of the integral homology decomposition Lemma \ref{IntroHomologyDecompositionOfTheCore} is that, given a minimal component $\F$ in a weight-1 SHS $\M,$ the restriction map $i^*:H^*(\M)\fun H^*(\F)$ is surjective. Thus, together with the triangle  \eqref{IntroSeidelTriangle} above, we get 

\begin{prop} Given an arbitrary minimal component $\F$ of a weight-1 SHS $\M$, for all $k\in \N_0$, $$ \text{\rk}(SH^k(\M))\geq b_k(\F).$$ \end{prop}

More precisely, we show that a block of $H^*(\M),$ that is isomorphic to $H^*(\F)$ via the restriction map, injects in $SH^*(\M)$ with the $c^*$-map. 
Thinking carefully, we can do even more: fixing a weight-1 action $\Fi$ with the fixed locus $\F=\sqcup \F_i,$ we can inject the cohomologies of all $\F_i$ which lie in minimal components to $SH^*(\M)$ via the $c^*$-map. For this to work, we have to use both the homology decompositions for the action $\Fi$ and each of those minimal components,\footnote{Which exists due to \cite{CaSo79}, as it then a holomorphic $\C^*$-action on a closed \KH manifold.} and the homology decomposition for the core,
and to prove the compatibility of those decompositions. Proving this, we get the following theorem, in which $b_i(X)$ denotes the $i$-th Betti number of a topological space $X$ and $\mu_i$ denotes 
the real dimension of fibres of the $(t\fun \infty)$-attraction bundle $\L_i \fun \F_i$ of a fixed component $\F_i.$\footnote{Recall Equation \eqref{IntroBBbundles}.}
%\footnote{Which makes sense, as one can view the fixed locus as the critical locus of the moment map for the $S^1$-part of $\Fi,$ and this moment map is a Morse-Bott function due to \cite{AtB83,Ki84}. }

\begin{thm}\label{boundonSHIntro}
	Let $(\M,\varphi)$ be a weight-1 SHS manifold, and let $\displaystyle \M^\Fi=\sqcup_i \F_i$ be the fixed locus of $\varphi$ decomposed into connected components $\F_i.$ Then
	%	the dual of the image $\displaystyle \bigoplus_{\{\a |\Labar=\Fmin\}} H_{*-\mu_\a}(\F_\a)\hookrightarrow H_*(\M)$ injects into $SH^*(\M)$ under the $c^*$ map.
	%	Thus, we get an estimate for ranks of symplectic cohomology:
	%
	%
	$$ {\rk}( SH^k(\M))\;\geq \sum_{\{\F_i\subset L\,|\,  L \text{ minimal}\}} b_{k-\mu_i}(\F_i)$$
	%	$$ \text{rk}\,( SH^k(\M,\omega_{J,K}))\;\geq \sum_{\{\a \,|\,\Labar=\Fmin,\, \Fi\in\Con1\}} b_{k-\mu_\a}(\F_\a)$$
	%
	%
	for all $k\geq 0$. In particular, $\rk (SH^{dim_{\C}\M}(\M))\geq  \# \{\text{minimal components}\}.$ %|\textrm{Min}(\M)|.$
\end{thm}
%Here the relation $\Labar=\Fmin$ means that the fixed set $\F_\a$ lies in the minimal component $\Fmin.$ 
Let us explain the indexing set $\{\F_i\subset L\,|\,  L \text{ minimal}\}$ better. 
It is a general feature of weight-1 conical actions that the fixed sets $\F_i$ are bijectively distributed among components of the core. 
%which follows from the \BB decomposition.%\cite[Prop. 4.6.1]{Gi15}. 
In the above sum, we use precisely those $\F_i$ that lie in \textit{minimal} components. Thus, when combined with the Betti numbers relation between $\F_i$ and $\M$ coming from Lemma \ref{IntroHomologyDecompositionOfTheCore} and deformation retraction $\M \simeq \L,$ we get the following corollary
%The term in the sum corresponding to the minimal component $\F_\varphi$ gives the simpler lower bound $\mathrm{rk}\,SH^k(\M,\omega_{J,K}))\geq b_{k}(\F_\varphi)$.

\begin{cor} When all core components of a CSR $\M$ are minimal, $H^*(\M)\fja{c^*} SH^*(\M)$ is an injection, thus
	symplectic cohomology is degree-wise bounded from below by the singular cohomology, $$\rk(SH^k(\M))\geq \rk(H^k(\M)), \ \forall k.$$
\end{cor}

\begin{ex}
	The last corollary applies in our standing example $X_{\Z/(n+1)}\fun \C^2/(\Z/(n+1)),$ as we have seen that all its core components are minimal.
	Thus, we get $$\text{rk}(SH^2(X_{\Z/(n+1)}))\geq \text{rk}(H^2(X_{\Z/(n+1)}))=n.$$
%	As the resolution $X_{\Z/(n+1)}$ deformation retracts to its core, $A_n$-chain of spheres, 
	It is known that $\text{rk}(SH^2(X_{\Z/(n+1)}))=n,$ due to \cite[Cor. 42]{EL17}, so in this case Proposition \ref{boundonSHIntro} gives the actual rank on the degree-2 part of symplectic cohomology.
\end{ex}

\subsection{Outline of the paper}

In Section \ref{SectionSemiprojective} we recall some features of semiprojective varieties, and prove some new ones, in particular the homology decomposition of the core 
(Lemma \ref{IntroHomologyDecompositionOfTheCore}). 
Along the way, we give a full proof that \BB pieces are smooth complex vector bundles.
In Section \ref{SectionSHS} we define SHS manifolds and prove their main features, in particular Lemma \ref{IntroCoreIsIsotropic}. We go through their examples, firstly mentioning Conical Symplectic Resolutions, where 
we provide some new statements. %and later on we mention Higgs moduli spaces and Hilbert schemes of points.
Section \ref{SectionSymplectic} describes symplectic structures on SHS manifolds.
In the first place we show that there is a $S^1$-invariant \KH structure with a moment map (Lemma \ref{IntrothereisaNonExactstructure}), coming from the variety structure, hence existing even on ordinary semiprojective varieties. 
Second, we show that there is a Liouville structure, coming from the $\C^*$-action and holomorphic symplectic structure (Proposition \ref{IntrocanonicalLiouville}).
Section \ref{SmoothCompCSRs} is the core, in which we prove the existence of smooth core components in weight-1 SHS manifolds (Theorem \ref{IntroMinimalComponentsTheorem}) 
and explain how one constructs them in practice, in the case of Conical Symplectic Resolutions. 
Along the way, we prove some useful statements regarding conical actions and their corresponding smooth core components.
In Section \ref{SymplTopMinCompns} we consider the symplectic topology of smooth core components, computing their Floer cohomologies as graded vector spaces (Theorem \ref{IntroLagrFloerMinComps}). Finally, in Section \ref{MinLagrasAndSH} we use the smooth core components obtained in Section \ref{SmoothCompCSRs} and the integral homology decomposition from Section \ref{SectionSemiprojective} in order to get lower bounds on symplectic cohomology of SHS manifolds (Theorem \ref{boundonSHIntro}).

\vspace{0.2cm}

\noindent \textbf{Acknowledgements.} This paper is a generalised and extended version of one part of my DPhil thesis, supervised by Alexander Ritter and Kevin McGerty, to both of whom I thank for many helpful conversations and suggestions. I am indebted to Dominic Joyce, Ivan Smith and Nick Sheridan for carefully reading the content of this paper and providing many useful suggestions.
I would also like to thank (by the order in the paper) to 
Ádám Gyenge for pointing out the paper \cite{CaGo83},
Nicholas Proudfoot for conversation that yielded Remark \ref{diffntdefnsOfCSR}, 
Dmitry Kaledin for the proof of Proposition \ref{ConnectedFibers}, 
Joel Kamnitzer for discussion regarding Example \ref{CSRaresmallinSHS}.
Paul Seidel and Nick Sheridan for conversations which coined parts of Proposition \ref{canonicalLiouville}, 
%Alexander Ritter and Kevin McGerty for proof of Lemma \ref{rootExtends}, and
and Nigel Hitchin, for pointing our the argument that proves Lemma \ref{SmoothIsExact}.
\section{On semiprojective varieties}\label{SectionSemiprojective}

Semiprojective varieties were introduced by Hausel--Rodriguez-Villegas \cite{HaR-V15}, whereas they were first considered by Simpson \cite[Thm. 11.2]{Si96}.
In this section we will prove some facts about these varieties which we will need in the later text. 
%Some of these we will just recall from \cite{HaR-V15}, but we will also prove some new properties which could be of general interest.

\begin{de}
	A \textbf{semiprojective variety} $(\M,\Fi)$ is a smooth quasi-projective variety $\M$ with an algebraic $\C^*$-action $\Fi$ satisfying two properties:
	\begin{enumerate}
		\item The set of fixed points $\M^{\Fi}$ is compact.
		\item Every point $x\in\M$ has a limit point $\displaystyle \lim_{\C^*\ni t\fun 0} \Fi_t(x),$ that is automatically a fixed point.
	\end{enumerate}
%	$$\forall x\in \M,\hspace{0.5cm} $$\lim_{t\fun 0}t\cdot x \in \M^*\C$$
Such a variety can have more such actions, which is why we specify $\Fi$ in the notation $(\M,\Fi).$\\\\
\noindent \textbf{ Notation.}
We will interchangeably use both notations $t\cdot x$ and $\Fi_t(x)$ for the $\C^*$-action, the former one in the case that it is clear which action we discuss.  
\end{de}
So, from the definition we see that although a semiprojective variety can be non-compact, it is ``tamed'' by the fact that all points converge to a compact set, when acting by an element $t\fun 0.$
Considering the points which converge when acting by $t\fun \infty,$ we come up with the following definition:

\begin{de}
	The \textbf{core} of a semiprojective variety $(\M,\Fi)$ is the set of points 
	$$\L_\Fi:=\{x \in \M \mid \lim_{\C^*\ni t\fun \infty} \Fi_t(x) \text{ exists} \}.$$
	For brevity, we will often use the notation $\L$ for the core, when the action $\Fi$ is fixed.
\end{de}

When a semiprojective variety $\M$ is compact, then it is equal to its core, so the notion seems redundant. 
However, when $\M$ is non-compact, the importance of the core arises as it is still compact, and completely controls the topology of $\M,$ being its deformation retract.

\begin{lm}\cite[Corollaries 1.2.2. and 1.3.6.]{HaR-V15} \label{CoreIsADefRetr}
	Given a semiprojective variety $(\M,\Fi),$ its core $\L$ is a proper variety and its deformation retract. 
\end{lm}

Here, recall the fact that an algebraic variety over $\C$ is \textbf{proper} if and only if it is compact in the analytic topology. Some authors use the term complete variety as well.

In the further sections, we will be interested in multiple $\C^*$-actions that act on the same semiprojective variety, so the natural question that arises is whether the corresponding cores are the same. We prove that it is indeed the case, at least when the actions commute.

\begin{lm}\label{CommutingActionsSameCore}
	Two commuting $\C^*$-actions on a semiprojective variety yield the same core.
\end{lm}
\begin{proof}
	Consider a semiprojective variety $\M$ with two different actions $\Fi^1$ and $\Fi^2$ that commute.	It is enough to prove $\L_{\Fi^1}\subset \L_{\Fi^2}.$
	Pick a point $x\in \L_{\Fi^1},$ let $\displaystyle x_1:=\lim_{t\fun \infty} \Fi^1_t(x_1)$ and let $\F_1$ be the connected component of $\M^{\Fi^1}$ which contains $x_1.$
	As the two $\C^*$-actions $\Fi^1$ and $\Fi^2$ commute and $\C^*$ is a connected algebraic group, they preserve each other's fixed loci connected components. 
	Hence, $\Fi^2(\F_1)=\F_1,$ and in particular, $\Fi^2_s(x_1)\in \F_1.$
	Now, notice that $\Fi^2_s(x)$ belongs to the core $\L_{\Fi^1},$ for arbitrary $s\in\C^*.$ Indeed,
	$$\lim_{t\fun \infty} \Fi^1_t \Fi^2_s (x) =  \Fi^2_s \big(\lim_{t\fun \infty} \Fi^1_t (x) \big) =\Fi^2_s (x_1)  \in \F_1.$$
	Now, as the core $\L_{\Fi^1}$ is a proper variety, the map $\C^*\fun \L_{\Fi^1},\  s \mapsto \Fi^2_s (x)$ extends to a regular map
	on $\C P^1,$ thus there is a limit $\displaystyle \lim_{s\fun \infty} \Fi^2_s (x),$ which precisely says that $x \in \L_{\Fi^2}.$
	Similarly $\L_{\Fi^2} \subset \L_{\Fi^1},$ giving the desired equality.
\end{proof}

Another way to control potential non-compactness of a semiprojective variety is by completing it. In the case that semiprojective variety is smooth, one can do it $\C^*$-equivariantly. 
%due to Sumihiro's completion theorem. %Any smooth semiprojective variety also has a smooth completion

\begin{lm}\label{CompletionLemma}
	Semiprojective variety $\M$ has a smooth equivariant projective completion, that is, a smooth projective variety $Y$ with a $\C^*$-action
	which has a Zariski-open $\C^*$-invariant subvariety isomorphic to $\M.$
\end{lm}
\begin{proof}
	Being smooth, quasi-projective variety $\M$ is normal, hence by \cite[Thm. 1]{Su74} there is a $\C^*$-equivariant embedding of $\M$ into $\C P^N$, hence the 
	completion by the closure $\ol{\M}$ of its image. To make the completion smooth, one uses the canonical resolution of singularities 
	$Y\fun \ol{\M}$ which has a functorial property \cite[Thm. 1.1]{fuctiorialResolution}, thus, it is equivariant with respect to smooth group actions. %(\cite[Thm. 1.0.3.(d)]{bierstone2011effective})
\end{proof}

This completion will be useful to our purposes for many reasons, first of which is the \BB decomposition. This decomposition was developed for proper varieties by \BB in his seminal paper \cite{BB73}. Here we prove it for semiprojective varieties, as a consequence of the proper case.

\subsection{\BB decomposition}
Given a semiprojective variety $(\M,\Fi)$, denote its fixed locus by $\F:=\M^{\Fi}.$ 
It is a smooth subvariety, being the fixed locus of an algebraic action of reductive group on a smooth variety, \cite[Lem. 5.11.1]{CGi97}.
Being proper, it breaks into %of the $\C^*$-action $\Fi,$ and, equivalently, of its $S^1$-part, is a smooth subvariety of $\M.$
finitely many connected components $\displaystyle \F=\sqcup_{i }\F_i.$ 
Focusing on the $\C^*$-convergence of the points in $\M,$ we define the sets of points that flow to each fixed component $\F_i:$
\begin{equation}\label{BBdecompositionEquations}
	\L_i:=\bigg\{x\in \M \bigm| \lim_{t\fun \infty} t\cdot x \in\F_i \bigg\}, \ \
	\D_i:=\bigg\{x\in \M \bigm| \lim_{t\fun 0} t\cdot x \in\F_i \bigg\}.
\end{equation}
One can view these as the algebraic version of the upward and downward Morse flow. %of the moment map $H,$ as the vector field of the $\R_+$-action is the gradient flow $\nabla H$ (Lemma \ref{GradientIsRPlus}). 
%Thus, these are smooth submanifolds in $\M.$ Moreover, on the algebraic side we have a version of the Bia\l{}ynicki-Birula decomposition for CSRs. 
Recall now the original result from \BB.

\begin{thm}%[Bia\l{}ynicki-Birula decomposition] 
	\cite[Sec. 4]{BB73} \label{BBDecompositionGeneral}
	Consider a smooth projective variety $Y$ with an algebraic $\C^*$-action on it. Denoting the decomposition of the fixed locus $\F=Y^{\C^*}=\bigsqcup_i \F_i,$ the upward and downward attracting sets
	$\L_i$ and $\D_i$ defined as in (\ref{BBdecompositionEquations}) are locally closed smooth subvarieties that decompose $Y,$ that is, $\bigsqcup_i \L_i= \bigsqcup_i \D_i=Y.$ Moreover, the natural morphisms 
	$$\L_i\fun \F_i, \ x\mapsto \lim_{t\fun \infty} t\cdot x,$$
	$$\D_i\fun \F_i, \ x\mapsto \lim_{t\fun 0} t\cdot x,$$ 
	are isomorphic to fibre bundles with affine fibres. %{\FZ Smoothly, they are isomorphic to complex vector bundles over $\F_i.$}
\end{thm}

\begin{cor}\label{BBDecompositionCSRs} Given a semiprojective variety $(\M,\Fi),$ its fixed locus $\M^{\Fi}=\F=\bigsqcup_{i}\F_i$ induces the decompositions 
	$\L=\bigsqcup_i \L_i, \M=\bigsqcup_i \D_i$ into smooth locally closed subvarieties, defined by (\ref{BBdecompositionEquations}). The natural morphisms
	$$\L_i\fun \F_i, \ x\mapsto \lim_{t\fun \infty} t\cdot x,$$
	$$\D_i\fun \F_i, \ x\mapsto \lim_{t\fun 0} t\cdot x,$$ 
	are isomorphic to fibre bundles with affine fibres. %{\FZ Smoothly, they are complex vector bundles over $\F_i.$ }
\end{cor}
\begin{proof}
	As every point $x\in\M$ has a limit $\lim\limits_{t\fun 0} t\cdot x$ and it must be a fixed point, the decomposition $\M=\sqcup_i \D_i$ is immediate. 
%	To prove the other decomposition, it is enough to prove that
%	\begin{equation}\label{characterisationOfTheCore}
%		\L=\{x\in \M \mid  \lim_{t\fun \infty} t\cdot x \text{ exist}\},
%	\end{equation}
%	which we do as in \cite[Lem. 4.5.2]{Gi15}. Namely, having a point $x$ such that $\lim\limits_{t\fun \infty} t\cdot x$  exist, the same will be true for its projection $y:=\pi(x),$ and then as the limit $\lim\limits_{t\fun 0} t\cdot y=0$ always exist, the map $\C^*\fun \L, \ t \mapsto t\cdot y$ extends to a regular map $f:\P^1\fun \M_0.$ As $\M_0$ is affine, $f$ must be a constant, hence $y=f(1)=f(0)=0,$ thus $x\in\pi^{-1}(0)=\L.$ Now, choose arbitrary $x\in \L.$ As $\L$ is projective, hence complete, the map $\C^*\fun \L, \ t \mapsto t\cdot x$ extends to a regular map $\P^1,$ thus the limit $\lim\limits_{t\fun \infty} t\cdot x$ exists.  
%	
	Now, by Lemma \ref{CompletionLemma} one can extend the $\C^*$-action on it to a smooth projective variety $Y$ which compactifies it, such that $\M\subset Y$ is a $\C^*$-invariant Zariski open subvariety. As $Y$ is smooth projective, it has the decompositions with the properties as in Theorem \ref{BBDecompositionGeneral}, thus it is enough to show that the decompositions $\L=\sqcup_i \L_i, \M=\sqcup_i \D_i$ are sub-decompositions of those in $Y$ (i.e. every piece of decomposition in $Y$ is either entirely contained in $\M$ or in $Y\setminus \M$). This is immediate: given any fixed locus component $\F_i$ in $\M,$ if there is a point $x\in Y$ such that 
	$\lim\limits_{t\fun 0} t \cdot x \in \F_i / \lim\limits_{t\fun \infty} t \cdot x \in \F_i,$ there is a neighbourhood $U$ of $0/\infty$ such that $\{t\cdot x \mid t\in U\}$ is contained in $\M$ (as $\M$ is an open neighbourhood of $\F_i$ in $Y$). Thus, as $\M\subset Y$ is $\C^*$-invariant, $x$ also belongs to $\M.$
\end{proof}

Given a $\C^*$-fixed component $\F_i,$ there is an induced $\C^*$-action on the tangent space of its arbitrary point $x \in \F_i,$ which yields a weight decomposition
% \C^*$-weights  Recall the weight decomposition of the tangent bundle of $\M$ restricted to $\F_i$ into subbundles,  
\begin{equation}\label{EqnWeightDecomp}
	T_{x} \M =\oplus_{k\in\Z} H_k(x),
\end{equation}
where the $H_k(x)$ consists of weight-$k$ vectors, i.e. vectors $v$ satisfying $t \cdot v = t^k v, t \in \C^*.$ In particular, $H_0(x)=T_x \F_i.$
In the following lemma we prove that this decomposition is uniform, forming a bundle decomposition of $T_{\F_i} \M.$
\begin{lm}\label{WeightDecompositionIntoBundles}
	The restriction of the tangent bundle of $\M$ to a $\C^*$-fixed connected component $\F_i$ decomposes into subbundles
	\begin{equation}\label{EqnWeightDecompBundles}
		T_{\F_i} \M =\oplus_{k\in\Z} H_k,
	\end{equation}
    where $H_k$ is the subbundle consisting of vectors on which $\C^*$ acts with weight $k.$  
\end{lm}
\begin{proof}
	For $x\in\F_i$ denote by $t_x$ the induced linearised action on $T_x \M$ by $t\in \C^*.$ 
	Now consider the characteristic polynomial $f(x)(X)=\det(X\cdot \mathrm{Id}- t_{x})$ (recall this is independent of choices of local coordinates since $t_x$ is an endomorphism of $T_x \M$).
	As the $\C^*$-action is algebraic and thus holomorphic, the coefficients of  $f(x)(X)$ are holomorphic functions with respect to $x\in \F_{i}$. Since $\F_i$ is a closed connected complex manifold, global holomorphic functions on $\F_i$ are constant.
	%\footnote{Using that non-constant holomorphic functions are open maps, and open compact subsets of $\C$ are points.} 
	Thus the coefficients of  $f(x)(X)$ are constant. Therefore, the eigenvalues of $t_x$ and their multiplicities are independent of $x\in\F_i$. Thus, the 
	weight subspace $H_k(x):=Ker(t_x-t^k Id \mid t\in \C^*)$ varies smoothly with respect to $x\in \F_i$ and is equidimensional, and thus forms a subbundle $H_k$ of $T_{\F_i}M.$
\end{proof}

\BB decomposition \cite[Thm 4.3(iii)]{BB73} also describes the tangent spaces of bundles $\L_i$ and $\D_i$ at $\F_i,$ which in the notation of the previous lemma are
\begin{equation}\label{BBdecompWeightSpaces}
	T_{\F_i}\L_i=\oplus_{k\leq 0} H_k, \qquad T_{\F_i}\D_i=\oplus_{k\geq 0} H_k. 
\end{equation}
Literature usually cites the \BB paper \cite{BB73} for the fact that the morphisms  $\L_i\fun \F_i$ and $\D_i\fun \F_i$ are diffeomorphic to complex\footnote{Meaning: $\R$-smooth vector bundles with complex vector space as fibres and transition functions in $GL(n,\C),$ where $n$ is the complex dimension of the fibre.} vector bundles on $\F_i,$ but this fact is actually never proven in that paper or elsewhere, to the author's knowledge. We fix that gap in the following proposition: %\ref{LemmaBBAreComplexVectorBundles} below. %, after recalling the weight decomposition.

\begin{prop}\label{LemmaBBAreComplexVectorBundles} 
	In the setup of Theorem \ref{BBDecompositionGeneral}, the morphisms $\L_i\fun \F_i$ and $\D_i\fun \F_i$ are diffeomorphic to $\C^*$-equivariant complex vector bundles over $\F_i.$
\end{prop}
\begin{proof}
	Let us prove the statement for the morphism $f:\D_i\fun \F_i,$ as the other one follows verbatim. Firstly, by \cite{BB73},\footnote{More precisely: Proofs of \cite[Thm. 2.5]{BB73} and \cite[Thm. 4.1]{BB73}.} we actually have a more precise information on the fibre bundle $f$ than stated in Theorem \ref{BBDecompositionGeneral}. Namely, its local trivializations are $\C^*$-equivariant, where the $\C^*$-action on the fibre $V$ is linear and isomorphic to the $\C^*$-action on the tangent space $\oplus_{k> 0} H_k.$  
	Thus, given trivialisations over two open sets $Y_i, Y_j \subset \F_i$ that intersect, the transition functions %$t_{ij}$  
	\begin{equation}\label{trfun}
		g_{ij}:  Y_i \cap Y_j \fun \text{Aut}_{\C^*}(V),
	\end{equation}
	land into the group $\text{Aut}_{\C^*}(V)$ of $\C^*$-equivariant algebraic\footnote{Recall that the setup of \cite{BB73} is algebraic.} automorphisms of the affine space $V.$ Thus, switching to the coordinate ring of $V$ and denoting $\dim V=n,$ those are graded isomorphisms of the polynomial ring $\C[x_1,\dots,x_n],$ where $deg(x_i)=a_i$ are the weights of the $\C^*$-action on $V.$
	In particular, when the weights on $V$ are all equal, we are left with linear automorphisms only, thus the transition functions of bundle $f$ indeed land in $GL(V)$. Hence, in that case $f$ has a structure of algebraic, thus holomorphic vector bundle, which is even stronger than what we need.
	%	(I think this is what BB writes in \cite[Remarks p.491]{BB73} that I have mentioned before). 
	Otherwise, one should expect non-linear transition functions in general, but the key point is that we can reduce the structure group $\text{Aut}_{\C^*}(V)$ to its linear subgroup $\text{GL}_{\C^*}(V),$ 
	by the following lemma:
	\begin{lm}
		Given a complex vector space $V$ with a linear $\C^*$-action with positive weights, the group of linear $\C^*$-invariant maps $\text{GL}_{\C^*}(V)$ is a deformation retract of the group 
		$\text{Aut}_{\C^*}(V)$ of $\C^*$-invariant polynomial automorphisms\footnote{Meaning: Polynomial maps whose inverses are also polynomial.} of $V.$
	\end{lm}
	\begin{proof}
		%	Notice that the notion of a polynomial map needs a choice of basis in $V,$ but then it is easy to see that the actual group $\text{Aut}_{\C^*}(V)$ of those does \textit{not} depend on this choice. %This is important to say as we do not have a global basis for V T_{\F_\a}\D_\a
		%Now, we can split the vector space in $V=\oplus_i V_i$ into weight spaces 
		%	Picking a basis of $V,$ a map in $\text{Aut}_{\C^*}(V)$ is of type 
		%	\begin{equation}\label{gradedpolyisos}
		%		(x_1,\dots,x_n)\fun (\sum_{\{j\mid a_j=a_i\}} b_{ij} x_j + \sum_{\{\a \mid \sum \a_j a_j = a_i, \sum_j \a_j >1 \}} c_\a^i x_1^{\a_1}\dots x_n^{\a_n})_{i=1,\dots,n}
		%	\end{equation}
		%	for some constants $b_{ij}, c_\a \in \C,$ and its inverse is of the same type. Notice that the maps in $\text{GL}_{\C^*}(V)$ are precisely those given by
		%	the first sum in \eqref{gradedpolyisos}, i.e. maps of type
		%		$$(x_1,\dots,x_n)\fun (\sum_{\{j\mid a_j=a_i\}} b_{ij} x_j)_{i=1\dots n}$$
		%	for some $b_{ij}\in\C.$ Though, not all values of $b_{i,j}$ are allowed. Indeed, 
		
		Split $V=\oplus_{i=1}^r V_i$ into weight spaces $\{V_i\}_{i=1\dots r}$ such that their weights increase with $i=1\dots r.$ It is clear that each map in $\text{Aut}_{\C^*}(V)$ or $\text{GL}_{\C^*}(V)$ has to preserve $V_i.$ Thus, denoting the coordinates in $V_i$ by $x_i^1,\dots, x_i^{s_i}$ and the vectors $x_i=(x_i^1,\dots,x_i^{s_i}),$ we have that an arbitrary map in $\text{Aut}_{\C^*}(V)$ is of type
		\begin{equation}\label{gradedpolyisos} \small
			(x_1,\dots,x_r) \mapsto (L_1(x_1), L_2(x_2)+ p_2(x_1),L_3(x_2) + p_3(x_1,x_2),\dots, L_r(x_r) + p_r(x_1,\dots,x_{r-1}))
		\end{equation}
		where $p_i$ are arbitrary polynomials and $L_i$ are invertible linear maps. Indeed, passing to the map of coordinate ring $\C[V]=\C[(x_i^j)_{i,j}]$ we see inductively by $i$ that $L_i$ need to be invertible in order to get all monomials $\{x_i^j\}_{j=1,\dots, s_i}$ in the image. Furthermore, the map \eqref{gradedpolyisos} is going to be invertible polynomial for arbitrary choices of $p_i,$ as one can show inductively as well (the inverse is $(x_1,\dots,x_r)\mapsto (L_1^{-1}(x_1),L_2^{-1}(x_2-p_2(L_1^{-1}(x_1))),\dots)$).
		Hence, the group $\text{Aut}_{\C^*}(V)$ deformation retracts (by letting all coefficients of $p_2,\dots,p_r$ in \eqref{gradedpolyisos} to go to zero) to its subgroup given by linear maps
		\begin{equation*}
			(x_1,\dots,x_r) \mapsto (L_1(x_1), L_2(x_2),L_3(x_2),\dots, L_r(x_r)),
		\end{equation*}
		which is precisely the group $\text{GL}_{\C^*}(V).$ Thus, the lemma is proved.
	\end{proof}
	Let us recall the following standard lemma from the theory of fibre bundles:
	\begin{lm}
		Given a $F$-fibre bundle $E$ whose transition functions land in Lie group $G\leq \text{Diff}(F)$ and a closed subgroup $H\leq G$ which is a deformation retract of $G,$ there is a $F$-fibre bundle $E'$ with transition functions landing in $H,$ which is isomorphic to $E$ as a $F$-fibre bundle.
	\end{lm}
	\begin{proof}
		We can associate a principal $G$-bundle $P$ to $E,$ constructed using the transition functions of $E.$ Then, the associated $F$-fibre bundle $P\times_G F$ is isomorphic to $E$ as a $F$-fibre bundle. Now, as $H\leq G$ is a deformation retract, $G/H$ is contractible, hence by \cite[Cor 2.4, Ch. VI]{Hu94} there is a $H$-reduction of $P,$ that is, a principal $H$-bundle $Q$ such that $P\iso Q\times_H G.$ Thus, by \cite[Thm. 3.1, Ch. VI]{Hu94} 
		we have isomorphism of $F$-fibre bundles $P\times_G F \iso Q\times_H F.$ By definition, transition functions of the associated bundle $E':= Q\times_H F$ are the same as the transition functions of $Q,$ hence they land in $H.$ As $E'\iso E,$ the lemma is proved. 
	\end{proof}
	Now, the proposition follows. Indeed, putting $G=\text{Aut}_{\C^*}(V)$ and $H=\text{GL}_{\C^*}(V)$ in the previous lemma, we get that $f:\D_i\fun \F_i$ is isomorphic to a $V$-fibre bundle with transition functions in $\text{GL}_{\C^*}(V),$ hence diffeomorphic to a $\C^*$-invariant complex vector bundle.
\end{proof}

\subsection{Homology decomposition of the core}\label{SectionHomologyDecompositionOfCore}
In this section, we will prove that the homology of semiprojective varieties and their core %satisfy the integral decomposition of homology.%In short, this means that their homology 
decomposes to the homology of their fixed loci, with certain shifts. This type of a result was proved by Carrell--Sommese \cite{CaSo79} for compact \KH manifolds, and later by Carrell--Goresky \cite{CaGo83} for singular complex varieties satisfying certain properties. We show that the core of a semiprojective variety satisfies these properties, inferring the integral decomposition of its homology.\\

\noindent \textbf{ Notation.} By $H_*(X)$ we will denote the singular homology of a topological space $X,$ with integer coefficients. However, all statements of this section work over arbitrary
coefficients.\\

Let us first recall the results from  \cite{CaGo83} that we need. Firstly, we will introduce the notion of a good decomposition.
\begin{de}
	Let $X$ be a proper algebraic variety. A decomposition $X=\bigsqcup_i \L_i$ into locally closed subvarieties is called \textbf{good} if the subvarieties $\L_i$ satisfy the following:
	\begin{itemize}
		\item[(1a)] For each $i$ there is a holomorphic map $p_i:\L_i\fun \F_i$ onto a proper complex variety $\F_i$ making $\L_i$ a fibre bundle with affine fibres. %$\C^{\mu_i}$
		\item[(1b)] Each $p_i$ extends meromorphically to $\ol{\L_i}$ in the sense that 
		$$\Gamma_i:= \ol{\{(x,p_i(x))\mid x\in \L_i\}} \subset X \times \F_i$$ is a closed subvariety of $X\times \F_i,$ containing the graph of $p_i$ as a Zariski open set.
		\item[(1c)] If $\F_i$ is singular, then it admits an analytic Whitney stratification such that if $A$ is a connected component of a stratum, then $g_i^{-1}(A)$ is irreducible, where
		$g_i:\Gamma_i \fun \F_i$ is the projection on the second factor
		\item[(1d)] There is a filtration of $X$ by closed subvarieties
		$$ \emptyset = Z_0 \subset \dots \subset Z_r = X$$
		such that (for some renumbering of $X_i$'s) $Z_i \setminus Z_{i-1} = \L_i,$ for $1 \leq i \leq r.$
	\end{itemize}
\end{de}

The technical condition (1c) allows one to define graded homomorphisms\footnote{Here we are using the standard shifting notation from homological algebra; Given a graded module $A_*,$ the graded module $A[k]_*$ is defined by $A[k]_i:=A_{i+k}.$}
\begin{equation}\label{ThomMaps}
	\eta_i:  H_*(\F_i)[-\mu_i] \fun H_*(X), \ [C]\mapsto [\ol{p_i^{-1}(C)}],
\end{equation}
for a generic cycle $C.$ Here the closure $\ol{p_i^{-1}(C)}$ is taken in $X,$ and $\mu_i$ is the real dimension of the affine fibre in $p_i.$
 
%	Also, notice that the result from \cite{CaGo83} holds for integral homology. Nevertheless, their proof goes over the $\Z/2$ coefficients (which we are using here) as well.} 
Then, Carrell--Goresky prove that, for any good decomposition of a proper variety, the map $\oplus \eta_i$ gives an isomorphism, hence a decomposition of homology of $X.$
In particular, they consider a case of a good decomposition that comes from a $\C^*$-action.
Namely, consider a smooth projective variety $Y$ with a $\C^*$-action having a $\C^*$-invariant subvariety $X.$ Denoting by $\F_1,\dots,\F_r$ the connected components of the $\C^*$-fixed locus in $X,$ define 
the decomposition $$X=\bigsqcup_i \L_i, \ \ \L_i:=\{x\in X \mid \lim_{t \fun \infty} t\cdot x \in X_j\}.$$ We say that the $\C^*$-action on $X$ is \textbf{good} if the associated decomposition $X=\bigsqcup_i \L_i$ is good. Here, the $p_i$ maps are given by $p_i:\L_i\fun \F_i, \ x\mapsto \lim_{t\fun \infty} t\cdot x.$
The integral decomposition theorem for this setup follows:

\begin{thm}\cite[Thm. 1', Sec. 2]{CaGo83}\label{HomologyDecompositionForSingularSubvariety}
	Let $Y$ be a smooth projective variety with a holomorphic $\C^*$-action.
	Suppose $X\subset Y$ is a closed $\C^*$-invariant subvariety, and let $X^{\C^*}=\bigsqcup_i \F_i$ be the decomposition of its fixed locus into connected components.
	Suppose that the action on $X$ is good. Then, there is a graded isomorphism
	$$\Phi= \oplus_i \eta_i: \bigoplus_i H_*(\F_i)[-\mu_i] \fun H_*(X),$$
	where maps $\eta_i$ are given by (\ref{ThomMaps}).
\end{thm}

We prove a lemma that gives a criterion for the action to be good on a subvariety $X.$

\begin{lm}\label{CarellGoreskyBBversion} 
%	In the setup of Theorem \ref{HomologyDecompositionForSingularSubvariety}, without assumption that the action on $X$ is good
	Let $Y$ be a smooth projective variety with a holomorphic $\C^*$-action.
	Suppose $X\subset Y$ is a closed $\C^*$-invariant subvariety, and let $X^{\C^*}=\bigsqcup_i \F_i$ be the decomposition of its fixed locus into connected components.
%	Let $Y$ be a smooth projective variety with a holomorphic $\C^*$-action.
%	Suppose $X\subset Y$ is a closed $\C^*$-invariant subvariety, and let $X^{\C^*}=\bigsqcup_i \F_i$ be the decomposition of its fixed locus into connected components.
	Assume further that:   
	\begin{enumerate}[(i)]
		\item \label{prvahipoteza} The fixed components $\F_i$ %are also connected components of the fixed locus in $Y.$ In particular, they 
		are smooth.
		\item \label{drugahipoteza} Their $(t\fun \infty)$-attracting sets %$\L_i:=\{x \in X \mid \lim\limits_{t\fun \infty} t\cdot x \in \F_i\}$ 
		in $X$ are the same as their attracting sets in $Y.$
	\end{enumerate}
	Then, the $\C^*$-action on $X$ is good. 
%	\begin{enumerate}[(1)]
%		\item \label{kleimprvi} The morphisms $p_i: \L_i \fun \F_i, \ x\mapsto \lim_{t\fun \infty} t\cdot x$ are fibre bundles with affine fibres.
%		\item \label{kleimdrugi} There is an isomorphism\footnote{Here we are using the standard shifting notation from homological algebra; Given a graded module $A_*,$ the graded module $A_*[k]$ is obtained from it by shifting \textbf{down} by $k.$ Also, notice that the result from \cite{CaGo83} holds for integral homology. Nevertheless, their proof goes over the $\Z/2$ coefficients (which we are using here) as well.} 
%		$$\Phi= \oplus_i \eta_i: \bigoplus_i H_*(\F_i)[-\mu_i] \fun H_*(X).$$
%		where $\eta_i:  H_*(\F_i)[-\mu_i] \fun H_*(X), \ [C]\mapsto [\ol{p_i^{-1}(C)}],$ for a generic cycle $C.$
%		Here the closure $\ol{p_i^{-1}(C)}$ is taken in $X,$ and $\mu_i$ is the real dimension of the bundle $p_i.$
%	\end{enumerate}
%	In particular, the \C^*-decomposition of a core is good.
\end{lm}
\begin{proof}
	Consider a fixed component $\F_i$ in $X.$ It belongs to a fixed component $\F_i'$ in $Y.$ As $X\subset Y$ is a closed subvariety, $\F_i \subset \F_i'$ is so as well. 
	Due to assumption (\ref{drugahipoteza}) we have that the map $\displaystyle p_i:\L_i\fun \F_i$ is a restriction of the map $\displaystyle p_i':\L_i'\fun \F_i', \ x\mapsto \lim_{t\fun \infty} t\cdot x$ where 
	$\L_i':=\{x \in Y \mid \lim\limits_{t\fun \infty} t\cdot x \in \F_i'\}$ is the attraction set of $\F_i'$ in $Y.$ Now, from Theorem \ref{BBDecompositionGeneral} applied to $Y,$ we know that the $p_i'$ is a fibre bundle with affine fibres, so the same holds to its restriction $p_i,$ which proves the property (1a). 
%	Now, we claim that one can apply the theorem \cite[Thm. 1', Sec. 2]{CaGo83} by Carrell-Goresky that claims the isomorphism (\ref{kleimdrugi}) in this setup. Indeed, as stated in that theorem, one needs the $\C^*$-action  on $X$ has to be \textit{good}, which means that the decomposition $X=\cup \L_i$ needs to satisfy certain conditions (1a)-(1d), \cite[p.368-369]{CaGo83}. 
	As mentioned in \cite[Rmk, Sec. 2]{CaGo83}, the conditions (1b) and (1d) are always satisfied in this setup. 
	Finally, the condition (1c) only makes sense if there are singular $\F_i,$ thus is vacuous due to our assumption (\ref{prvahipoteza}). 
	%Finally, the condition (1a) is exactly the claim (\ref{kleimprvi}), hence the the $\C^*$-action on $X$ is indeed {good}, 
%	and the isomorphism (\ref{kleimdrugi}) holds. For its description given in the claim (\ref{kleimdrugi}) we refer to \cite[p.369]{CaGo83}.
\end{proof}
%Notice that the maps $p_i$ are indeed complex vector bundles by the general theory of Bia\l{}ynicki-Birula decomposition \cite{BB73}, as in Theorem \ref{BBDecomposition}.

Using the previous lemma, we finally reach the main point of this section - the homology decomposition of the core of a semiprojective variety.
\begin{prop}\label{HomologyDecompositionOfTheCore}
	Given a smooth semiprojective variety $(\M,\Fi)$ and its fixed locus $\M^{\Fi}=\bigsqcup_i \F_i,$ there is a decomposition of the homology of its core,
	$$\Phi= \oplus_i \eta_i: \bigoplus_i H_*(\F_i)[-\mu_i] \fun H_*(\L).$$ Here $\mu_i$ are real dimensions of the fibres of attracting sets $\L_i\fun\F_i,$ and
	$\eta_i$ are given by \eqref{ThomMaps}.
	In particular, this gives us the relation between the Betti numbers of $\M$ and the fixed loci:
	\begin{equation}\label{BettiNumbers} 
		b_k(\M)=b_k(\L) = \sum_{i} b_{k-\mu_i}(\F_i).
	\end{equation}
\end{prop}
\begin{proof}
	By Lemma \ref{CompletionLemma} there is a smooth $\C^*$-equivariant completion $Y$ of $\M,$ and thus the core $\L_\Fi$ is a closed $\C^*$-invariant subvariety of smooth projective $Y.$
	The fixed set $\L^{\Fi}=\M^{\Fi}=\sqcup_i \F_i$ is smooth and their ($t\fun \infty$)-attracting sets in the core are the same as ones in $\M$ (by definition), and the latter are the same as in $Y,$ as explained in the proof of Corollary \ref{BBDecompositionCSRs}. Thus we have that the core satisfies the conditions of Lemma \ref{CarellGoreskyBBversion}, hence the $\C^*$-action on it is good, and we have the homology decomposition of Theorem \ref{HomologyDecompositionForSingularSubvariety}.
\end{proof}

\begin{rmk}\label{CounterexampleByJoyce}
	We remark that the assumptions in Lemma \ref{CarellGoreskyBBversion} are indeed needed, or at least assumption (\ref{drugahipoteza}), as the following example shows. Namely, consider a $\C^*$-action on $Y=\P^2$ given by
	$$t\cdot [z_0,z_1,z_2]=[z_0,tz_1,t^2z_2].$$ Its fixed locus are points $\F_0=[1:0:0], \ \F_1=[0:1:0], \ \F_2=[0:0:1].$ Choose $X$ to be the union of three complex lines which these points form.
	We see that assumption (\ref{prvahipoteza}) from the theorem holds, but notice that assumption (\ref{drugahipoteza}) does not.
	Namely, the attracting set $\L_2'=\{x \in Y \mid \lim\limits_{t\fun \infty} t\cdot x \in \F_2\}$ attached to the point $\F_2$ is equal to $[z_0, z_2, 1]\iso \C^2$ whereas the attracting set $\L_2$ in $X$ attached to the same point is equal to the union of two lines, $[z_0:0:1],$ and $[0:z_1:1],$ thus indeed $L_2\neq L_2'$. 
	It is immediate that $\L_2\fun \F_2$ is not an affine bundle, and also that the homology isomorphism between fixed loci $\F_i$ and $X$ cannot hold, as $H_*(X)=\k[0]\oplus \k[-1] \oplus \k^3[-2],$ 
	thus has a bigger rank than the total rank of all $H_*(\F_i).$
\end{rmk}

\begin{rmk} We remark that the consequence of previous corollary, i.e. the correspondence between Betti numbers of the fixed loci and of the core of a semiprojective variety (\ref{BettiNumbers}) 
	was also proved in \cite[Theorem 1.3.7]{HaR-V15}, by different means.
\end{rmk}
\section{Semiprojective Holomorphic Symplectic manifolds}\label{SectionSHS}

In this section we will restrict our attention to smooth semiprojective varieties that also have holomorphic symplectic structure, which has positive weight with respect to the $\C^*$-action.
Those are the main objects of interest in this paper. We will prove some further properties that these objects satisfy, in addition to the general ones that hold for semiprojective varieties mentioned in the previous section. In the end of the section we will list examples of these objects that have appeared in the literature so far. We start with the definition:

\begin{de}\label{DefinitionSHS}
	A \textbf{Semiprojective Holomorphic Symplectic (SHS)} manifold $(\M,\Fi,I,\om_\C)$ is a smooth connected semiprojective variety $(\M,I,\Fi)$\footnote{Here the complex structure $I$ comes from the 
    structure of a smooth algebraic variety on $\M.$} with an $I$-holomorphic symplectic form $\om_\C$
	that has a positive weight with respect to the $\C^*$-action $\Fi,$ i.e. there is an integer $s>0$ such that $$t \cdot \om_\C = t^s \om_\C, \forall t \in \C^*.$$
	Such $\C^*$-actions we will call \textbf{conical}. The integer $s$ we call the \textbf{weight} of an SHS $(\M,\Fi).$
	Therefore, by a \textbf{weight-1 SHS} we will call an SHS manifold that admits a weight-1 conical action. 
	We will often refer to an SHS as any of its substructures (e.g. $(\M,\Fi),$ $(\M,I),$ $(\M,\Fi,I)$ or $(\M,\om_\C)$) for brevity, 
	depending on which part of its structure we want to emphasize.
\end{de}

\begin{rmk}
	We remark that similar class of semiprojective varieties were considered in \cite{HaR-V15}, called \textit{hyper-compact} varieties.
	However, their notion is more restrictive, asking for \HK structure and a weight-1 action, whereas we want to keep the generality given by the definition above, 
	hence we have chosen to give it another name.
\end{rmk}
Having the symplectic structure tuned in, the core of an SHS manifold has an additional important feature given in the next lemma. The sketch of the proof goes back to Nakajima \cite[Thm. 5.8]{Nak94a}, whereas we write it here in full details for completeness. Recall that a subvariety in a smooth symplectic variety $(M,\om)$ is called \textbf{$\om$-isotropic} if the symplectic form vanishes on its smooth locus, 
and \textbf{Lagrangian} if in addition it is half-dimensional.
\begin{lm}\label{CoreIsIsotropic}
	The core of an SHS manifold $(\M,\Fi,\om_\C)$ is $\om_\C$-isotropic. Furthermore, the core of a weight-1 SHS is a Lagrangian subvariety.
\end{lm}
\begin{proof}
Denoting the fixed locus $\M^{\Fi}=\F=\bigsqcup_{i}\F_i,$ recall that by Corollary \ref{BBDecompositionCSRs} the core decomposes into $(t\fun\infty)$-attracting pieces $\L=\bigsqcup \L_i.$
Also, recall that by Proposition \ref{LemmaBBAreComplexVectorBundles} each $$p_i: \L_i\fun \F_i, \ x\mapsto \lim_{t \fun \infty} t \cdot x,$$ is a $\C^*$-equivariant vector bundle, with the fibre being $\C^*$-equivariantly isomorphic to the vector space $\oplus_{k<0} H_k = T_{\F_i} \L_i.$
Now, let us prove that the arbitrary piece $\L_i$ is $\om_\C$-isotropic. Pick a non-fixed point $y\in\L_i,$ and two vectors $u_1,u_2 \in T_y \L_i.$
Then, 
\begin{equation}\label{ActionOnSymplForm}
	\om_\C(u_1,u_2)=t^{-s} \om_\C(t\cdot u_1, t\cdot u_2), 
\end{equation}
where $s$ is the %$\om_\C$-
weight of $(\M,\Fi).$
Now, letting $t\fun \infty,$ we get $\om_\C(u_1,u_2)=0,$ as long as the limit of vectors $\lim_{t\fun \infty} t \cdot u_l,\ l\in\{1,2\}$ exist in the first place. Denote the limit point $y_\infty:=\lim_{t \fun \infty} t \cdot y.$
By Proposition \ref{LemmaBBAreComplexVectorBundles}, there is a $\C^*$-equivariant bundle isomorphism 
$$ \Phi: p_i^{-1}(U) \fun \big(\bigoplus_{k<0} H_k\big) \times U,$$ for some neighbourhood $U$ of $y_{\infty},$ thus we can pass the limit problem to the trivial bundle. So, having an arbitrary vector $\xi=(\eta,\sigma)$ at the point $(v,y_{\infty}),$ we see that there is a limit 
$\lim_{t\fun \infty} t \cdot \xi=\lim_{t\fun \infty} t \cdot (\eta,\sigma)=\lim_{t\fun \infty} (t\cdot \eta,\sigma)=(0,\sigma),$ at point $(0,y_{\infty}).$ Here $\lim_{t \fun \infty} t\cdot \eta=0$ as $\eta\in \oplus_{k<0} H_k.$
Pulling back via $\Phi,$ we get that limits $\lim_{t\fun \infty} t \cdot u_l, \ l\in\{1,2\}$ exist and moreover, are tangent to $\F_i.$ %so belong to the 0-weight space.
Now, consider a fixed point $y\in\F_i,$ and two vectors $u_1,u_2\in T_{\F_i} \L_i = \oplus_{k<0} H_k.$ 
The equation \eqref{ActionOnSymplForm} implies that $\om_\C$ %on $T_y \M = \oplus_{k\in\Z} H_k$ 
evaluated at two homogeneous vectors $w_1\in H_{k_1}, w_2\in H_{k_2}$ gives zero if $k_1+k_2 \neq s.$ 
In particular, $\om_\C|_{\oplus_{k<0} H_k }= 0,$ which proves that $\L$ is $\om_\C$-isotropic.
Moreover, as $\om_\C$ is symplectic, it induces a non-degenerate pairing 
\begin{equation}\label{PairingByOmC}
	\om_\C: H_k \otimes H_{s-k}\fun \C, 
\end{equation}
thus $H_k \iso H_{s-k}^*$ and $\dim(H_k)=\dim(H_{s-k}).$ In particular, when $s=1$ we get that the vector space $\oplus_{k<0} H_k$ is half-dimensional in $T_y \M.$ As $T_y \L_i =\oplus_{k<0} H_k,$ the proposition follows.
\end{proof}

We have an important corollary of this Lemma, which shows how symplectic geometry can constrain (algebraic) geometry of the core.

\begin{cor}\label{CoreDecompositonWeight1}
Given a weight-1 SHS manifold $(\M,\Fi),$ its core $\L=\bigsqcup_i \L_i$ has pure dimension ${1 \over 2} \dim \M$ and its irreducible components are precisely the closures $\overline{\L_i}.$
\end{cor}

\begin{rmk}
In particular, this gives an elegant proof of pure-dimensionality of Springer fibres of type A, classically proved by Spaltenstein in \cite{Spa76}.
These varieties are fibres of the Springer resolution, one of the central objects in Geometric Representation Theory. 
They can be seen as cores of Resolutions of Slodowy varieties, which are examples of weight-1 Conical Symplectic Resolutions, hence weight-1 SHS manifolds as well, 
thus the previous corollary indeed shows the pure-dimensionality of Springer fibres.
\end{rmk}

\begin{ex}\label{DuValResolutionsExampleSHS} 
	In the lowest dimension $(dim_\C =2)$ we have a nice family of examples of SHS manifolds called resolutions of Du Val singularities.\footnote{In the literature they are also called simple surface singularities, Kleinian singularities, ADE singularities or rational double points.}
	These are the minimal resolutions of quotient singularities
	$$\pi_{\Gamma}:X_\Gamma\fun \C^2/\Gamma$$
	for all finite subgroups $\Gamma\leq SL(2,\C).$ Any such group $\Gamma$ (up to conjugation) can be labelled bijectively via an ADE Dynkin graph $Q_\Gamma$ which is called its \textit{McKay graph}.\footnote{Which comes from representation theory of the group $\Gamma.$}
	The Dynkin graph $A_n$ corresponds to the cyclic group $\Z/(n+1),$ the graph $D_n$ corresponds to the binary dihedral group $BD_{4(n-2)}$, whereas the graphs $E_6,E_7$ and $E_8$ correspond to the binary tetrahedral, octahedral and icosahedral groups, respectively.
	Du Val proved in \cite{DuVal34} that the central fibre $\pi_{\Gamma}^{-1}(0)$ consists of a union of 2-spheres, whose dual graph of intersections is exactly the
	McKay graph $Q_\Gamma.$
	We will also include the case of a trivial group $\Gamma=\{e\}$ which corresponds to the resolution of $\C^2/{e}= \C^2$ hence $\C^2$ itself, and we will consider it as a $A_0$ surface singularity.
	
	Varieties $X_\Gamma$ are quasi-projective, being obtained via finite set of blow-ups of affine varieties $\C^2/\Gamma.$ 
	The natural holomorphic-symplectic structure on $X_\Gamma$ comes as the pull-back 
	%$\om_\C:={\pi_\Gamma}^*\om$ 
	of the quotient of the standard symplectic structure on $\C^2$ defined on $(\C^2-\{0\})/\Gamma.$ Thus, $\om_\C$ is defined away from the central fibre but it can be proved that it extends symplectically on it.\footnote{E.g. by work of Brieskorn and Slodowy that makes the isomorphism of $\pi_{\Gamma}$ with the restriction of Springer resolution to the intersection of the Slodowy slice with the nilpotent cone, \cite{SloBook}.}
	%using the work of Kronheimer \cite{Kro89}, who proved that the varieties $X_\Gamma$ are actually \HKL.} 
	 
	The standard conical weight-2 action on these spaces comes from the dilation action 
	\begin{equation}\label{ActionOnC2}
		\C^*\dejstvo \C^2, \ t\cdot (z_1,z_2)=(t z_1,t z_2),
	\end{equation}
	making $X_\Gamma$ indeed a weight-2 SHS.
	
	The core of this action is exactly the central fibre $\L_\Gamma:=\pi_{\Gamma}^{-1}(0),$ and its irreducible components are 2-spheres forming an ADE Dynkin tree.
	$\L_\Gamma$ is half-dimensional, hence Lagrangian, and indeed one can show that in types D and E the standard action is actually a square of a weight-1 conical action. In type A, this depends on the parity of $n,$ but nevertheless there is a weight-1 action that commutes with the standard one, hence they yield the same core (Lemma \ref{CommutingActionsSameCore}). We will construct all such weight-1 commuting actions in Example \ref{DuVal_Weight1actions_TypeA}.
\end{ex}

\begin{ex}\label{CotangentBundlesAsSHS} Another simple family of SHS manifolds are cotangent bundles $T^*X,$ where $X$ is any projective variety. 
	As a vector bundle over projective variety is always quasi-projective,\footnote{Given such a bundle $E\fun X,$ we can embed it as an open subset $E \fja{\iso} \P(E \oplus 1) \subset \P(E \oplus \mathcal{O}_X)=\bar{E}$ of its projective completion $\bar{E},$ which is a $\P^n$-bundle over projective $X,$ hence a projective variety itself.}
	so is $T^*X$ in particular. It has a canonical holomorphic symplectic structure, and a natural conical weight-1 $\C^*$-action that 
	contracts the fibres with weight one. Thus, it is a weight-1 SHS manifold. Its Lagrangian core is $X$, the zero section. 
	In particular, in the lowest dimension we obtain SHS manifolds $T^*\Sigma_g,$ where $\Sigma_g$ is a smooth projective curve of an arbitrary degree $g \geq 0.$
\end{ex}

Having in mind the two last families of examples, which overlap in the example $T^*\C P^1$, we propose a question of independent interest.
\begin{que}\label{ClassificationOfSHSInDim2}
	Are there any more SHS manifolds of complex dimension 2, beyond Du Val resolutions and cotangent bundles $\{T^*\Sigma_g\}_{g\geq 1}$?
\end{que}

Considering the fact that its core is a Liouville skeleton (Proposition \ref{canonicalLiouville}(\ref{CoreIsSkeleton})), given any graph $Q$ and the union of arbitrary closed real surfaces $\Sigma_g$ intersecting according to that graph, there is always a Weinstein manifold whose skeleton is the corresponding union of surfaces $\L_Q.$ Those spaces were studied in \cite{EL19}. However, the question is, whether that Weinstein manifold can be enhanced to a holomorphic symplectic structure.

Now we will mention two important corollaries of the existence of holomorphic symplectic structure on an SHS. 

Firstly, given an SHS $\M$ of complex dimension $2n,$ we have that the top exterior power $\omega_{\C}^{\wedge n}$ of its holomorphic symplectic form $\om_\C$ trivialises the canonical bundle $\mathcal{K}:=\Lambda^{2n,0}T^*\M.$ Using the general fact that $c_1(\mathcal{K})=c_1(T^*\M)=-c_1(T\M),$ this implies:

\begin{lm}\label{LemmaCalabiYau}
	Any SHS $(\M,I)$ satisfies $c_1(T\M,I)=0.$ %where $I$ is its complex structure.
\end{lm}

Moreover, we notice that the holomorphic symplectic structure on $\M$ can be upgraded to an almost \HK structure, which we define now:
\begin{de}\label{DefAlmostHKstructure}
	Given a manifold $M$ an \textbf{almost \HK structure} on it is given by the quadruple $(g,I,J,K),$ where $g$ is a Riemannian metric and $I,J,K$ are $g$-orthogonal almost complex structures satisfying $IJ=K.$ This yields non-degenerate 2-forms $\om_I,\om_J,\om_K,$ defined by $\om_S(\cdot,\cdot):=-g(\cdot,S\cdot),$ for $S=I,J,K.$ An almost \HK structure is called \textbf{\HK} if we have $\nabla^g I= \nabla^g J= \nabla^g K= 0,$ where $\nabla^g$ is the Levi-Civita connection of $g.$ In particular, this implies that $I,J,K$ are complex structures and $\om_I,\om_J,\om_K$ are \KH forms.
\end{de}

\begin{lm}\label{LemmaCSRsAreAlmostHK}
	Any SHS manifold $(\M,I,\om_\C)$ can be enriched with an almost \HK structure $(g,I,J,K),$ such that $\om_\C=\om_J+i\om_K.$ 
	In particular, we have that the almost complex structures $S_t:=cos (t) J + sin (t) K$ satisfy $c_1(T\M,S_t)=0.$
\end{lm}
\begin{proof}
	Briefly, a holomorphic symplectic structure on a complex manifold of dimension $2n$ corresponds to a (torsion free) $Sp(2n,\C)$-structure on the tangent bundle of $\M,$ and an almost \HK structure corresponds to an $Sp(n)$-structure, and the key thing is that $Sp(2n,\C)$ deformation retracts to $Sp(n)=Sp(2n,\C)\cap U(2n),$ hence an SHS can indeed be enriched with an almost \HK structure. 
	Now we will explain this in more details, for the convenience of the reader.
	
	Firstly, having a holomorphic symplectic structure $\om_\C$ on $\M,$ we have local trivialisations of $\pi: T\M \fun \M,$ given by Darboux bases, i.e. vector fields 
	$(e_1,\dots,e_n,f_1,\dots,f_n)$ such that $\om_\C(e_i,e_j)=\om_\C(f_i,f_j)=0, $ $ \om_\C(e_i,f_j)=\delta_{ij}.$ Given two such trivialisations 
	$$\Fi_i: \pi^{-1}(U_i) \fun U_i \times \C^{2n}, \ \ \Fi_j: \pi^{-1}(U_j) \fun U_j \times \C^{2n},$$
	on their overlap we have 
	\begin{align*}
		\Fi_j \circ \Fi_i^{-1}: (U_i \cap U_j) \times \C^{2n} &\fun (U_i \cap U_j) \times \C^{2n}, \\ 
		(b,v)&\mapsto (b,g_{ij}(b)v),
	\end{align*}
	where %the transition functions 
	$g_{ij}(b)\in Sp(2n,\C).$ %(as a map sending a Darboux basis to a Darboux basis is symplectic).
	Thus, one gets an associated $Sp(2n,\C)$-principal bundle $P\fun \M$ using the transition functions $g_{ij}.$ As $Sp(2n,\C)$ deformation retracts to its compact form $Sp(n),$ the quotient 
	$Sp(2n,\C)/Sp(n)$ is contractible, hence by \cite[Cor 2.4, Ch. VI]{Hu94} there is an $Sp(n)$-reduction of $P,$ that is, a principal $Sp(n)$-bundle $Q$ such that $P\iso Q\times_{Sp(n)} Sp(2n,\C).$
	Then, by \cite[Thm. 4.1, Ch. VI]{Hu94},\footnote{Notice the slight difference between stated therein and here, as our transition function $g_{ij}$ is their $g_{ij}^{-1}$ and we take $r_i$ to be their $r_i^{-1}.$} that means that there are maps $r_i:U_i\fun Sp(2n,\C)$ such that 
	$r_j(b)g_{ij}(b)r_i(b)^{-1}\in Sp(n),$ for every $i,j.$ Next, denote 
	$$\widetilde{r}_i: U_i \times \C^{2n} \fun U_i \times \C^{2n}, \ (b,v) \mapsto (b, r_i(b)v)$$ and
	define new local trivialisations $\widetilde{\Fi}_i:=\widetilde{r}_i \circ \Fi_i$ %\pi^{-1}(U_i) \fun U_i \times \C^{2n}$ 
	of $T\M.$ Then, for all $i,j$ we have 
	\begin{align*}
		\widetilde{\Fi}_j \circ \widetilde{\Fi}_i^{-1}: (U_i \cap U_j) \times \C^{2n} &\fun (U_i \cap U_j) \times \C^{2n}, \\ 
		(b,v)&\mapsto (b,r_j(b)g_{ij}(b) r_i(b)^{-1}v),
	\end{align*}
	thus, as $r_j(b)g_{ij}(b)r_i(b)^{-1}\in Sp(n),$ we get a $Sp(n)$-structure on $T\M.$ Now we will show that it yields an almost \HK structure $(g,I,J,K)$ on $\M$ whose holomorphic symplectic part is the given one, $\om_\C=\om_J+i\om_K.$ 
	
	Denote by $(g^0,I^0,J^0,K^0,\om_\C^0)$ the standard corresponding structures on $\C^{2n}\iso \H^n.$ First, notice that $\Fi_i^*\om_\C^0=\om_\C,$ and $\Fi_i^*I^0=I$ (where pull-back is understood fibre-wise, $\Fi_i: T_b\M \fun \{b\}\times \C^{2n}$), as we have used complex Darboux basis for defining $\Fi_i.$ As $\widetilde{r}_i^*\om_\C^0=\om_\C^0, \ \widetilde{r}_i^*I^0=I,$ we have 
	$\widetilde{\Fi}_i^*\om_\C^0=\om_\C, \ \widetilde{\Fi}_i^*I^0=I$ as well.
	
	Next, we define structures on $T\M$ using the pull-back from local trivialisations $\widetilde{\Fi}_i:$  
	$$g:=\widetilde{\Fi}_i^*g^0,   J:=\widetilde{\Fi}_i^*J^0, K:=\widetilde{\Fi}_i^*K^0.$$
	They are well-defined (i.e. independent of $i$) as the transition functions $\widetilde{\Fi}_j \circ \widetilde{\Fi}_i^{-1}$ fibre-wise lie in $Sp(n),$ in particular, preserve $g^0,J^0,K^0.$
	Moreover, together with $I$ and $\om_\C,$ they form an almost \HK structure on $\M,$ being a pull-back of the standard almost \HK structure on $\C^{2n}.$ This proves the first part of the lemma.
	
	Now, notice that given a vector $u=(u_I,u_J,u_K)\in S^2,$ a linear combination 
	$$S_u=u_I I + u_J J + u_K K$$
	is an almost complex structure. %\footnote{Indeed, to prove it, one needs the relations $IJ=-JI, \ JK=-KJ, \ KI=-IK$ which follow from $I^2=J^2=K^2=-\text{Id}$ and $K=IJ.$} 
	Thus, we get a smooth $S^2$-family of almost complex structures $S_u,$ hence corresponding complex vector bundles $(T\M,S_u)$ are all isomorphic,\footnote{As $S^2$ is path-connected.} and together with Lemma \ref{LemmaCalabiYau} we get $c_1(T\M,S_u)=0.$ In particular, we have  $c_1(T\M,S_t)=0$ for all $S_t=cos (t) J + sin(t)  K.$
	%	 complex 2-form 
	%	$$\om_\C^J:=\om_K+i \om_I.$$
	%	One can directly check, using the relations from Definition \ref{DefAlmostHKstructure}, that it is a $J$-holomorphic form.\footnote{I.e., it satisfies $\om_\C^J(J\cdot,\cdot)=i \om_\C^J(\cdot,\cdot),$
	%	$\om_\C^J(\cdot,J\cdot)=i \om_\C^J(\cdot,\cdot).$}
	%	Also, it is non-degenerate 2-form
	%	Thus, its top exterior power $(\omega_{\C}^J)^n$ is a non-vanishing J-holomorphic volume form that trivialises the canonical bundle, hence as in Lemma \ref{LemmaCalabiYau} we have
	%	$c_1(T\M,J)=0.$ Finally, we have the same conclusion for any $S_t$ due to the fact $(T\M,S_t)$ are all isomorphic as complex vector bundles, as $S_t$ is continuous path of almost complex structures.
\end{proof}

The previous lemma motivates the following definition:

\begin{de}
	Given a SHS manifold $(\M,I,\om_\C),$ an almost \HK structure $(g,I,J,K)$ on it such that $\om_\C= \om_J+i \om_K$ %the given holomorphic symplectic structure (in the sense of Lemma \ref{LemmaCSRsAreAlmostHK}) 
	we will call \textbf{compatible}.
\end{de}

\begin{de}\label{DefinitionHKCSR} Let $(\M,\Fi, I,\om_\C)$ be an SHS. We call it a \textbf{semiprojective \HK manifold} and abbreviate by SHK if there is a compatible \HK structure $(g,I,J,K)$ 
	such that the $S^1$-part of $\Fi$ acts by $g$-isometries. %hence preserves $\om_I.$
	On SHK manifolds we only consider $\C^*$-actions satisfying this condition.
\end{de}
In fact, most known SHS are known to be SHK manifolds (all examples of Conical Symplectic Resolutions and Moduli spaces of Higgs bundles). 
%Thus, assuming the SHS condition in Proposition \ref{SteinStructure} and in the graded part of Theorem \ref{LagrFloerMinComps} still keeps them applicable for most of our examples.

\subsection{Conical Symplectic Resolutions} \label{CSRs}

Most important and by far the broadest family of known SHS manifolds are Conical Symplectic Resolutions (CSRs). They were defined in \cite{BPW16,BLPW16}, though these objects were already being considered by various authors, most notably Kaledin \cite{Ka06,Ka08,Ka09} and Namikawa \cite{Nam08,Nam11}.

These manifolds are of interest to theoretical physics \cite{HaSp18,GrHa20,HaKa20}, algebraic geometry \cite{Ka06,Ka09,Nam08,BeSch16} and representation theory \cite{Nak94a,Nak98,Nak01,MO12,Nak15,BFN16}. Also, some of them have very interesting differential-geometric features \cite{Kro89,KroNak90,Nak94a,BD00,Nak18}. Examples of CSRs include many well-known families of spaces such as resolutions of Du Val singularities (Example \ref{DuValResolutionsExampleSHS}), Hilbert schemes of points on them, hyperpolygon spaces, Nakajima quiver varieties, hypertoric varieties, cotangent bundles of flag varieties and nilpotent Slodowy varieties. All known examples of CSRs have a complete \HK metric and in fact are SHK manifolds, %(Definition \ref{DefinitionHKCSR}) 
which makes them a particularly nice geometric setting to work in.
%\footnote{For compact \KH manifolds,  by Yau's theorem (together with a Bochner's formula and Berger's holonomy classification) the presence of a holomorphic symplectic form  $\omega_{\C}$ guarantees the existence of a \KH form $\omega$ such that $\omega,\mathrm{Re}(\om_\C), \mathrm{Im}(\om_\C)$ is a hyperkähler structure. However, it is not known whether this also holds in our non-compact setting.} 

\begin{de}\label{DefCSRwithFi}
	A \textbf{Conical Symplectic Resolution (CSR)} is a projective resolution\footnote{Meaning: $\M$ is a smooth variety, and $\pi$ is an isomorphism over the smooth locus of of $\M_0$.} $\pi:\M \fun \M_0$ of a normal %singular 
	affine variety $\M_0$, where $(\M,\om_\C)$ is a holomorphic symplectic manifold and $\pi$ is equivariant with respect to $\C^*$-actions on $\M$ and $\M_0$ (both denoted by $\Fi$). These actions satisfy two conditions:
	\begin{enumerate}
		\item[(1)] the symplectic form $\om_\C$ has a weight $k\in \N$, so $t \cdot \om_\C = t^k \om_\C,$ and
		\item[(2)]\label{contracts} the action $\Fi$ contracts $\M_0$ to a single fixed point $x_0$, so $\displaystyle \lim_{t\fun 0} t \cdot x=x_0$ for all $x\in \M_0.$
	\end{enumerate}
	%Such $\C^*$-actions are called \textbf{conical} of \textbf{weight} $k$.  
	In general, a CSR can have many such actions, of possibly different weights $k,$ so we often denote a CSR by $(\M,\Fi)$ to emphasise the choice of $\Fi$.
\end{de}

\begin{lm}\label{CSRisSemiproj} A CSR $\M$ is an SHS manifold, and its core is the central fibre $\L=\pi^{-1}(x_0)$ of the resolution $\pi:\M\fun \M_0.$
\end{lm}
\begin{proof}
	Notice first that a CSR $\M$ is semiprojective. Indeed, by the condition (2) we have that the fixed locus of $\Fi$ lies in the central fibre $\pi^{-1}(x_0),$ which is projective (as $\pi$ is projective), hence proper, thus so is the fixed locus as its closed subset. 
	Also, the condition (2) along with the fact that $\pi$ is equivariant and proper tells us that every point $x\in\M$ has a limit $ \lim_{t\fun 0} t \cdot x.$
	Indeed, it amounts to say that the holomorphic map from a punctured disc $\accentset{\circ}{D^2}\fun \M, t \mapsto t\cdot x,$ can be extended to $D^2.$ But it is true as it is bounded, having the image in $\pi^{-1}(D^2 \cdot \pi(x))$ which is compact, hence bounded. Here, $D^2 \cdot \pi(x)$ is well-defined as every point in $\M_0$ has a limit for the action when $t\fun 0.$ (condition (2))
	Being projective over an affine variety $\M_0$, the variety $\M$ is quasi-projective. Indeed, $\pi:\M \fun \M_0$ being a projective morphism means that there is an embedding $\M\hookrightarrow \M_0 \times \C P^n$ for some $n$ (\cite[p.103]{Ha77}). As $\M_0$ is an affine variety there is an embedding $\M_0 \hookrightarrow \C^m$ for some $m,$ hence we have the sequence of embeddings
	$$\M_0 \times \C P^n \hookrightarrow \C^m \times \C P^n \hookrightarrow 
	\C P^m \times \C P^n \hookrightarrow \C P^{(m+1)(n+1)-1},$$ the last one being the Segre embedding.
	Altogether we get a holomorphic embedding $\M \stackrel{\iota}{\hookrightarrow} \C P^N,$ where $N:=((m+1)(n+1)-1),$ so $\M$ is quasi-projective.
	Together with previous, we get that a CSR $\M$ is indeed semiprojective, thus an SHS manifold from the definition. 
	
	The fact that the core is equal to the central fibre $\L=\pi^{-1}(x_0)$ follows from \cite[Lem. 4.5.2]{Gi15}.
	Let us assume that $x\in \pi^{-1}(x_0).$ Then, as $\pi^{-1}(x_0)$ is proper, the map $\C^*\fun \L, \ t \mapsto t\cdot x$ extends to a regular map on $\P^1,$ thus the limit $\lim\limits_{t\fun \infty} t\cdot x$ exists.
	Now, having a point $x$ such that $\lim\limits_{t\fun \infty} t\cdot x$  exist, the same will be true for its projection $y:=\pi(x),$ and then as the 
	limit $\lim\limits_{t\fun 0} t\cdot y=0$ always exist, the map $\C^*\fun \L, \ t \mapsto t\cdot y$ extends to a regular map $f:\P^1\fun \M_0.$ As $\M_0$ is affine, $f$ must be a constant, hence $y=f(1)=f(0)=0,$ thus $x\in\pi^{-1}(0)=\L.$  
\end{proof}

We remark that our definition of CSR is slightly different from the original one from \cite[Sec. 2]{BPW16}, 
%as we require $\M_0$ to be singular, in order to exclude the trivial example $\C^{2n}\fun \C^{2n}.$\footnote{And indeed this condition only excludes $\C^{2n}.$ Namely, having a smooth affine variety $\M_0$ with a $\C^*$-action that contracts it to a single point, by the Bia\l{}ynicki-Birula decomposition (Corollary \ref{BBDecompositionCSRs}) applied to it, we get that it must be an affine space indeed.}  
%	 which is non-interesting for the material written in this thesis. Namely, its core is just a point, so there are no Lagrangians in it (the content of Chapter \ref{quivvar}). Also, its $S^1$-invariant \KH structure is exact and Liouville indeed, so the content of Chapter \ref{NonExactSH} is already established for this example (see \cite[Sec. 3]{OanceaEnsaios}).
%Notice further a difference from the original definition of CSR, 
in that we do not require $\M_0$ to be the affinisation $\Aff(\M):=Spec(\Gamma(\M,\O_\M))$ of $\M,$ although we do require $\M_0$ to be a normal variety. However, we will show that these two definitions are equivalent in the following lemma, which the author proved together with Nicholas Proudfoot.
\begin{lm}\label{diffntdefnsOfCSR}
	In Definition \ref{DefCSRwithFi} of CSR, requiring $\M_0$ to be normal is equivalent to requiring for $\M_0$ to be the affinisation of $\M.$
\end{lm}
\begin{proof} 
	Assume that $\M_0$ is the affinisation of $\M.$ Then, the coordinate rings $\C[\M]:=\Gamma(\M,\O_\M)$ and $\C[\M_0]$ are isomorphic. Being smooth, $\M$ is normal, thus its local rings $\O_{\M,p}$ are integrally closed. As an intersection of integrally closed rings, the coordinate ring $\C[\M]=\cap_{p\in \M} \O_{\M,p}$ is also integrally closed. 
	It follows that the coordinate ring of $\M_0$ is integrally closed.  Since $\M_0$ is affine, this implies that it is normal.
	
	Assume now that $\M_0$ is normal. As $\M_0$ is affine, we have a factorization of $\pi$ through the affinisation map, $\M \fun \Aff(\M) \fun \M_0,$ thus it is enough to show that the second map is an isomorphism. One of the properties of the affinisation map is that proper connected subvarieties get shrunk to points.	
	In particular, as the fibres of $\pi: \M \fun \M_0$ are connected (Proposition \ref{ConnectedFibers}) and projective (since $\pi$ is projective),
	they shrink to points via the map $\M \fun \Aff(\M)$, which yields that the map $g: \Aff(\M) \fun \M_0$ is a bijection, in particular, has finite fibres.
	Moreover, it is birational as $\M \fun \M_0$ is. Thus, one can use a variant of Zariski's main theorem which states that a birational morphism with finite fibres to a normal variety is an isomorphism onto an open subset. Hence, $g$ is isomorphism to an open subset of $\M_0,$ thus being surjective, it is indeed an isomorphism.
\end{proof}
It is also true that the fibres of a CSR $\pi$ are connected, which the author learned from Dmitry Kaledin. Here one essentially uses the normality of $\M_0$.

\begin{prop}%{ \normalfont \cite{Ka19}} 
		\label{ConnectedFibers}
	Given a CSR $\pi:\M \fun \M_0$, the fibres of $\pi$ are connected. 
\end{prop}
\begin{proof}
	Take a Stein factorisation of $\pi$, that is: there is a variety $\M_0'$ and morphisms $\pi':\M\fun\M_0'$ and $g:\M_0'\fun \M_0$ such that $\pi=g\circ \pi'$, the fibres of $\pi'$ are connected, and $g$ is a finite morphism. As $\pi$ is birational, so is $g,$ but then (a variant of) Zariski's main theorem states that $g$, being a birational morphism with finite fibres to a normal variety, is an isomorphism onto an open subset of $\M_0.$ As $\pi$ is surjective, the same holds for $g.$ Hence, $g$ is an isomorphism, thus all fibres of $\pi$ are connected.
\end{proof}

In the end of this section we want to remark on size of the family of CSRs inside the family of SHS manifolds. 
Let us consider the following example.
\begin{ex}\label{CSRaresmallinSHS}
We have seen in Example \ref{CotangentBundlesAsSHS} that $T^*X$ is an SHS manifold, given \textit{any} smooth projective variety $X.$ 
These are CSRs \textit{only} in the case when $X$ is a flag manifold $X=G/P$ (here $G$ is a semisimple algebraic group and $P$ is a reductive subgroup).
Indeed, given a conical symplectic resolution $$\M:=T^*X \fun \Aff(\M)$$ whose $\C^*$-action contracts the fibres of $T^*X$ with weight 1, the ring of global functions $R:=\Gamma(\M,\O_\M))=\C[\Aff(\M)]$ has a set of homogenous generators $\{x_1,\dots,x_n\}$ with maximal weight 1. 
The last fact, together with the normality of the ring $R,$ is precisely saying that $\Aff(\M)$ is a conical symplectic singularity with maximal weight 1, thus the Namikawa's theorem \cite{Nam18} applies, stating that $\Aff(\M)=\ol{\O}$ must be a closure of a normal nilpotent orbit $\O$ in a semisimple $\mathfrak{g}.$ Then, by \cite[Thm 0.1]{Fuu03}, we indeed have 
$\M \iso T^*(G/P),$ for some parabolic subgroup $P$ of $G$ (and $G$ given by $Lie(G)=\mathfrak{g}$). 
\end{ex}

Considering the previous (topologically simplest) examples of cotangent bundles, one could expect that, just as the set of flag manifolds being very small in the class of projective varieties, it may be that the class of CSRs are very small in the class of SHS manifolds, and there are many more families of these spaces yet to be discovered, beyond the CSRs and two families described in the following Sections \ref{HiggsModuli} and \ref{HilbSchemes}. This is the main motivation for generalising the setup of CSRs (for which the content of this paper was originally written) to the setup of SHS manifolds.

\subsection{Higgs branches of gauge theories}\label{HiggsBranch}
The principal family of examples of SHS manifolds  
comes from the holomorphic symplectic reduction of vector spaces, 
in the physics-related literature usually referred as the \textbf{Higgs branch} of the 3-dimensional supersymmetric gauge theory defined by a 
pair $(G,N),$ where $G$ is a complex reductive algebraic group and $N$ is its complex representation. 
%which carry $N=4$ supersymmetry.
Given the linear action $G\dejstvo N,$
we can enhance it to a Hamiltonian action $G\dejstvo T^*N,$ where we take the usual holomorphic symplectic structure on
$M:=T^*N=N\oplus N^*.$ Thus, given the associated moment map $\mu: M \fun \mathfrak{g}^*,$ 
one can construct the ordinary GIT quotient\footnote{Where GIT is the standard abbreviation for Geometric Invariant Theory, \cite{GITbook}.}
$$\M_0:=\mu^{-1}(0) \sslash G=\text{Spec}(\C[\mu^{-1}(0)]^G).$$ 
More generally, given an arbitrary character $\chi:G\fun \C^*,$ one can consider its twisted GIT quotient
\begin{equation}\label{GITtwistedHiggsBranch}
	\M_\chi=\mu^{-1}(0)\sslash_{\chi} G := \text{Proj}\big(\bigoplus_{n\geq 0} \C[\mu^{-1}(0)]^{G,\chi^n} \big),
\end{equation}
where $\C[\mu^{-1}(0)]^{G,\chi^n}=\{f\in \C[\mu^{-1}(0)] \mid f(g\cdot m)=\chi^n(g)f(m), \ \forall m \in M\}.$
In particular, for $\chi=1$ we get the ordinary GIT quotient $\M_0.$ 
By the construction, there is an associated morphism 
\begin{equation}\label{GITmorphsimHiggsBranch}
	\pi:\M_{\chi} \fun \M_0.
\end{equation}
which is projective due to the general GIT theory.
Moreover, both $\M_\chi$ and $\M_0$ inherit a natural algebraic, hence holomorphic Poisson structure from $M.$
%Moreover, restricted to $\M_\chi^s$ is smooth (which is true for a generic choice of $\chi$) 
To discuss smoothness of $\M_\chi,$ we have to introduce the notion of (semi)stable points, existing in any GIT quotient setup.
%The space $\M_{\chi},$ defined algebraically by \eqref{GITtwistedHiggsBranch}, can be seen topologically as a set of closed $\chi$-semistable orbits in $\mu^{-1}(0).$ 
The point $m \in M$ is called $\chi$-semistable if for any $z \in \C, \ z\neq 0,$ the closure of the orbit $(m,z)$ is disjoint from the $M \times {0},$ where
the action on $M\times \C$ is defined as a product $g\cdot (m,z)= (g \cdot m, \chi^{-1}(g) z).$ The set of semistable points we denote by $M^{ss}$
A point $m \in M^{ss}$ is stable if its stabiliser $G_m$ is finite and its orbit is closed in $M^{ss}.$ The locus of stable points we denote by $M^{s}.$ 
The important fact is that $\M_{\chi}$ can be seen topologically as a set of closed $\chi$-semistable orbits in $\mu^{-1}(0).$ 
So, if $G$ acts freely on the locus of $\chi$-semistable points in $\mu^{-1}(0),$ we have that $\M_{\chi}=\mu^{-1}(0)^s/G$ %=\mu^{-1}(0)^s/G$ 
is non-singular and its Poisson structure is non-degenerate, hence is holomorphic symplectic.\footnote{We will not delve in details and proofs of these standard statements, rather we refer the reader to a nice exposition of GIT under Hamiltonian reduction, written in \cite[Ch. 9]{Kir16}.}
We will assume that this is the case in the text below. This condition happens in many cases; when $N$ is a space of framed quiver representations (example \eqref{NakaQuivVar} below), 
it holds for the generic choice of $\chi.$

By construction, $\M_\chi$ is a quasi-projective variety, so in order to be an SHS manifold it is left to find a conical action on it.
There is a standard weight-2 conical action on $T^*N$  
\begin{equation}\label{FullActionHiggsBranch}
	\C^*\dejstvo T^*N=N \oplus N^*, \ 	t \cdot (x,\xi):=(t x, t \xi).
\end{equation}
%dilation on the whole $T^*N=N\oplus N^*,$ making $\M_\chi$ into a weight-2 SHS manifold indeed. 
Furthermore, there is also a weight-1 action on $T^*N$
\begin{equation}\label{HalfActionHiggsBranch}
	\C^*\dejstvo T^*N=N \oplus N^*, \ t \cdot (x,\xi):=(x, t \xi),
\end{equation}
made by dilating the cotangent fibres, which is not conical, and thus not always conical on the quotient $\M_\chi.$ 
More precisely, we have the following:

\begin{lm}\label{WhenHiggsBranchWeight1Action} Given a pair $(N,G)$ and a character $\chi,$ its Higgs branch $\M_\chi$ together with the action induced by \eqref{FullActionHiggsBranch} is a weight-2 SHS manifold, 
provided that it is non-singular. Moreover, it is a weight-1 SHS only if $N^G = {0}.$
\end{lm}
\begin{proof}
The fact that $\M_\chi$ is holomorphic symplectic and quasi-projective is already explained above. Let us prove that the action induced by \eqref{FullActionHiggsBranch} is indeed conical on $\M_\chi.$
Firstly, notice that this action commutes with the action by $G,$ hence is well defined on both $\M_\chi$ and $\M_0.$ Moreover, it makes the morphism $\pi$ equivariant. Now, notice that $\M_0^{\C^*}=[(0,0)].$ Indeed, given an arbitrary point
$[x,\xi]\in \M_0,$ we have $\lim_{t\fun 0} t\cdot [x,\xi]= \lim_{t\fun 0} [t x, t \xi]= [0,0],$ hence every point in $\M_0$ flows to $[0,0]$ when $(t\fun 0),$ thus it is indeed the only fixed point.
Hence, the semiprojectivity of $\M_\chi$ follows the same way as for a CSR (Lemma \ref{CSRisSemiproj}).
Finally, as the action \eqref{FullActionHiggsBranch} is weight-2 on $T^*N,$ it is on the quotient $\M_\chi$ as well, hence it is a weight-2 SHS manifold.

Assume now that $N^G \neq 0,$ and let us prove that $\M_\chi$ indeed does not have a weight-1 conical action.
Denoting by $F:=N^G,$ as $G$ is reductive there is a $G$-invariant split $N=F\oplus N^{norm},$ given by the weight-decomposition of the representation $N$ (notice that $F$ is the 0-weight space). 
This induces a  symplectic $G$-invariant split 
\begin{equation}\label{DecompositonHiggsBranch}
M= T^*F \oplus T^*N^{norm},
\end{equation} 
and as the moment map of the product of two actions is the sum of the moment maps, %the formula for the moment map of a  of two actions
we have $\mu^{-1}(0)=T^*F \times (\mu^{norm})^{-1}(0),$ where $\mu^{norm}: M^{norm}:=T^* N^{norm} \fun \g^*$ is the moment map of the action $G \dejstvo T^* N^{norm}.$
Moreover, as one can directly show from the definition of (semi)stability, 
the set of $\chi$-(semi)stable points splits as well, hence there is a $G$-invariant split
%\begin{equation}\label{splitHiggsBranch2}
$	\mu^{-1}(0)^{s}=T^*F \times (\mu^{norm})^{-1}(0)^{s},$
%\end{equation}
%of an induced action on the cotangent bundle,\footnote{Given an action $G\dejstvo X,$ the moment map $\mu :T^*X\fun \mathfrak{g}^*$ is given by 
%	$\langle \mu(x,\lambda),a \rangle = \langle \lambda, \xi_a(x)\rangle, \ \forall a \in \g,$ where $\xi_a$ is the fundamental vector field on $X$ associated to the element $a \in \g.$} 
%one verifies that $T^*F \subset \mu^{-1}(0).$ Moreover, as $G$ acts trivially on $T^*F,$ and, as one can show, it consists of $\chi$-semistable points, it passes to the
%quotient, 
thus quotienting out by $G$ we get 
\begin{equation}\label{splitHiggsBranch2}
\M_\chi = T^*F \times \M_\chi^{norm}.
\end{equation}
As the $\C^*$-action \eqref{FullActionHiggsBranch} commutes with $G,$ the split \eqref{DecompositonHiggsBranch} is $\C^*$-invariant, as both summands are sums of weight spaces of $G\dejstvo M.$ 
Thus, the induced split \eqref{splitHiggsBranch2} is $\C^*$-invariant as well. 
Notice that $T^*F\iso \C^{2n}$ is a weight-2 SHS with the dilation action, whereas $\M_\chi^{norm}$ is a weight-2 SHS, being the Higgs branch for the pair $(G,N^{norm}).$ 
Thus, we have the induced split of cores
$$\L_\chi = \{0\} \times \L_\chi^{norm}.$$
%Hence, the core of $\_\chi$ is a product of cores of $T^*F$ and $\M_\chi^{norm},$ thus
%t is diffeomorphic to the core of $\M_\chi^{norm}.$ 
Hence, the core $\L_\chi$ cannot be half-dimensional in $\M_\chi,$ thus by Corollary \ref{CoreDecompositonWeight1}, $\M_\chi$ cannot be a weight-1 SHS.
%notice further that the weight-1 action need not having the same core, but the homology argument together with the deformation retraction to the core saves the day
\end{proof}

\begin{rmk}
	We remark that in many cases a stronger statement, that $\pi$ is a weight-2 CSR, is true as well. 
	However, this reduces to a smaller class of examples, and for our purposes we only need the space
	$\M_\chi$ to be an SHS, hence do not impose this restriction.
\end{rmk}

Two important examples of this construction are:
\begin{enumerate}
	\item \label{NakaQuivVar} Nakajima Quiver Varieties, where $N=N(Q,V,W)$ is constructed as a moduli space of framed quiver representations of 
	a quiver $Q$ of the given dimension $V$ and framing $W.$
	\item \label{HTvar} (Additive) Hypertoric Varieties, where $G\iso (\C^*)^k$ is an algebraic torus acting linearly on a vector space $N.$ 
\end{enumerate} 
For the former family, the necessary condition $N^G=0$ for the existence of a weight-1 conical action on the space $\M_\chi$ 
translates to the condition that there are no loop-edges in the chosen quiver $Q.$
For such quivers, Nakajima proved in \cite[Sec. 5]{Nak94a} that the weight-1 action coming from dilation of cotangent fibres \eqref{HalfActionHiggsBranch} is indeed conical, 
so in the case of quiver varieties this condition is sufficient as well.

\subsection{Moduli spaces of Higgs bundles}\label{HiggsModuli}
Another important family of SHS manifolds, that is parallel but disjoint from the CSRs are the Moduli spaces of Higgs bundles, constructed by Hitchin \cite{Hi87}, and since then considered by many authors. Hitchin considered the differential-geometric moduli space constructed via infinite-dimensional \HK reduction, whereas Nitsure \cite{Nitsure91} constructed their algebraic-geometric version $\mathcal{M}'(r,d,\Sigma),$ the moduli scheme of stable pairs $(E,\phi)$ where $E$ is a vector bundle of rank $r$ and degree $d$ on a projective curve $\Sigma$ and $\phi$ a morphism 
$\phi: E \fun E \oplus K_\Sigma$ (so-called Higgs field), where $K_\Sigma$ is the canonical bundle of $\Sigma.$\footnote{In fact, Nitsure considered the generalisation of this, by switching $K_\Sigma$ with an arbitrary line bundle $L$ on $\Sigma,$ but for brevity we are not considering those moduli spaces.} 
Nitsure proves that $\mathcal{M}'(r,d,\Sigma)$ is a smooth quasi-projective variety whenever $(n,d)=1.$ There is a natural $\C^*$-action on $\mathcal{M}'(r,d,\Sigma),$ given by scaling the Higgs field 
$t\cdot (E,\phi)= (E,t\phi).$ 
By Hitchin's construction, the moduli space $\mathcal{M}'(r,d,\Sigma)$ is \HKL, hence holomorphic symplectic and the $\C^*$-action acts with weight-1 on its form. To see that they are also semiprojective, consider the so-called Hitchin fibration 
$$p:\mathcal{M}'(r,d,\Sigma)\fun \oplus_{i=1}^r H^0(\Sigma,K_\Sigma^i)$$
given by coefficients of the characteristic polynomial of $\phi.$ 
It is $\C^*$-equivariant, with a linear $\C^*$-action on the affine space $\oplus_{i=1}^r H^0(\Sigma,K_\Sigma^i)$ that has a single fixed point.
Together with the fact that $p$ is proper\footnote{Due to Hitchin for $n=2$ and Nitsure in general.}, the semiprojectivity follows the same way as for CSRs (Lemma \ref{CSRisSemiproj}).

\subsection{Hilbert schemes of points on $T^*\Sigma$}\label{HilbSchemes}
The last known family of examples (to the author's knowledge) of SHS manifolds are the Hilbert schemes of points on cotangent bundles of smooth projective curves, $\Hilb^n(T^*\Sigma_g), g\geq 0.$ 
These generalise the lowest-dimensional examples from Example \ref{CotangentBundlesAsSHS}. Apart from the $g=0$ examples (which are quiver varieties of type $\tilde{A_1},$ \cite{Kuz07}),
this family is also disjoint from the CSRs. These manifolds were considered by Nakajima \cite{Nak99}. 
According to \cite{Gro60}, $\Hilb^n(X)$ of
a quasi-projective $X$ is again quasi-projective. Hence, as $T^*\Sigma_g$ is quasi-projective (Example \ref{CotangentBundlesAsSHS}), so is $\Hilb^n(T^*\Sigma_g).$ 
It has a natural symplectic structure and the conical weight-1 $\C^*$-action, both coming from the ones on $T^*\Sigma_g,$ see \cite[Sec 7.2]{Nak99} for details.
The fixed locus of the $\C^*$-action is
\begin{equation}\label{FixedPointsHilbTSigma}
	\Hilb^n(T^*\Sigma_g)^{\C^*}=\bigsqcup_{\{\nu : \text{ partition of } n\}}\ \Sym^\nu \Sigma_g,
\end{equation}
where $\Sym^{\nu} \Sigma_g := \Sym^{\a_1}\Sigma_g \times \Sym^{\a_1}\Sigma_g \times \cdots \times\ \Sym^{\a_n}\Sigma_g $ for the partition $\nu$
%$\nu=(1^{\a_1} 2^{\a_2} \dots n^{\a_n}).$ 
%$\nu$ given as a sum $n= \sum_i \a_i i$ 
in which the number $i$ appears $\a_i$ times. %=(1^{\a_1} 2^{\a_2} \dots n^{\a_n}).$
In particular, the fixed locus is compact, and $\Hilb^n(T^*\Sigma_g)$ are indeed weight-1 SHS manifolds.

At the end of this section, let us go back to Question \ref{ClassificationOfSHSInDim2} on classification of SHS manifolds amongst complex surfaces. As we have seen, $\Hilb^n(T^*\Sigma)$ is an SHS manifold because $T^*\Sigma$ is so. In the same way, $\Hilb^n(X_\Gamma)$ is an SHS manifold, where $X_\Gamma$ is the minimal resolution of a Du Val singularity. 
%It belongs to the examples from Section \ref{CSRs}, as $\Hilb^n(X_\Gamma)$ are known to be quiver varieties of affine ADE type, \cite{Kuz07}. 
So, if the answer to Question \ref{ClassificationOfSHSInDim2} is affirmative and there is another SHS surface $X$ that does not belong to these two classes, then $\Hilb^n(X)$ would be a new class of examples of SHS manifolds.
\section{Symplectic structures}\label{SectionSymplectic}

In this section we describe the \textit{real} symplectic structures on SHS manifolds. Thus we will assume that symplectic forms are real, unless otherwise stated. 

\subsection{Non-exact symplectic structure}\label{NonExactSymplonCSR}

In this section we show that an arbitrary semiprojective $(\M,\Fi)$ admits a non-exact %Calabi-Yau 
\KH structure which is invariant under the $S^1$-part of $\Fi.$
Recall first that, given a function $H$ on a symplectic manifold $(M,\om),$ its \textbf{Hamiltonian vector field} $X_H$ is defined by $\om(\cdot, X_H)=dH(\cdot).$ The flow of this field preserves the symplectic form $\om.$ In particular, given a symplectic manifold $(M,\om)$ with a symplectic $S^1$-action, we call its \textbf{moment map} a function $H$ whose Hamiltonian vector field $X_H$ is equal to the vector field of the $S^1$-flow. The moment map need not to exist in general, though we will show that it does in our setup.

\begin{lm} \label{thereisaNonExactstructure} Any semiprojective variety $(\M,I,\Fi)$ has an $I$-compatible $S^1$-invariant %Calabi-Yau 
	\KH structure $\om_I$
that admits a moment map $H:\M \fun \R.$
The vector field of the $\R_+$-part of $\Fi$ is $X_{\R_+}=\nabla H$ and the vector field of the $S^1$-part is $X_{S^1}=X_H.$  %In particular, $X_H=I \nabla H.$
\end{lm}
\begin{proof}
	By Lemma \ref{CompletionLemma}, there is a $\C^*$-equivariant completion $\M \hookrightarrow  Y,$ where $Y$ is a smooth projective variety, thus has a \KH structure. 
	Averaging its Riemannian
	metric over $S^1,$ one can make it $S^1$-invariant. Thus, we get an $S^1$-invariant \KH structure on $Y,$ which has non-empty set of fixed points 
	(as $Y^{S^1} \supset \M^{S^1} = \M^{\Fi} \neq \emptyset$), thus by the proof of \cite[Lem. 2]{Frankel59}, it has a moment map. 
	As $\M\subset Y$ is an open $S^1$-invariant subset, restricting to $\M$ we get the \KH structure $(\M,g,I,\om_I)$ together with the moment map $H:\M \fun \R.$
	 
%	One can pull-back the Fubini-Study form $\om_{FS}$ from $\C P^N$, and then average it over the $S^1$-part $\Fi_t$ of the $\C^*$-action,
%	$$\om_I:=\int_{S^1} \Fi_t^*(\iota^* \om_{FS})dt,$$ 
%	getting an $I$-compatible $S^1$-invariant symplectic form $\om_I$ on $\M.$ The I-compatibility exactly means that $g(\cdot,\cdot):=\om_I(\cdot,I\cdot)$ is a Riemannian metric. As $I$ is a complex structure, the triple $(g,I,\om_I)$ forms an $S^1$-invariant \KH structure. 
%	
%	It is Calabi-Yau due to Lemma \ref{LemmaCalabiYau}.
	
	The $S^1$-part ($X_{S^1}=X_H$) is immediate from the definition of the moment map. Then, as the $\C^*$-action is holomorphic, we have $X_{S^1}=I X_{\R_+}.$ On the other hand, from
	$$\om_I(\xi,X_H)=dH(\xi)=g(\xi, \nabla H)=\om_I(\xi,I \nabla H),$$ we get that $X_{H}= I \nabla H,$ and thus the claim $X_{\R_+}=\nabla H$ follows immediately.
\end{proof}
	
For a curious reader, we will explain why the symplectic form of this \KH structure on $\M$ is almost never exact.
	
\begin{lm}
	When $\M$ is not a point or equal to $\C^{2n},$ for some $n,$ the symplectic form of \KH structure from Lemma \ref{thereisaNonExactstructure} is non-exact.
\end{lm}
\begin{proof}
	Let us assume that $\M$ is not a single point or $\C^{2n},$ for some $n.$ 
	Then, the core $\L$ is not a single point, as otherwise $\M$ would contract to it via the $\C^*$-action and thus would be isomorphic to an affine space. The last fact is due to \BB decomposition applied to semiprojective varieties, \cref{BBDecompositionCSRs}. Thus the core is not a point.
	
	Assume first that it is fixed by the $\C^*$-action, i.e. $\M^{\Fi}=\L.$ Then, $\L$ is a smooth
	$I$-holomorphic subvariety, thus $\om_I$-symplectic submanifold of $\M.$
	%\footnote{Due to Theorem \ref{FixedPointsKahler}\eqref{CritLocusIsSmooth} and Lemma \ref{LemmaMomentMapMorseBott}.} 
	%		Being a fixed locus of an algebraic torus action on a smooth vareity, \cite{Iv72}. Although $\M$ is not projective as, it is smooth quasi-projective, thus normal, so one can compactify it \C^*-equivariantly due to \cite[Thm. 1]{Su74}. 
	If $\om_I$ was indeed exact on $\M$, its restriction on $\L$ would yield an exact symplectic form on a connected closed manifold, which is impossible due to Stokes theorem.
	
	Otherwise, there is a point $x\in\L$ which is non-fixed; its $\C^*$-orbit can be extended to a holomorphic map $u: \C P^1 \fun \L,$ by \cite[Lem. II-A]{So75}.\footnote{Strictly speaking, one can use this result only for a compact \KH manifold with a $\C^*$-action. However, being smooth, $\M$ is a normal variety, so due to \cite[Thm. 1]{Su74} one can compactify it $\C^*$-equivariantly to a smooth projective variety $X,$ which is then naturally \Kh.}
	Hence, if $\om_I=d\th$ was indeed closed on $\M$, the Gromov lemma together with the Stokes formula makes a contradiction,
	$0=\int_{\partial \C P^1} u^*\th = \int_{\C P^1} u^*\om_I= {1\over 2}\int_{\Sigma} ||du||^2>0.$ 
\end{proof}

\subsection{Canonical Liouville structure}\label{ExactSymplOnCSR}

In this section we construct a canonical Calabi-Yau Liouville structure on a given SHS $(\M,\Fi).$ Moreover, we prove that, at least in the case of CSRs, 
different commuting conical actions yield isomorphic Liouville structures. Let us start with a definition of Liouville manifolds:

\begin{de}\label{DefLiouvilleManifold}
	A \textbf{Liouville manifold} is a non-compact exact symplectic manifold $(M,\om=d\th)$ which has a compact submanifold $K\subset M$ with boundary, such that there is a symplectomorphism 
	\begin{equation}\label{defnLiouville}
		(M\setminus int(K),\om)\iso (\Sigma \times [1,+\infty),d(R\a))
	\end{equation}
	where $\Sigma = \partial K,\ \alpha=\th|_\Sigma,$ $R$ is the coordinate on $[1,\infty],$ and $R\alpha$ pulls back to $\theta$ via the symplectomorphism. 
	The subset $M\setminus int(K)\iso (\Sigma \times [1,+\infty))$ we will call the \textbf{convex end}.
	The flow of the vector field $Z$ defined by $i_Z \om=\th$ is called the \textbf{Liouville flow} of the Liouville manifold $(\M,\om).$ Its Lie derivative satisfies $L_Z \om =\om.$
	The \textbf{Liouville skeleton} is the set of points in $M$ which do not escape every compact set under the Liouville flow.
\end{de}

Motivated by the last definition, any vector field $Z$ on a symplectic manifold $(M,\om)$ satisfying $L_Z \om =\om$ we call a \textbf{Liouville vector field.}
There is a standard lemma for finding a Liouville manifold structure on an arbitrary open symplectic manifold.

\begin{lm} \label{LemaLiouville} 
	A symplectic manifold $(M,\om)$ such that 
	\begin{enumerate}
		\item There is a hypersurface $\Sigma=\partial K$ that bounds a compact submanifold $K$.
		\item There is a Liouville vector field $Z$ defined on $M$ which is positively-integrable and non-zero outside of $\innt(K)$.
		\item $Z \pitchfork T\Sigma$ with $Z$ pointing outside of $K.$
	\end{enumerate} is a Liouville manifold.
\end{lm}

In the setup of this lemma, the coordinate $R$ appearing in Definition \ref{DefLiouvilleManifold} is exactly $R=e^t,$ where  $t>0$ is the time needed to flow via the vector field $Z$ starting from $\Sigma$, in order to reach the given point. In particular, $\Sigma = \{R=1\}$ and $Z=R\partial_R.$ Moreover, by the Cartan formula $\om=L_Z\omega=d(i_Z \omega)+i_Z(d\omega)=d(i_Z\omega)$, 
hence we can take $\theta=i_Z \omega$ to be the primitive 1-form.

Next we define isomorphisms between two Liouville manifolds.

\begin{de}
	A \textbf{Liouville isomorphism} between Liouville manifolds $(M_1,d\th_1)$  and $(M_2,d\th_2)$ 
	is a diffeomorphism $\psi : M_1 \fun M_2$ satisfying $\psi^*\th_2=\th_1 +df$ where 
	$f$ is compactly supported.  
	%This means that on [ρ; ∞) × ∂M0 ⊂ Mˆ0 for some ρ ≫ 0, it has the form φ(r, y) = (r − f(y), ψ(y)),
	%where ψ : ∂M0 → ∂M1 is a contact isomorphism, satisfying ψ∗α1 = e^f α0 for some function f.
\end{de}

It is immediate that any Liouville isomorphism is symplectomorphism and compatible with the Liouville flow at infinity.

Let $M$ be a manifold. Given a 1-form $\th$ on it such that $d\th=:\om$ is symplectic, together with a compact hypersurface $\Sigma,$ such that the vector field $Z$ defined by $i_Z \om=\th$ satisfies the conditions from Lemma \ref{LemaLiouville}, $(M,d\th)$ is a Liouville manifold. Thus, assuming those conditions are satisfied, we will call the pair $(M,\th)$ a \textbf{Liouville structure} on $M.$
There is a deformation lemma of Moser type for Liouville structures which we will need, which follows from \cite[Lem. 5]{SeiSm05}. %from paper \cite{SeiSm05} 
\begin{lm}\label{DeformationLemma}
	Let $(M,\th_s)_{s\in S}$ be a smooth family of Liouville structures on $M$ having the same hypersurface $\Sigma,$ parametrised on a 
	connected manifold $S.$ Then all the $(M,d\th_s)$ are mutually Liouville isomorphic.
\end{lm}

\begin{rmk}
	We remark that, what we call Liouville manifold here, in \cite[Sec. 2]{SeiSm05} is called a \textit{complete finite type convex symplectic manifold}, defined as a triple $(M,\th,\phi),$ where $\phi:M\fun \R$ is a smooth exhausting function, such that the Liouville flow $Z$ of $\th$ is positively-integrable and there is a constant $c_0$ such that the level sets $\phi^{-1}(c)$ are transversal to $Z$ for $c\geq c_0.$
	Given a Liouville manifold $(M,d\th),$ one can turn it into this setup, by choosing $\phi:=log (R)$ on the convex end (where $R$ is the radial coordinate, recall Definition \ref{DefLiouvilleManifold}), and smooth it out on its complement by a cut-off function. Then, having the family of Liouville structures $(M,\th_t),$ one immediately has
	that the corresponding triples $(M,\th_t,\phi_t)$ make a \textit{complete finite type convex symplectic	deformation}, which is the condition of cited \cite[Lem. 5]{SeiSm05} that ensures Liouville isomorphisms between $(M,d\th_t).$ 
\end{rmk}

Now we prove that an SHS manifold has a canonical Liouville structure, the fact which grew from the author's conversations with Paul Seidel and Nick Sheridan.
Denote by $\om_J:=\mathbb{R}e(\om_\C)$ and $\om_K=\mathbb{I}m(\om_\C).$

\begin{prop} \label{canonicalLiouville} Any SHS manifold $(\M,\Fi,\om_\C)$ has a canonical family of isomorphic Calabi-Yau Liouville structures $(\M,a\th_J+b\th_K),$ 
	parametrised by $(a,b)\in \R^2\setminus \{0\}$ such that:
	\begin{enumerate}
		\item The Liouville vector field $Z$ is the $1/s$-multiple of the vector field of the $\R_+$-action, where $s$ is the $\om_\C$-weight of $\Fi.$
		\item The Liouville 1-forms are $\th_J=i_Z\om_J, \th_K=i_Z \om_K,$ thus primitives of the real and imaginary part of $\om_\C.$ 
		\item \label{CoreIsSkeleton} The Liouville skeleton is the core $\L$ itself.
	\end{enumerate}
\end{prop}
\begin{proof}
	As $\om_\C$-weight of $\Fi$ is equal to $s,$ we have that $t \cdot \om_\C = t^s \om_\C$ for every $t\in\C^*.$ Using the real values of $t$ and taking a derivation this gives us 
	$L_{X_{\R_+}} \om_\C= s \om_\C,$ where $X_{\R_+}$ is the vector field of the $\R_+$-part of the action. Therefore, the vector field $Z:=\frac{1}{s} X_{\R_+}$ satisfies 
	\begin{equation}\label{RplusFlowIsLiouvilleScaled}
		L_Z \om_\C =\om_\C,
	\end{equation}
	which gives 
	$L_Z \om_J =\om_J$ and $L_Z \om_K =\om_K.$ Thus, $Z$ is a Liouville vector field with respect to any linear combination of $\om_J$ and $\om_K,$ 
	and it gives primitive 1-forms $\th_J:=i_Z \om_J,$ $\th_K:=i_Z \om_K$ 
	It is positively-integrable, as it comes from the $\R_+$-action which is globally defined. 
	
	Now, in order to use Lemma \ref{LemaLiouville}
	we have to prove that there is a hypersurface $\Sigma$ bounding a compact set, such that $Z \pitchfork T\Sigma$ with $Z$ pointing outside and $Zeros(Z) \subset \innt(K).$
	First, as the core $\L$ is compact (due to Lemma \ref{CoreIsADefRetr}), 
	we can construct its compact neighbourhood $K$ whose boundary $\partial K$ is smooth. 
	Then, recall that the flow of $X_{\R_+},$ thus of $Z$ exists for all times on $\M \setminus \L.$ This gives a $(\R,+)$-principal bundle 
	\begin{equation}\label{PRincBundle}
		\M \setminus \L \fun (\M \setminus \L)/\R =:\mathcal{M}.
	\end{equation}	
	Notice that each point $x \in \partial K$ belongs to an $\R_+$-orbit in $\M \setminus \L$ and moreover each orbit in $\M \setminus \L$ intersects $\partial K,$ as it goes from infinity and eventually reaches near the core. Thus, there is a continuous surjective map $\partial K \fun \mathcal{M},$ proving that $\mathcal{M}$ is compact.
	Now, as any $(\R,+)$-principal bundle is trivial,\footnote{These are classified by homotopy classes of maps to the classifying space $B\R,$ which is contractible. The last fact is due to the long exact sequence of homotopy groups coming from the fibration $E\R \fun B\R,$ and Whitehead theorem.} there is a smooth section $f: \mathcal{M} \fun \M \setminus \L$ of \eqref{PRincBundle}
	whose image $\Sigma:=f(\mathcal{M})$ is then a compact hypersurface in $\M\setminus \L$ transverse to the $(\R,+)$-action, hence, to flow of $Z.$ 
	
	Hypersurface $\Sigma$ bounds a compact set, call it $K.$ It is left to prove that $Z$ is outward pointing on $\Sigma$ and $Zeros(Z) \subset \innt(K).$ 
	Indeed, if $Z$ is pointing inward on $\Sigma,$ then each point $x \in \Sigma$ will have a limit of flow of $Z,$ or equivalently of $\R_+$-action in $\innt(K),$ 
	which is impossible unless $x \in \L$ (by definition of the core), which is a contradiction. 
	Thus, $Z$ is outward pointing on $\Sigma,$ and thus each point has a backwards limit in $\innt (K),$ showing that $Zeros(Z) \cap \innt(K) \neq \emptyset,$ so the core 
    has non-empty intersection with $\innt(K),$\footnote{As $Zeros(Z)=\M^{\R_+}=\M^{\Fi} \subset \L.$} 
    thus must be completely contained in it, being connected\footnote{As it is a deformation retract of connected $\M$ (recall that in Definition \ref{DefinitionSHS} of an SHS manifold we ask for it to be connected).} and disjoint from the boundary $\Sigma.$
	
	Altogether, the pair $(\Sigma,Z)$ satisfies the conditions from Lemma \ref{LemaLiouville}, with respect to all non-zero linear combinations $a \om_J + b \om_K$ of symplectic forms, thus yields Liouville structures $\{a\th_J+b\th_K \mid (a,b)\in\R^2 \setminus \{0\}\}$ which are then all Liouville isomorphic due to Lemma \ref{DeformationLemma}.
	The Calabi-Yau property of these symplectic forms is due to Lemma \ref{LemmaCSRsAreAlmostHK}. The statement that the core $\L$ is the skeleton follows immediately from its definition.
%	In the end, we comment that the choice of the hypersurface $\Sigma$ %homogeneous generators 
%		does not affect the obtained Liouville structure, as one uses the same vector-field, hence can get a symplectomorphism between the corresponding convex ends. 
%		%Indeed, denote by $\Phi_0$ and $\Phi_1$ two polynomials obtained by the formula \eqref{polynomialForLiouvilleStructure} by using two different sets of homogeneous coordinates. 
%		Indeed, choosing
%		two hypersurfaces $\Sigma_0$ %: =\Phi_0^{-1}(c_0)$ and 
%		$\Sigma_1,$ %:=\Phi_1^{-1}(c_1).$ 
%		the flow of the Liouville vector field $Z$ yields a contactomorphism 
%		\begin{equation}\label{contactomorphismequation}
%			\Psi: (\Sigma_0,\a_0:=\th_{J,K}|_{\Sigma_0}) \fun (\Sigma_1,\a_1:=\th_{J,K}|_{\Sigma_1}), \ \Psi^*\a_1= e^{f}a_0,
%		\end{equation}
%		where $f=f(y)$ is the time of the Liouville flow that takes from point $y\in\Sigma_0$ to hit $\Sigma_1.$ The equality in (\ref{contactomorphismequation}) is due to $L_Z \th_{J,K}=\th_{J,K}$ and the fact that $\a_0$ and $\a_1$ are restrictions of $\th_{J,K}.$ 
%		Finally, $\Psi$ induces a symplectomorphism $\phi$ of different convex ends $(\Sigma_0 \times [1,+\infty))$ and $(\Sigma_1 \times [1,+\infty)),$ defined by
%		$\phi: (e^r,y)\fun (e^{r-f(y)}, \Psi(y)).$ 
\end{proof}

\begin{rmk}
In particular, the Liouville isomorphisms of the structures $(\M,a \th_J + b \th_K)$ give us the monodromy map landing in 
$\text{Symp}(\M,\om_J),$ obtained by going around the origin in the parameter space $\R^2\setminus \{0\}.$
Studying these monodromy maps for certain SHS manifolds might be an interesting further avenue of our research.
\end{rmk}

We will denote the family of linear combinations used in the previous proposition by 
\begin{equation}
	\th_{J,K}:=\{a\th_J + b \th_K \mid \R^2\setminus \{0\} \}, \ \ \om_{J,K}:=\{a\om_J + b \om_K \mid (a,b)\in \R^2\setminus \{0\} \}.
\end{equation}
As the family $(\M,\th_{J,K})$ of Liouville structures are mutually Liouville isomorphic, 
all corresponding symplectic cohomologies $SH^*(\M,\Fi,\om_{J,K})$ are mutually isomorphic.\footnote{This is a standard statement, see e.g. \cite[p. 12]{Sei08}.} %Appendix \ref{LiouvilleMfldsAppendix}).
In fact, in the case of a CSR, one can choose the hypersurface $\Sigma$ more canonically, giving a description for the symplectic cohomology.

\begin{prop}\label{CorProofThatPhiIsRegular} 
Given a CSR $\pi:\M \fun \M_0,$ its symplectic cohomology can be represented 
	$$SH^*(\M, \Fi,\om_{J,K})=HF^*(\M,\Phi,\om_{J,K})$$ as the Floer cohomology of the polynomial $\Phi=\sum_i {|\pi^*(f_i)|}^{\frac{2w}{w_i}}$ obtained by lifts $\pi^*(f_i)$ of a set $f_i$ of $\Fi$-homogeneous generators of the coordinate ring $\C[\M_0]$. Here $w_i$ is the weight of $f_i,$ and $w=lcm(w_i).$
\end{prop}
\begin{proof}
	Firstly, as $\M_0$ is an affine variety with a $\C^*$-action, its coordinate ring has induced $\Z$-grading. 
	Pick a finite set $(f_i)_i$ of its homogeneous generators, denoting their weights by $w_i$ respectively, and denote by $w$ their least common multiple. 
	Define the real polynomial on $\M,$ considered in \cite[Prop. 2.5]{BPW16}
	\begin{equation*}\label{polynomialForLiouvilleStructure}
	\Phi:=\pi^* (\sum_i |f_i|^{\frac{2w}{w_i}}).
	\end{equation*}
	As $\pi$ is $\C^*$-equivariant and the polynomials $f_i$ have weights $w_i,$ denoting the $\C^*$-action on $\M$ by $\Fi_t$
	we have $\Fi_t^* \Phi =| t |^{2w}\Phi,$ thus
	$L_Z \Phi=2w\Phi.$ As $L_Z \Phi = {d} \Phi (Z),$ we see that $0$ is the only singular value of $\Phi,$ and hence for any $\Phi_0>0$ the hypersurface $$\Sigma:=\Phi^{-1}(\Phi_0)$$ is smooth. As ${d} \Phi (Z)=2w \Phi,$ the pair $(\Sigma,Z)$ satisfies the conditions from Lemma \ref{LemaLiouville}, so we can use $\Sigma=\{\Phi=\Phi_0\}$ as the hypersurface for Liouville structure. Let us prove that
	$\Phi=\Phi_0 R^{2w},$ where $R$ is the radial coordinate of our Liouville structure.
%	with respect to an arbitrary linear combination $a \om_J + b \om_K$ of symplectic forms, thus yields Liouville structures for each of them. 
%	The unital linear combinations	$\om_t =\cos t \om_J + \sin t \om_K$ are all Liouville isomorphic due to Lemma \ref{DeformationLemma}.
	%	{\FZ Ritter: p.10 proof of Cor.2.8 needs a bit more details. I didn't understand
	%		the "hence $\alpha=\alpha_0 R^{2w}$" statement. You seem to be claiming
	%		that $L_Z f = f$ implies $f=constant.R.$ So it may be worth adding details
	%		here.}
	Indeed, the equation $L_Z \Phi=2w\Phi$ implies that $\partial_R \log \Phi = 2w/R$ (as $Z=R\partial_R$), therefore the Liouville flow maps the level set $\Sigma=\Phi^{-1}(\Phi_0)$ to another level set of $\Phi$. This implies that in the $\Sigma\times [1,\infty)$ coordinates from \eqref{defnLiouville}, $\Phi$ is constant in the $T\Sigma$ directions. Thus the $Z$-directional derivative equation is enough to determine $\Phi$, as claimed. Finally, $\Phi=h(R)$ is an admissible Hamiltonian such that
	$\lim_{R\fun +\infty} h'(R)=+\infty$,  so its Floer cohomology $HF^*(\M,\Phi)$ gives the symplectic cohomology by definition (with suitable perturbations to achieve nondegeneracy), \cite[Sec. 3d]{Sei08}.
\end{proof}

\begin{rmk}
	In particular, computing the exact symplectic cohomology $SH^*(\M,\Fi,\om_{J,K})$ of CSRs is an interesting question in itself. Apart from the case of cotangent bundles of flag manifolds $T^*(G/P)$ (which is computable due to Viterbo isomorphism \cite{Vi96, Ab15}), amongst all CSRs the exact symplectic cohomology is computed only for resolutions of Du Val singularities of type A and D, due to Etg\"{u}-Lekili \cite[Cor. 42 and Cor. 46]{EL17}.% where they have used Legendrian surgery techniques \cite {BEE12} as a means of computation.
\end{rmk}
%Let us show that the level-sets of the polynomial $\Phi$ are connected as we will need it later.
%
%\begin{lm} Given a CSR $(\M,\Phi),$ and an arbitrary set of homogeneous generators $f_k$ of $\CM0,$ the level sets $\Phi^{-1}(a)$ of the corresponding polynomial $\Phi$ are connected.
%\end{lm}
%\begin{proof} Firstly, $\Phi^{-1}(0)=\L$ and we know that the core $\L$ is connected by Proposition \ref{ConnectedFibers}. Any other $a>0$ is a regular value of $\Phi$, so $\Sigma:=\Phi^{-1}(a)$ is a smooth hypersurface. Denote by  $\Sigma^0= \Sigma\cap \pi^{-1}(\M_0^{reg}),$ where 
%	$\M_0^{reg}$ is the set of smooth points in $\M_0.$ 
%	Thus, the restriction $\pi:\pi^{-1}(\M_0^{reg})\fun \M_0^{reg}$ is a homeomorphism.
%	
%	Now, notice that $\pi(\Sigma^0)$ is isomorphic to the quotient $\M_0^{reg}/\R_+.$ We have already proved connectedness of $\M_0^{reg}$ in Proposition \ref{CSRIsConnected}, so its quotient by the $\R_+$-action is connected as well. Thus, $\pi(\Sigma^0)$ is connected, and its homeomorphic preimage $\Sigma^0$ is connected as well. As the set $\pi^{-1}(\M_0^{reg})$ has a complement of real codimension 2 in $\M,$ its intersection $\Sigma^0= \Sigma\cap \pi^{-1}(\M_0^{reg})\subset \Sigma $ with the hypersurface $\Sigma$ is a dense subset. Thus, $\Sigma$ is the closure of $\Sigma^0,$ hence it is connected as well.
%\end{proof}
Next, we show that in the case of a CSR the Liouville structure constructed in Proposition \ref{canonicalLiouville} does not depend on the choice of the conical action, 
in a set of commuting ones.

	\begin{prop}\label{ConicalActionsCommuteInduceSameLiouville}
		Given a CSR $\M,$ two different conical actions that commute yield isomorphic Liouville structures on it.
	\end{prop}
	\begin{proof}
		Consider two conical actions $\Fi^1$ and $\Fi^2$ on $\M$ that commute and have weights $s_1$ and $s_2$ respectively. As in Proposition \ref{canonicalLiouville}, $1/s_1$ and $1/s_2$-multiples of their $\R_+$-vector fields yield two Liouville  vector fields $Z_1$ and $Z_2,$ with respect to linear combinations $a\om_J+ b \om_K$ of symplectic forms. As we proved in the same proposition that a choice of $(a,b) \neq (0,0)$ does not impact on the Liouville structure, we will use the form $\om_J$ for simplicity.
		Setting $Y:=Z_2-Z_1$, we define a family of
		Liouville vector fields $$Z_t=Z_1+tY$$ and corresponding primitive 1-forms 
		$\theta_t$ of $\omega$ defined by
		$\theta_t=i_{Z_t}\omega_J$.	%$i_{Z_t}\theta_t=\omega_J.$ 
		By Lemma \ref{DeformationLemma} it is enough to prove 
		that the vector fields $Z_t$ intersect the hypersurface\footnote{Here, $\Phi_2$ is the polynomial
			from \eqref{polynomialForLiouvilleStructure} for the action $\Fi^2.$} $\Sigma_2=\Phi_2^{-1}(c)$ transversely outwards for some $c>0,$ in other words ${d}\Phi_2 (Z_t)>0.$	
		We prove that by projecting to $\M_0.$
		
		Firstly, as the actions $\Fi^1$ and $\Fi^2$ commute, there is a set $\{f_k\}_{k=1}^N$ of generators of $\CM0$ that are homogeneous under both actions. Thus, we have an embedding $j: \M_0 \hookrightarrow \C^N, \ p \mapsto (f_1(p),\dots,f_N(p)).$ Denote by $z_1,\dots,z_N$ the coordinates in $\C^N.$ Denote the weights of $f_k$ by $w_k^1$ and $w_k^2$ under the actions
		$\Fi^1$ and $\Fi^2$ respectively, and their corresponding least common multiples by $w^1, w^2.$ 
		Then, the map $\nu:=j\circ \pi:\M\fun \C^N$ is $\C^*$-equivariant with $\Fi^1$ and $\Fi^2$, where the corresponding actions on $\C^N$ are given by $$t \cdot (z_1,\dots, z_N)=(t^{w_1^i}z_1,\dots,t^{w_N^i}z_N),$$ for $i=1$ and $i=2$ respectively. The vector fields of the $\R_+$-parts of these actions on $\C^N$ are given by the standard formulas
		$V_i(z)=\sum_{k=1}^N w_k^i (x_k\partial_{x_k} + y_k \partial_{y_k}),$ where $z_k=x_k+i y_k.$
		
		Now, fixing some $c>0,$ we want to prove ${d}\Phi_2 (Z_t(p))>0$ for an arbitrary point $p\in \Phi_2^{-1}(c).$ As $\Phi_2=\pi^*\a_2=\nu^*\a_2,$ we have 
		${d}\Phi_1 (Z_t(p))=d (\nu^*\a_1)(Z_t(p))= {d}\a_1 (\nu_*(Z_t(p))),$ where $\a_2:=\sum_k |z_k|^{\frac{2w^2}{w_k^2}},$
		and $\nu_*(Z_t)=\nu_*((1-t)Z_1 + t Z_2)= (1-t) V_1 + t V_2,$ as $\nu$ is $\C^*$-equivariant. Thus, using 
		${d}\a_2=\sum_{k=1}^N {\frac{2w^2}{w_k^2}} (x_k {d}_{x_k} + y_k {d}_{y_k})$ we have
		\begin{align*}
			{d}\Phi_2 (Z_t(p))&= {d}\a_2 ((1-t) V_1 + t V_2) = (1-t) \ {d}\a_2 (V_1) + t \ {d}\a_2(V_2)\\
			&=(1-t)\sum_{k=1}^N {\frac{2w^2}{w_k^2}}w_k^1 (x_k^2+y_k^2)+ t \sum_{k=1}^N {\frac{2w^2}{w_k^2}}w_k^2 (x_k^2+y_k^2),
		\end{align*}
		thus it is positive whenever $\nu(p)=(x_1,y_1,\dots,x_N,y_N)\neq 0,$ that is to say, for all $c>0.$ 
		
		Thus, by Lemma \ref{DeformationLemma}, $(\M,d\th_t)$ are Liouville isomorphic. In particular, 
		$(\M,d\th_0)$ and $(\M,d\th_1)$ are Liouville isomorphic as well, hence the proposition is proved.
	\end{proof}

% {\FZ QUESTION: Looking at the proof, it should work also when actions $\Phi^0$ and $\Phi^1$ do not necessarily commute, but their corresponding sets of homogeneous generators $f_k^0$ and $f_k^1$ are in linear dependence, 
% 	 \begin{equation}\label{LinDepOfGens}
% 	 (f_k^0)_{k=1}^N=L (f_k^1)_{k=1}^N, \text{ for some }L\in GL(N,\C)
% 	 \end{equation}
% 	 Thus, we have an algebraic question: Given a finitely generated commutative Poisson algebra $A$ that has two different non-negative
% 	 gradings $A=\oplus_{k\geq 0} A_k^0=\oplus_{k\geq 0} A_k^1$ such that $A_0^0=A_0^1=\C \ll 1 \ra,$ do they necessarily have 
% 	 homogeneous generators $f_k^0$ and $f_k^1$ that satisfy (\ref{LinDepOfGens})
% 	 Kevin or Alex, do you know whether this is possibly true?
% }
% {\K You seem to need the generators to be simultaneously homogeneous for the two actions, rather than linearly independent (the argument doesn't require linear independence at all I don't think). If there are generators which are simultaneously homogeneous, then the actions commute, so I think the result is sharp.
% }

There are a lot of instances of commuting conical actions to be found in the standard examples of CSRs, such as quiver varieties, Slodowy varieties, or hypertoric varieties. 

\subsection{Stein structure}

In this section we show that the Liouville structure on SHS manifolds from the previous section can be made Stein when they are also SHK.
This is true for all known examples of CSRs, or in the case of Moduli spaces of Higgs bundles.

Recall that an SHS manifold $(\M,I,\om_\C)$ is an SHK when it can be enriched with the \HK structure $(g,I,J,K)$ such that $\om_\C=\om_J + i \om_K.$ On such manifolds we consider only $\C^*$-actions whose
$S^1$-parts are $g$-isometric, hence preserve $\om_I.$ %In particular, let us assume that the $S^1$-action is $\om_I$-Hamiltonian, i.e. there is a moment map $H:\M\fun S^1.$ 
We have the following proposition, whose proof basically follows \cite[Prop. 9.1]{Hi87}.

\begin{prop}\label{SteinStructure}
An SHK manifold $(\M,\Fi),$ whose $S^1$-part has a proper moment map with respect to $\om_I,$ is a Stein manifold with respect to any linear combination of complex structures $\{J,K\}.$
\end{prop}
\begin{proof}
%Recall that on SHS $(\M,I,\om_\C)$ a compatible \HK structure $(g,I,J,K)$ is such that $\om_\C=\om_J + i \om_K,$ where $\om_S:=-g(\cdot, S \cdot).$
%Given such structure, which is $S^1$-isometric, being $I$-holomorphic, it is $\om_I$-symplectic. 
Notice first that the moment map H is bounded from below. Indeed, 
by definition, moment map $H$ satisfies $dH(\cdot)= \om_I(\cdot, X_{S^1}),$ where $X_{S^1}$ is the vector field of the $S^1$-part of $\Fi.$
Hence, as $\om_I(\cdot, \cdot)=-g(\cdot, I \cdot),$ we have $\nabla H = - I X_{S^1} = X_{\R_+}.$
Thus, the negative gradient flow of $H$ is equivalent to the $(t \fun 0)$-action flow, which has limit points and they are compact, by the semiprojectivity of $(\M,\Fi).$ 
Hence, there is a global minimum of $H,$ so it is bounded from below.

On the other hand, recall by \eqref{RplusFlowIsLiouvilleScaled} that we have $L_Z \om_J= \om_J,$ for $Z:=\frac{1}{s} X_{\R_+}$ where $s$ is the weight of $\Fi.$
By Cartan formula, that gives $d (i_Z \om_J) = \om_J,$ so considering $\th_J:= i_Z \om_J$ we have
\begin{align*}
\th_J(\cdot) &= \om_J (Z, \cdot)= \frac{1}{s} \om_J (X_{\R_+}, \cdot)= \frac{1}{s} \om_J (\nabla H, \cdot)= - \frac{1}{s} g(\nabla H, J\cdot) = -\frac{1}{s} \om_I(\nabla H, I J\cdot) \\ 
&=-\frac{1}{s} \om_I(-I X_{S^1}, I J\cdot)=\frac{1}{s} \om_I( X_{S^1}, J\cdot) = -\frac{1}{s} \om_I( J \cdot , X_{S^1}) = -\frac{1}{s}  dH ( J \cdot),
\end{align*}
Altogether, $\om = -d (dh \circ J),$ where $h=\frac{1}{s} H$ is proper and bounded from below, and as $\om_J$ is a \KH form, this is precisely (one of the equivalent) condition saying that $(\M,J)$ is Stein.
As $\om_K(\cdot, \cdot)= -g(\cdot, K \cdot)$ as and $L_Z \om_\C = \om_\C,$ the same proof works for any linear combination of $J$ and $K.$ 
\end{proof}

In the case of spaces coming as Higgs branch of gauge theories (Section \ref{HiggsBranch}), we can in fact give a criterion when this Stein structure is subcritical.
As explained in the Section \ref{HiggsBranch}, these spaces are defined as GIT quotients of complex reductive group $G$ acting on $T^*N,$ where $N$ is a complex vector space, 
and when smooth, they inherit a holomorphic symplectic structure from $T^*N.$ 
However, one can also define them via the \HK quotient construction, as the compact form $K$ of $G$ acts by \HK isometries on the vector space $T^*N$ enhanced naturally with a \HK structure. 
Moreover, the $S^1$-restriction of the $\C^*$-dilation action on the whole $T^*N=N\oplus N^*$ acts by \HK isometries and commutes with the action of $K,$ hence it passes to the quotient. 
It has the proper moment map, as it does on $T^*N,$ so we have a conclusion.

\begin{cor}\label{SubcritStein} The Higgs branch $\M_\chi$ of the gauge theory associated to a pair $(G,N)$ is a SHK manifold, hence has a Stein structure. 
	Moreover, this structure is subcritical when $N^G\neq {0}.$
\end{cor}
\begin{proof}
	The first statement is due to the above. Assuming $N^G\neq {0},$ in the proof of Lemma \ref{WhenHiggsBranchWeight1Action} we have concluded that then the core $\L_\chi$ is less than half-dimensional.
	Having in mind that the core is precisely the Liouville skeleton (Proposition \ref{canonicalLiouville}), this is one of the equivalent conditions of the Stein structure to be subcritical.
\end{proof}
\section{Smooth core components} \label{SmoothCompCSRs}

In this section we construct a family of smooth components of the core of an arbitrary weight-1 SHS $\M.$ 
Those are constructed as the minima of moment maps of the $S^1$-part of different weight-1 conical actions, thus we call them {minimal}. 
They are significant from the symplectic-topological perspective, being Lagrangian submanifolds of $\M$ with its Liouville structure, on which we discuss in further sections.
However, being smooth core components, they are also interesting in their own right. %also interesting hence are also interesting outside the realm of symplectic topology. 
For instance, when $\M$ is a Resolution of a Slodowy variety, its core is a Springer fibre, and minimal components are (some of) its smooth components, on which we discuss further in \cite{FZ2}.
%and also show that their existence yields some lower bounds on the rank of symplectic cohomology $SH^*(\M,\om_{J,K}).$

\subsection{Existence of a smooth core component} \label{existencesmoothcomp}

%We will start this section by recalling the main facts about the geometry of the core $\L$ of a CSR $\M.$ The proof is due to V. Ginzburg for an arbitrary CSR, but the key ideas are from H. Nakajima's proof of the same theorem for quiver varieties \cite[Thm. 5.8]{Nak94a}. We denote by $t\cdot x$ the $\C^*$-action, both on $\M$ and $\M_0.$

%\begin{thm}{\normalfont\cite{Gi15}} \label{LagrCoreComp} Given an arbitrary CSR $(\M,\Fi)$, its core $\L$ is an $\om_\C$-isotropic subvariety of $\M$ equal to 
%\begin{equation}\label{CoreIsSkeleton}
%\L=\bigg\{x\in \M \mid \lim_{t\fun \infty} t\cdot x \text{  exists}\bigg\}.
%\end{equation}
%Denoting the decomposition of fixed locus into connected components by $\M^{\Fi}=\F=\sqcup_i \F_i,$ there is a partition $\L=\sqcup_i \L_i$ of the core by smooth locally closed subvarieties
%\begin{displaymath} \L_i:=\bigg\{x\in \M \mid \lim_{t\fun \infty} t\cdot x \in \F_i \bigg\}.\end{displaymath}
%Moreover, if $\Fi$ is a weight-1 action, each ${\L_i}$ is a $\om_\C$-Lagrangian submanifold, thus the core $\L$ has pure dimension ${1 \over 2} \dim \M$ and its irreducible components are precisely the closures $\overline{\L_i}.$
%\end{thm}
%\begin{proof}
%	The core $\L$ is $\om_\C$-isotropic due to \cite[Thm. 4.2.1(2)]{Gi15}. We have already seen (\ref{CoreIsSkeleton}) and the partition $\L=\sqcup_i \L_i$ in the proof of Corollary \ref{BBDecompositionCSRs}. The last part regarding the weight-1 case is due to \cite[Sec. 4.5]{Gi15}.
%\end{proof}

Recall from Corollary \ref{CoreDecompositonWeight1} that, given a weight-1 SHS $(\M,\Fi),$ its core $\L=\bigsqcup_i \L_i$ has pure dimension ${1 \over 2} \dim \M$ and its irreducible components are precisely the closures $\overline{\L_i}.$ 
  
In particular, given a weight-1 SHS $\M,$ the irreducible components $\ol{\L_i}$ of its core $\L$ are proper complex varieties, thus have well-defined fundamental classes in singular homology, which freely generate the top-dimensional homology of the core $H_{top}(\L).$\footnote{Indeed, an irreducible proper complex algebraic variety admits a triangulation, making it a closed oriented connected pseudomanifold, and these have fundamental classes by \cite[Sec. 24]{SeTh80}. The statement about the core $\L$ follows from the Mayer-Vietoris sequence since intersections of irreducible components have real codimension bigger or equal to 2.}
Hence, together with Lemma \ref{CoreIsADefRetr} we get the following:

\begin{lm}\label{LemaNonIsotopic} Given a weight-1 CSR $\M,$ the fundamental classes $[\ol{\L_i}]_{i}$ make a basis of $H_{top}(\L,\Z)\iso H_{mid}(\M,\Z).$ In particular, $\ol{\L_i}$ are non-isotopic in $\M.$
\end{lm}

Given a weight-1 SHS $(\M,\Fi)$, notice that its irreducible core components $\ol{\L_i}$ in general need not be smooth. However, we will prove that at least one of them is always smooth. Recall first that by Lemma \ref{thereisaNonExactstructure}, there is an $S^1$-invariant $I$-compatible \KH form $\om_I$ on $\M,$ where $I$ is the complex structure of $\M,$ and $S^1$ acts via $\Fi.$ Moreover, by the same lemma, with respect to $\om_I$ there is a moment map $H:\M\fun \R$ for the $S^1$-action, and $\nabla H$ is the vector field of the $\R_+$-part of %the $\C^*$-action
$\varphi$.

\begin{prop}\label{minimalcomp} Given an SHS $(\M,\Fi),$ the minimum of its $S^1$-moment map $H$ is achieved along a single $\Fi$-fixed component, which lies in the core of $\M.$ 
Moreover, when $\Fi$ is weight-1 this component is an irreducible component of the core,
%$\overline{\L_{i}}$ in $\L,$ 
therefore a holomorphic Lagrangian submanifold of $(\M,\om_\C).$ 
\end{prop}
\begin{proof}
We claim that the global minimum of the moment map $H$ exists and is attained on the core $\L$. Firstly, we have
$$Crit(H)=Zeros(\nabla H)=Zeros(X_H)=Zeros(X_{S^1})=\M^{S^1} = \M^{\Fi} = \bigsqcup_{i} \F_i  \subset \L,$$
thus, if the minimum exists, it is in the core $\L.$
As the core $\L$ is proper (Lemma \ref{CoreIsADefRetr}), hence compact, the minimum of $H|_{\L}$ exists. It must be attained on a subset of $Crit(H),$ so by the equation above, the minimum is attained for a certain family of fixed components $\{\F_i\}_i.$ 

In fact, there is always a {single} such component. 
Indeed, as $\M$ is connected, equation \eqref{BettiNumbers} from Proposition \ref{HomologyDecompositionOfTheCore},
\begin{equation}\label{Bettinumberequation}
 b_k(\M) = \sum_{i \in A} b_{k-\mu_i}(\F_i),
\end{equation}
tells us that precisely one component can contribute to $b_0(\M)=1,$ and that component $\F_i$ must have $\mu_i=0$. Here, $\mu_i$ is the dimension of the bundle 
$\L_i\fun \F_i, \ x\mapsto \lim_{t\fun \infty} t\cdot x,$ which is zero precisely when there are no $\nabla H$ flowlines converging to $\F_i$ (recall that $\nabla H$ is the vector field of the $\R_+$-action). But that exactly happens when $\F_i$ is the minimum. This proves the first part of the proposition.

Assume now that $\Fi$ is a weight-1 action. %We want to show that its minimum $\F_i$ is an irreducible component of the core $\L.$ 
By previous, we have that the dimension of bundle $\L_i\fun \F_i$ is zero, so $\L_i=\F_i.$
Moreover, as $\F_{i}$ is already a closed subvariety of $\M,$ we have that 
\begin{equation}\label{TheEpicenterOfThesis}
	\overline{\L_{i}}=\L_{i}=\F_{i}.%=\L_{i_0}.
\end{equation}
Hence, together with Corollary \ref{CoreDecompositonWeight1}, we conclude that the irreducible component of the core $\overline{\L_{i}}$ is smooth. 
It is a Lagrangian submanifold of $(\M,\om_\C),$ due to Lemma \ref{CoreIsIsotropic}.
\end{proof}

\begin{rmk} We remark that, as the careful reader would notice, one does not need to introduce a moment map in the previous argument at all. 
Indeed, the Betti number equation \eqref{Bettinumberequation} already gives us the existence of a component $\F_i$ having $\mu_i=0,$ thus a component equal to its attracting set $\L_i,$ hence an irreducible component. 
However, the author firstly came up with the moment map viewpoint and only later learned of the algebraic-geometric one, so decided to stick to this first proof, thinking that it provides a richer
perspective to the geometry of this manifolds as well. Also, one could notice that the \eqref{Bettinumberequation} is precisely the condition for the moment map $H$ to be perfect, 
which is proved for $S^1$-actions on compact \KH manifolds in \cite{AtB83,Ki84}. 
Even in the non-compact case (such is ours), as Nakajima notices in \cite[p.63]{Nak99}, this would hold if $H$ was a proper function. 
This for example holds in the CSR case or for the Moduli space of Higgs bundles, but the issue is that we do not know whether we can make $H$ to be proper function for general SHS manifolds, hence cannot claim the equation \eqref{Bettinumberequation} purely from the moment map perspective and the Morse-Bott theory.
\end{rmk}

In particular, for the case of a Moduli space of Higgs bundles (Section \ref{HiggsModuli}), the component of its core that Proposition \ref{minimalcomp} recovers is the well-known moduli space of stable vector bundles \cite{AtB83}. In the case of $\Hilb^n(T^*\Sigma)$ (Section \ref{HilbSchemes}) this recovers the $\Sym^n(\Sigma),$ the biggest component in the decomposition \eqref{FixedPointsHilbTSigma}. 
Unlike these two cases, in the examples of CSRs one generally gets a family of different core components considering different actions, as we are going to see in  
Example \ref{DuVal_Weight1actions_TypeA}, and more generally for quiver varieties of type A in the forthcoming article \cite{FZ2}.

We also show a converse of Proposition \ref{minimalcomp}, that is, if an action has a fixed component, it must be a power of a weight-1 action. We first need a technical lemma, whose proof is due to Alexander Ritter and Kevin McGerty.

\begin{lm}\label{rootExtends}
	Consider a smooth irreducible complex algebraic variety $X$ with an algebraic $\C^*$-action $\Fi.$ If it has an $s$-th root $\psi$ on a Zariski open subset $U,$ then it extends to an holomorphic $s$-th root on the whole of $X.$
\end{lm}
\begin{proof}
	Firstly, as an algebraic variety $X$ is a separable scheme, the diagonal $\Delta\subset X\times X$ is a closed subvariety. Then, seeing the actions as functions $\Fi,\psi:\C^*\times X \fun X,$ from the condition that $\Fi=\psi^s$ holds on $U,$ we have
	that $\Fi(\eps t,x)=\Fi(t,x)$ holds for $x\in U$, for any $s$-th root of unity $\eps$. In other words,
	the image $F(\C^*\times U)$ of the morphism 
	$$F:\C^* \times X \fun X\times X,\ (t,x)\mapsto (\Fi(t,x),\Fi(\eps t,x))$$
	lies in the diagonal $\Delta.$ As $X$ is irreducible, $\C^*\times X=\ol{\C^*\times U}.$ Thus, by continuity of $F,$  $$F(\C^*\times X)=F(\ol{\C^*\times U})\subset \ol{F(\C^*\times U)}\subset \ol{\Delta} = {\Delta}.$$ Hence, 
	$\Fi(\eps t,x)=\Fi(t,x)$ holds everywhere on $X.$
	Thus, the map $\ol{\psi}(t,x):=\Fi(t^{1/s},x)$ is well-defined for all $x\in X.$
	It is a holomorphic map, as this is a local condition, and one can always locally choose a holomorphic branch of the $s$-th root.
	Furthermore, it agrees with $\psi$ on $U,$ thus it is a holomorphic extension of $\psi$ to $X.$
\end{proof} 

\begin{prop}\label{HavingAMinimalCompMeansWeight1Action}
	Consider a weight-$s$ SHS $(\M,\Fi)$ and suppose that the action $\Fi$ fixes a Lagrangian submanifold. Then there is a weight-1 conical action $\psi$ such that $\Fi=\psi^s.$ 
\end{prop}
\begin{proof}
	Denote by $\F$ the $\Fi$-fixed Lagrangian %component of the core.
	%	There cannot be more than one such component due to the same Morse-Bott argument in the proof of Proposition \ref{minimalcomp}). That 
    Notice that its weight-decomposition \eqref{EqnWeightDecomp} 
	%(Theorem \ref{FixedPointsKahler} and Lemma \ref{LemmaMomentMapMorseBott}) 
	$$T_\F \M= H_0 \oplus H_k$$
	consist just of spaces $H_0= T\F$ and its $\om_\C$-dual $H_k.$ 
	Indeed, as $\F$ is a Lagrangian, $\dim_{\C}H_0={1 \over 2} \dim_{\C}\M$, therefore by the non-degenerate pairing \eqref{PairingByOmC} also $\dim_{\C} H_k={1 \over 2}\dim_{\C}\M$, 
	so $H_0\oplus H_k$ already has the full dimension, and there cannot be any more weight spaces in the decomposition.
	Thus, on the corresponding attracting bundle $\D\fun\F$ (recall \BB pieces, \eqref{BBdecompWeightSpaces}) the action acts by weight $k,$ hence one can define the $k$-th root of the action $\Fi$ on it, call it $\psi.$ The bundle $\D$ has maximal dimension in $\M,$ as its tangent space at $\F$ is $H_0 \oplus H_k.$ 
	%	as the core is Lagrangian so $\F,$ and thus $H_0$ is half-dimensional, and $H_k$ is as well by duality. 
	Thus, as $\D$ is a locally closed subset of maximal dimension in the irreducible $\M,$ its closure $\ol{\D}$ must be equal to $\M,$ and $\D$ is open in it.\footnote{Recall that locally closed means open in its closure.}
	Hence, one can use Lemma \ref{rootExtends} to conclude that $\psi$ is extendable from $\D$ to an action on the whole $\M,$ such that $\Fi=\psi^k.$ Thus, $\psi$ is a weight-1 action. Recall that the conical property for the action means that under the limit $t\fun 0$ the $\C^*$-action contracts $\M$ to the compact fixed locus. As the fixed loci of $\psi$ and $\Fi$ are equal, and $\Fi=\psi^k$ is conical, $\psi$ must be conical as well.
\end{proof}

\begin{rmk}\label{AmongWeight2OnlySeekForEven}
In particular, considering a weight-2 action that has a fixed component, from the last Proposition we get that it \textbf{has} to be an even action (a square of an action). We will have this in mind while searching for weight-2 actions that yield minimal components, in Section \ref{SubsectionConstructingWeight1ConicalActions}.
\end{rmk}

\subsection{Minimal components}\label{SectionMinimalComponents}
The smooth core component from Proposition \ref{minimalcomp}, obtained as the minimum locus of an $S^1$-moment map, is fixed under a conical weight-1 action. Moreover, by the Betti number argument in the proof, any such action will have a single minimal component. That motivates the following definition: 
%{\FZ REFER BACK TO Lemma \ref{CommutingActionsSameCore}, saying if choose a family of commuting actions, same core, and same Liouville structure in CSR, but anyhow exact wrt the same liouville 1-form}

\begin{de} Given an SHS $(\M,\Fi)$, a \textbf{minimal component} $\F_\phi$ of its core $\L_\Fi$
is the component that is fixed under a weight-1 conical action $\phi$ that commutes with $\Fi.$\footnote{We need the commuting condition in order to ensure that these actions have the same core, recall Lemma \ref{CommutingActionsSameCore}.}
We denote by $\text{Con}_1(\M,\Fi)$ the set of all weight-1 conical actions commuting with $\Fi$ and by
$$\mathrm{Min}(\M,\Fi):=\{\F_\phi \mid \phi\in \Con1\}.$$ the collection of all minimal components in the core $\L_\Fi.$
\end{de}

In principle, given an arbitrary SHS, we can have many weight-1 conical actions. We show that the different actions yield different minimal components, when they commute.

\begin{prop}\label{DifntActionsDiffntMinComp}
Let $\M$ be a weight-1 SHS. Different commuting weight-1 conical actions on it yield different minimal components of its core.
\end{prop}
\begin{proof} Firstly, recall that by Lemma \ref{CommutingActionsSameCore} we have that commuting conical actions yield the same core. %{\FZZ}
Let us pick a \KH metric $g$ on $\M$ from Lemma \ref{thereisaNonExactstructure}. 
Having two commuting conical actions $\varphi^1$ and $\Fi^2$, we can integrate over their $S^1$-actions to get a \KH metric
$$\widetilde{g}:=\int_{S^1\times S^1} (\Fi_t^1)^*(\Fi_s^2)^* g \,dt\,ds$$
that is $S^1$-invariant for both actions (since they commute).
Now, let us assume that the actions $\varphi^1$ and $\Fi^2$ have weight 1 and that they have the same minimal component $\F_{min}.$ 
Recall the weight-decomposition of a fixed point $x\in \F_{min}$ %The tangent space of an arbitrary point  has the induced $\C^*$-action on it, hence splits 
\begin{equation}\label{weightdecomp}
T_x \M= \bigoplus_{k\in \Z} H_k,
\end{equation}
where $H_0=T_x \F_{min}$ does not depend on the action.
%according to the weight decomposition of the action. d
As the $S^1$-action preserves the Hermitian structure $\langle \cdot ,\cdot \rangle=\widetilde{g}(\cdot,\cdot)-i\widetilde{g}(\cdot,I\cdot),$ the weight decomposition (\ref{weightdecomp}) is $\widetilde{g}$-orthogonal.
%$$\langle t^m v ,t^n w \rangle= t^{m-n} \langle v ,w \rangle$$
The symplectic form $\om_\C$ has weight 1, so as in Lemma \ref{CoreIsIsotropic} it induces a non-degenerate pairing
$$\om: H_k \oplus H_{1-k} \fun \C$$
between the weight spaces $H_k.$ As $T_x \F_{min}=H_0$ and its dimension (being a Lagrangian) is half of the dimension of $\M,$ we deduce that $$T_x \M= H_0 \oplus H_1.$$ As this is an orthogonal decomposition, we have that $H_1=H_0^{\perp},$ so it is independent of the action. Hence, two actions $\varphi_1$ and $\varphi_2$ induce the same $S^1$-actions on the normal bundle $N\F_{min}$ of $\F_{min}.$ 

By the equivariant tubular neighbourhood theorem for isometric actions of compact Lie groups \cite[Thm. 2.2, Ch. VI]{Bre72}, there is an $S^1$-invariant tubular neighbourhood $\mathcal{N}\F_{\min}$ of $\F_{\min}$ given 
%by the composition of a contraction $\psi$ of normal bundle to a neighbourhood of a zero section and 
by the $S^1$-equivariant exponential map
%$$\phi= exp_{\epsilon} \circ \psi: N\F_{\min} \fja{\iso} \mathcal{N}\F_{\min}.$$
$$\phi: U \fja{\iso} \mathcal{N}\F_{\min},$$
where $U\subset N\F_{\min}$ is a neighbourhood of the zero section.
Hence, the restrictions to $S^1$ of the two actions $\varphi^1$ and $\Fi^2$ agree on the open subset $\mathcal{N}\F_{min}$ of $\M.$ As they act holomorphically, due to analytic continuation they need to agree on the whole of $\M.$ As $\C^*$ is the complexification of $S^1,$ there is a unique holomorphic extension of a holomorphic $S^1$-action, hence the $\C^*$-actions $\varphi^1$ and $\Fi^2$ agree as well, and the proposition is proved.
\end{proof}

%In the setup when the SHS manifold is an SHK (recall Definition \ref{DefinitionHKCSR}), we can omit the commuting condition:

%\begin{prop}\label{DifntActionsDiffntMinCompHKcase}
%Let $\M$ be a weight-1 SHK manifold. Different weight-1 HK conical actions on it induce different minimal components of its core.
%\end{prop}
%\begin{proof}
%By definition, the $S^1$-part of each HK conical action preserves the \HK metric $g$ on $\M.$ Hence, given two different HK conical actions, we already have the metric that is preserved by their $S^1$-parts, so the proof of Proposition \ref{DifntActionsDiffntMinComp} goes through.
%\end{proof}

%Strictly speaking, Propositions \ref{DifntActionsDiffntMinComp} and \ref{DifntActionsDiffntMinCompHKcase} are independent of each other, but one should bear in mind that in the all known examples to date, CSR are HKCSR and all known conical actions on them are HK conical actions, so one can use (the stronger) Proposition \ref{DifntActionsDiffntMinCompHKcase} which is a finer one, as it does not need the commutativity between the two actions as a condition. Altogether, we have the following theorem:
Thus, composing Propositions \ref{minimalcomp} and \ref{DifntActionsDiffntMinComp} we have the following:
\begin{thm}\label{MinimalComponentsTheorem}
	Given a weight-1 SHS $(\M,\Fi),$ there are at least $N$ smooth irreducible components of its core, 
	%which are smooth exact non-isotopic Lagrangian submanifolds of $(\M,\om_{J,K}),$ 
	where $N$ is the maximal number of commuting weight-1 conical actions on $\M.$
%	=\max \{N_1,N_2\}\geq 1$ and
%	\begin{enumerate}
%		\item $N_1=$ the maximal number of commuting weight-1 conical actions.
%		\item $N_2=$ the number of HK conical actions (if $\M$ is also a HKCSR).
%	\end{enumerate}
\end{thm}

Now let us give an example of how this theorem can be applied in practice.

\begin{ex}\label{DuVal_Weight1actions_TypeA}
	Consider the minimal resolution of Du Val singularity of type $A_{n}$ (Example \ref{DuValResolutionsExampleSHS}). 
	It is given as the minimal resolution
	$$\pi_{\Z/(n+1)}: X_{\Z/(n+1)} \fun \C^2/\Z/(n+1).$$
	The quotient singularity $\C^2/\Z/(n+1)$ is isomorphic to subvariety $XY=Z^{n+1}$ of $\C^3,$ given by the categorical quotient map
	$\C^2 \fun \C^3, (z_1,z_2)\mapsto (z_1^{n+1},z_2^{n+1}, z_1z_2).$
	The standard action\footnote{The one coming from weight-1 dilation in $\C^2.$} in these coordinates is hence equal to $$t\cdot (X,Y,Z)=(t^{n+1} X, t^{n+1} Y, t^2 Z).$$
	However, it is a weight-2 action and we are rather looking for weight-1 actions. It is easy to see that its square root exists iff $n$ is odd.
	Moreover, $n$ different actions
	$$t\cdot_k (X,Y,Z):=(t^{k} X, t^{n+1-k} Y, t Z), \ k=1,\dots,n $$ 
	are all weight-1 and lift to the resolution. They all commute and are different, hence produce exactly $n$ different core components on $X_{\Z/(n+1)}.$
	As we know that the core $\pi_{\Z/(n+1)}^{-1}(0)$ is an $A_n$ Dynkin tree of spheres, we see that this method exhaust \textbf{all} core components in this case.
%			 Du Val singularities of type A (Example \ref{DuVal_MinCompAreAll_QuiverVarSide}). Using the Maffei-Lusztig description for the coordinate ring of the affine quiver variety 
%		(Theorem \ref{CoordRingQuivVar}), one can pass to the usual description of the coordinate ring of Du Val singularity of type $A_{n-1},$
%		$\C[\M_0]\iso \C[X,Y,Z]/V(XY-Z^n),$ together with its Poisson brackets, given by
%		$\{X,Y\} = - n Z^{n-1}, \{Z, X\} = X, \{Y, Z\} = Y.$
%		For $n\geq 3,$\footnote{The case $n=2$ is covered in the next example, as then $\M=T^*\C P^1.$} we can directly check that the graded Poisson automorphism of the ring $\C[\M_0]$ must be of 
%		type $(X,Y,Z)\fun(tX,t^{-1} Y, Z)$  or $(X,Y,Z)\fun(tY,(-1)^n t^{-1} X, - Z).$ Automorphisms of the second type swap the exceptional spheres in the resolution, thus act non-trivially on the homology, hence are not in the connected component of identity. Finally, one can compute that the group $GL(\ww)$ acts on $\M_0$ in these coordinates by $t\cdot (X,Y,Z)=(t^n X, t^n Y, t^2 Z),$ i.e. exhausts all Poisson automorphisms of the first type. The equality (\ref{KevinNevinsEquation}) follows, as a conical symplectomorphism $\phi$ on $\M$ projects to a Poisson automorphism $\phi_0$ on $\M_0,$\footnote{Recall that a conical symplectomorphism is $\pi$-compatible.} thus picking an element $g\in GL(\ww)$ whose action on $\M_0$ is equal to $\phi_0,$ its action on $\M$ is equal to $\phi$ on an open dense subset, hence by analytic continuity, on the whole $\M.$
%	%	
\end{ex}

At the end of this section, we give an interesting proposition showing how algebraic condition on actions have geometric consequence on their corresponding minimal components.
\begin{prop} Consider an SHS $\M$ with two different commuting conical weight-1 actions $\Fi^1, \Fi^2.$
	If their composition is an even action, minimal components $\Fminn, \Fminnn$ are disjoint.
\end{prop}
\begin{proof}
	Given two commuting weight-1 conical actions $\Fi^1$, $\Fi^2,$ the composition $\Fi^{12}_t:=\Fi^1_t \circ \Fi_t^2$ is a weight-2 conical action. 
	Assuming that the intersection $\Fminn \cap \Fminnn$ is non-empty, let us show that  
	%Let us show that
	$$\Fminn \cap \Fminnn=\F_{\Fi^{12}},$$ where $\F_{\Fi^{12}}$ is the component of the fixed locus of $\Fi^{12}$ on which the minimum is attained, 
	well-defined by the first part of Proposition \ref{minimalcomp}. 
	
	Firstly, notice that the completion Lemma \ref{CompletionLemma} holds for any algebraic group action, as the references used therein apply verbatim. 
	In particular, it holds for the $\C^*\times \C^*$-action given by the product $\Fi^1 \times \Fi^2.$ Then, like in the Lemma \ref{thereisaNonExactstructure}, we conclude that there is 
	a \KH structure on $\M$ with the \KH form that is $S^1\times S^1$-invariant, which has moment maps $H_1$ and $H_2$ with respect to the actions $\Fi_1$ and $\Fi_2.$ 
	As they commute, we also get the moment map $H_{12}= H_1 + H_2$ for the action $\Fi^{12}.$
	%
	%Like in the proof of Proposition \ref{DifntActionsDiffntMinComp}, there is a \KH metric $g$ that is $S^1$-invariant for both $\Fi^1$ and $\Fi^2$. It is also $S^1$-invariant for $\Fi^{12}$, and
	%
	% there exist corresponding moment maps $H^1,H^2,H^{12}$ of those $S^1$-actions with respect to the \KH form $\om_I.$ From the general theory on symplectic reduction, we have $H^{12}=H^1+H^2$, 
	Thus, the minimum of $H^{12}$ is attained exactly at the intersection of minima of $H^1$ and $H^2,$ hence $\Fminn \cap \Fminnn=\F_{\Fi^{12}}.$
	
	Now, if the composition $\Fi^{12}$ is an even action, its square root is weight-1 conical, hence the minimal component $\F_{\Fi^{12}}$ is an irreducible component of the core.
	Being the intersection of irreducible components $\F_{\Fi^1}$ and $\F_{\Fi^2},$ this cannot work unless they are all coincide. 
	But equal minimal components imply equal actions, by Proposition \ref{DifntActionsDiffntMinComp}, hence we get the contradiction with our assumption that $\Fi^1$ and $\Fi^2$ are different.
\end{proof}

\begin{rmk}
	The other direction of the last claim of the previous theorem is not true. As an example, consider the resolution of a Du Val singularity $X_{\Z/5}\fun \C^2/\Z/5,$ 
	%(Example \ref{DuVal_Weight1actions_TypeA}), 
	and compose the weight-1 actions whose minimal components are the ``outer spheres'' in the Dynkin $A_4$-chain of spheres which constitute the core. 
	These actions are explicitly $t\cdot (X,Y,Z)=(tX,t^4Y,tZ)$ and $t\cdot (X,Y,Z)=(t^4 X,t Y, t Z).$
	The composition of these two actions $t \cdot (X,Y,Z)=(t^5 X,t^5 Y, t^2 Z)$ is not even, but their minimal components are disjoint.
\end{rmk}

\subsection{Constructing weight-1 conical actions on CSRs}\label{SubsectionConstructingWeight1ConicalActions}

Theorem \ref{MinimalComponentsTheorem} allows us to find smooth core components of SHS manifolds using weight-1 conical actions that commute. Here we show how one could find a family of such actions in the case of Conical Symplectic Resolutions.

Firstly, CSRs usually come with a natural weight-2 action. Typical examples are Higgs branch of gauge theories $(G,N)$ from Section \ref{HiggsBranch}, where the dilation action on the vector space
$T^*N$ yields a weight-2 action on the reduced space. On the Springer-theoretic side, there is a natural Kazhdan action that acts with weight-2 on Slodowy varieties.

Nevertheless, having a weight-2 CSR $(\M,\phi),$ we will show how one can construct a family of commuting weight-1 actions on it. We will start by the following definition first:

\begin{de}\label{DefinitionConicalSymplectomorpshimsGroup}
	Having a weight-2 CSR $(\M,\phi),$ we define its group of \textbf{conical symplectomorphisms} $Symp_{\phi}(\M,\omC)$ as the group of algebraic $\pi$-compatible (preserving the fibres of $\pi:\M\fun \M_0$) symplectomorphisms that commute with $\phi.$
\end{de}

\begin{lm} The group $Symp_{\phi}(\M,\omC)$ is finite-dimensional.
\end{lm}
\begin{proof}
	Firstly, we will check that $Symp_{\phi}(\M,\omC)$ acts faithfully on $\Gamma(\M,\O_\M)\iso \C[\M_0].$\footnote{This isomorphism holds due to Remark \ref{diffntdefnsOfCSR}.} It is enough to prove that if an element $\Fi$ of $Symp_{\phi}(\M,\omC)$ fixes $\Gamma(\M,\O_\M)\iso \C[\M_0],$ then $\Fi=Id.$ Now, if $\Fi$ fixes $\Gamma(\M,\O_\M)\iso \C[\M_0],$ that means that the induced map on $\M_0$ is identity (recall $\M_0$ is affine), hence the map on $\M$ is the identity on the open dense set $\M^{reg}=\pi^{-1}(\M_0^{reg})$. Moreover, the set of points in $\M$ that are fixed by $\Fi$ is closed, so we conclude that it has to be the whole of $\M$ (recall that $\M$ is irreducible).
	
	Hence, the induced action $Symp_{\phi}(\M,\omC)\dejstvo \C[\M_0]$ is faithful. Further, as elements of $Symp_{\phi}(\M,\omC)$ commute with $\phi,$ they preserve the grading on $\C[\M_0]$ induced by it. Thus, fixing some set $(f_i)_i$ of homogeneous generators of $\C[\M_0]$ whose weights we denote by $(w_i)_i,$ we have that the action $Symp_{\phi}(\M,\omC)\dejstvo \C[\M_0]$ is determined by the induced actions on the weight-spaces $\C[\M_0]^{w_i},$ on which this group acts linearly. Thus, we get the induced monomorphism $Symp_{\phi}(\M,\omC)\hookrightarrow \prod_i GL(\C[\M_0]^{w_i})$ into a finite-dimensional
	group, hence $Symp_{\phi}(\M,\omC)$ is finite-dimensional itself.
\end{proof}
Now, we construct a family of weight-2 actions by composing $\phi$ with 1-parameter subgroups $$S_t \leq Z(Symp_{\phi}(\M,\omC))$$ of the centre of $Symp_{\phi}(\M,\omC).$
These subgroups all lie in the identity-component $Z(Symp_{\phi}(\M,\omC))^{\circ}$ of this centre, which is isomorphic to $(\C^*)^n,$ for some $n\in\N.$ Thus, these subgroups %1-parameter subgroups $$S_t\leq Z(Symp_{\phi}(\M,\omC))^{\circ}$$ 
are labelled by some integer lattice $L_\phi(\M,\omC)\iso \Z^n.$

%{\A Doesn't this line of argument turn it into a harder problem about the existence of lifts of algebraic symplectomorphisms? Also you keep working with Symp, but since $H^1=0$ it may be easier to just work with Ham - this should reduce the problem to Poisson commuting functions? (although my bet would be that your Symp-phi group is infinite dimensional since Symp always is)}
Then, the 1-parameter subgroups $S_t$ for which the actions $\phi_t S_t$ are \textbf{conical} and \textbf{even}
%\footnote{Meaning: being a square of an action.} 
correspond to some subset of the lattice $\Theta_\phi \subset L_{\phi}(\M,\omC).$ By construction, these actions corresponding to $\Theta_\phi$ commute, thus by Theorem \ref{MinimalComponentsTheorem} they induce different minimal components in the core. As the number of core-components is finite, so is the number of these actions, and the set $\Theta_\phi$ itself. 
%We prove some more geometric properties of this set in the following lemma:
%\begin{lm}
%	The set $\Theta_\phi \subset L_{\phi}(\M,\omC)$ that corresponds to even and conical actions is a convex subset of $\Z^n$
%	%and central-symmetric with respect to the origin $0\in\Z^N.$
%\end{lm}
%\begin{proof}
%Firstly, as the action of an abelian group $Z(Symp_{\phi}(\M,\omC))^{\circ}$ on $\C[\M_0]$ commutes with $\phi,$ there is a set $(f_i)_i$ of $\phi$-homogeneous generators of $\C[\M_0]$ which are preserved by the action of $Z(Symp_{\phi}(\M,\omC))^{\circ}.$ Denote their 
%$\phi$-weights by $(w_i)_i$ respectively. The central symmetry of \Theta_\phi basically means that if 
%
%that is symmetric with respect to the origin. Without proving the last sentence, we just explain it here briefly. Firstly, the even actions $\phi_t S_t$ correspond to a sublattice of $L_{\phi}(\M,\omC)_{fin}.$ Namely, consider two even actions $\phi_t S_t^1=\Fi_1^2, \ \phi_t S_t^2=\Fi_2^2,$ which correspond to two points $m_1$ and $m_2$ in the lattice. We get that $(S_t^1)^{-1} S_t^2,$ and thus 
%	$\phi_t S_t^2 (S_t^1)^{-1} S_t^2)^k$ is even, which than in the lattice notation gives $m_2+(m_2-m_1)k,$ thus we get a sublattice condition.
%The conical condition on the action can be seen as a set of , when considering the way that action acts on polynomials $\C[\M]$
%
%composition of even actions is even
%Then, the condition that $\phi_t S_t$ is conical
%\end{proof}

The square roots of the actions that correspond to $\Theta_\phi$ yield weight-1 conical actions, thus minimal components (Moreover, by Remark \ref{AmongWeight2OnlySeekForEven}, these are the only ones in $L_{\phi}(\M,\omC)$ that could yield minimal components). Thus, we have a corollary to Theorem \ref{MinimalComponentsTheorem}. 

\begin{cor}
	Given a weight-2 CSR $(\M,\phi),$ the number of minimal components of its core is at least $|\Theta_\phi|.$ 
\end{cor}

%The natural question to ask is whether every CSR has a weight-2 action, and moreover, whether there is a canonical one $\phi$ such that all other conical actions commute with it. If that is true, the set $\Theta_\phi$ would yield all possible minimal components.\footnote{more precisely, all minimal components that arise from weight-1 actions which are algebraic and $\pi$-invariant.}
%
%\begin{que}\label{QuestionWeight2ConAction}
%	Does every CSR have a weight-2 action such that all other conical actions commute with it? 
%\end{que}

Passing to the realm of Nakajima quiver varieties %(Chapter \ref{quivvar})% 
and Slodowy varieties, %(Chapter \ref{Springer}), 
there exist canonical weight-2 actions $\phi$, and explicit subgroups $Symp_{\phi}(\M,\omC)'$ of $Symp_{\phi}(\M,\omC).$ Thus, the calculation of the convex subsets $\Theta_\phi$ of the lattices $L_{\phi}(\M,\omC)'$ that correspond to $Z(Symp_{\phi}(\M,\omC)')^{\circ}$ becomes rather feasible, which we do in \cite{FZ2}, for type A. %Section \ref{NakajimaAreOnlyEvenConicalTwistedFullActions} and Section \ref{TwistedKazhdanActionsSection}, respectively. 
In addition, we believe that, in the case of Nakajima quiver varieties of type A, these subgroups $Symp_{\phi}(\M,\omC)'$ are actually equal to the whole group $Symp_{\phi}(\M,\omC)^{\circ}.$ 

\section{Symplectic Topology of minimal components}\label{SymplTopMinCompns} %\FZ $H^1=0$ and aHK issues. SPLIT CSR or Not?}
In this section we will observe the symplectic-topological viewpoint of the previous section. Recall that we have proved that any SHS $(\M,\Fi,\om_\C)$ has a canonical family of isomorphic Liouville structures with symplectic forms $\om_{J,K},$ such that the core $\L$ is its Liouville skeleton. 
Here $\om_J:=\R e(\om_\C)$ and $\om_K:=\mathbb{I} m(\om_\C)$ and $\om_{J,K}$ is the set of their non-zero linear combinations.
Moreover, in the case that $\M$ is a weight-1 SHS, any smooth component of the core is a smooth Lagrangian submanifold of $(\M,\om_{J,K}).$ In particular, the minimal components are non-isotopic Lagrangian submanifolds of $\M.$ 
In this section we show that they are also \textit{exact} Lagrangian submanifolds, therefore %a non-trivial 
objects of the compact Fukaya category $\mathcal{F}(\M,\om_{J,K}),$ %of $\M,$ 
and we compute their mutual Floer cohomology groups as graded vector spaces.

\begin{rmk}\label{OnCoefficients} \textbf{On coefficients.} As the symplectic structure we discuss here is exact, there is no need to use the Novikov field for the coefficients. Thus, we will use coefficients for Lagrangian Floer cohomologies over a field $\K.$ 
Now, as the Lagrangians that we consider in general may not be spin (take $\C P^2$ in $T^* \C P^2$ for instance), we would have to work over characteristic-2 fields in general. 
However, in some particular cases  of SHS manifolds one may bypass this issue and use arbitrary coefficients field $\K$. 
Namely, given a CSR $(\M,\om_\C)$ there is a notion of a \textit{theta characteristic} on it,\footnote{We use the name coined by Maulik-Okounkov \cite{MO12}, where they considered the case of quiver varieties.} 
a class $\kappa \in H^2(\M,\Z/2)$ with a property that $\kappa|L =w_2(L)$ for any any smooth $\om_\C$-Lagrangian subvariety $L \subset \M.$
Given a cotangent bundle $T^*X,$ this condition on the zero section uniquely determines $\kappa.$ 
For quiver varieties it was constructed in \cite[2.2.8]{MO12}, where they used it for twisting the signs in the formulas for quantum multiplication, \cite[1.2.5]{MO12}.
More generaly, for spaces belonging to the Higgs branch of gauge theories (Section \ref{HiggsBranch}), %in particular quiver varieties and hypertoric varieties, 
it was constructed in \cite[Sec. 2.4]{BLPW16}.
For general CSRs, the existence of such class $\kappa$ is not known. However, the aforementioned cases give almost all known examples, hence let assume that we have one for a moment.
Following Seidel \cite[Rmk. 12.1 and Rmk. 11.16]{SeiBook}, for defining Lagrangian Floer cohomologies in a symplectic manifold $M$ over arbitrary field $\K,$ 
one can consider twisted $\text{Pin}$-structures over Lagrangians instead of ordinary \text{Pin}-structures, where twist is given by an arbitrary class $w\in H^2(M,\Z/2).$ 
Given a Lagrangian $L\subset M,$ the obstruction for existence of such structure is given by $w_2(L)+w|L.$ 
Hence, given a theta characteristic $\kappa$ on a CSR $\M,$ and twisting the Pin structures by $-\kappa,$ smooth $\om_\C$-Lagrangian subvarieties have well defined Lagrangian Floer cohomologies 
over arbitrary field $\K.$ In particular, this applies to all Lagrangians that we are going to consider in this and in the following section. 
Moreover, the twisting will not affect the features of Lagrangian Floer cohomologies that we are going to need.
Having this in mind, throughout this and the following section, given an SHS $\M$: %\ref{MinLagrasAndSH}, given a SHS $\M:$ 
\begin{enumerate}
	\item If it is a cotangent bundle $T^*X$ of a projective variety $X,$ we will use an arbitrary field $\K.$
	\item If it belongs to the Higgs branch of gauge theories, in particular if it is a quiver variety or a hypertoric variety, we will use an arbitrary field $\K.$
	\item Otherwise, we will use a field with $\text{char}(\K) = 2.$
\end{enumerate}
\end{rmk}

Now we will show that any smooth core component of a weight-1 SHS is a Bohr--Sommerfeld Lagrangian, by the argument analogous to the main result of \cite{BGL15}, where the authors consider the case of Moduli spaces of Higgs bundles. We first introduce the notion of Bohr--Sommerfeld Lagrangian, which is slightly stronger than the definition\footnote{Which requires that the integration of the symplectic form $\int_\om: H_2(M,L,\Z) \fun \R$ lands in integers. In our case this integration gives the zero map.} that can be found elsewhere.

\begin{de}
	Given an exact symplectic manifold $(M,\om=d \th)$, a {Lagrangian submanifold} is a half-dimensional submanifold $i:L\hookrightarrow M$ 
	%such that $i^*\om=0$.
	%In other words, the pull back of the primitive 1-form $\th$ is 
	satisfying $i^*\om=0,$ or equivalently $d(i^* \th)=0$. In particular, when the form $i^* \th=df$ is exact, the Lagrangian is called \textbf{exact}.
	When $i^* \th$ vanishes, the Lagrangian is called \textbf{Bohr--Sommerfeld}.
\end{de}

\begin{lm}\label{SmoothIsExact} Given a weight-1 SHS $\M,$ with a Liouville 1-form from Proposition \ref{canonicalLiouville}, its compact Bohr--Sommerfeld Lagrangians are exactly smooth components of the core $\L.$ In particular, smooth core components are exact Lagrangian submanifolds of $(\M,\om_{J,K}).$

\end{lm}
\begin{proof}
Recall that the Liouville 1-forms in Proposition \ref{canonicalLiouville} are defined by $$\th_{J,K}=i_Z \om_{J,K},$$ where $Z$ is the vector field of the $\R_+$-part of the $\C^*$-action.
Given a smooth core component $L,$ by Corollary \ref{CoreDecompositonWeight1}, it is equal to $\ol{\L_i},$ where $\L_i \fun \F_i$ is the $(t\fun \infty)$-attracting set of a fixed component $\F_i$
The open dense subset $\L_i \subset L$ is $\C^*$-invariant, which in particular yields $Z\in T\L_i$, hence the form 
$\th_{J,K}(\xi)=\om_{J,K}(Z,\xi)$ must vanish for any $\xi \in T\L_i,$ as $\L_i$ is Lagrangian. As vanishing is a closed condition, when $\overline{\L_i}$ is smooth we deduce that $\th_{J,K}$ vanishes on it.

On the other hand, given a Bohr--Sommerfeld Lagrangian $L,$ the vanishing of any form $\th_{J,K}=\om_{J,K}(Z,\cdot)$ on it implies that $Z\in TL^{\perp_{\om_{J,K}}}=TL,$
hence $L$ is $Z$-invariant. Thus, if there was a point $x\in L$ that is not inside the core, its $Z$-flow in the direction $(t\fun +\infty)$ would go to infinity, preventing $L$ from being compact, contradiction.
%So $\overline{\L_\a}$ is a Bohr--Sommerfeld Lagrangian submanifold, and in particular it is exact. 
%Exactness implies that its Lagrangian Floer cohomology is isomorphic to the ordinary cohomology, as claimed.
\end{proof}

\begin{cor}\label{MinimalComponentsExactLagrTheorem}
	Given a weight-1 SHS $(\M,\Fi),$ there are at least $N$ smooth closed exact non-isotopic 
	Lagrangian submanifolds of $(\M,\om_{J,K}),$ where $N$ is the maximal number of commuting weight-1 conical actions.
%	=\max \{N_1,N_2\}\geq 1$ and
%	\begin{enumerate}
%		\item $N_1=$ the maximal number of commuting weight-1 conical actions.
%		\item $N_2=$ the number of HK conical actions (if $\M$ is also a HKCSR).
%	\end{enumerate}
\end{cor}
\begin{proof}
By Theorem \ref{MinimalComponentsTheorem} and Lemma \ref{SmoothIsExact}. The fact that they are non-isotopic follows from Lemma \ref{LemaNonIsotopic}.
\end{proof}

%{\FZ DEGREE SHIFT. WRITE THAT THERE IS A Z-grading, due to Seidel's paper symplectically knotted, that the intersection of them is kahler thus even dimsnional so the indeces are not in $\Z/2$ but actually in $\Z$ and beforehand, since $H^1(L_1)=H^1(L_2)=0$ there is a grading in the first place! -> oh crap maybe not, as we know that in Q coefficeints only. As Alex says, $H^1$ has free part of $H_1$ and torsion part of $H_0.$ SO ALL GOOD}

In the next theorem we will compute Lagrangian Floer cohomologies of minimal components, as graded vector spaces. Recall that the Lagrangian Floer cohomology $HF(L_1,L_2)$ of two Lagrangians $L_1$ and $L_2$ in a symplectic manifold $(M,\om)$ is a homology of the chain complex $CF(L_1,L_2)$ whose generators are intersections of two Lagrangians, and the differential counts the pseudoholomorphic strips with ends on these intersections. In particular, when Lagrangian $L$ is exact, it recovers the ordinary cohomology $HF(L,L)\iso H(L).$ %For more extensive discourse on this subject we refer to \cite{FOOO}.
Following Seidel \cite[Sec. 2]{Sei00}, when the ambient manifold $(M,\om)$ satisfies $2c_1(M,\om)=0,$ 
we can make a notion of graded Lagrangians $L$ %satisfying %$H^1(L,\Z)=0,$ 
and for each such Lagrangian there is a $\Z$-worth of choices of its gradings, $\widetilde{L}[k],k \in \Z.$

In particular, by \cite[Ex. 2.9]{Sei00}, having a $J$-complex (where $J$ is an almost complex structure compatible with $\om$) volume form $\Om,$ these gradings can be defined in the following way.
Denote by $\mathcal{L}$ the Lagrangian Grassmannian of $M,$ i.e. the bundle over $M$ whose fibre $\mathcal{L}_x$ above $x\in M$  consists of Lagrangian subspaces of $(T_x M,\om).$
Now, define the \textbf{squared phase map} 
$$\text{det}_\Om^2: \mathcal{L} \fun S^1, \ \text{det}_\Om^2(\Lambda)={\Om(\xi_1,\dots,\xi_n)^2 \over |\Om(\xi_1,\dots,\xi_n)|^2},$$
for any\footnote{The pull-back formula for the top-degree forms shows that this does not depend on the choice of basis.} basis $\xi_1,\dots,\xi_n$ of $\Lambda.$ 
Consider the pull-back of the universal cover $\R\fun S^1$ via squared phase map
\begin{equation} \label{eq:theta-covering}
	\widetilde{\mathcal{L}}= \{ (\Lambda,t) \in \mathcal{L} \times \R \mid \text{det}_\Om^2(\Lambda) = e^{2\pi i t}\}.
\end{equation}
The \textbf{grading} of $L$ is a lift of the natural map $L\fun \mathcal{L}$ to a map $\widetilde{L}: L \fun \widetilde{\mathcal{L}}.$ Thus, for every $x\in L$ we have $\widetilde{L}(x)= (T_x L, t)$ and  $\text{det}_\Om^2(T_x L) = e^{2\pi i t}.$
Obviously, if $\widetilde{L}$ is a grading, so is 
$$\widetilde{L}[k]:L \fun \widetilde{\mathcal{L}}, \ \ \widetilde{L}[k](x):=(T_x L, \pi_\R (\widetilde{L}(x)) - k),$$ 
for every $k\in\Z,$ and every grading is obtained in this way. Here, $\pi_\R: \widetilde{\mathcal{L}} \fun \R$ is the projection. 
Thus, we get an $\Z$-worth of gradings for all Lagrangians that admit a single one.
Moreover, this notion depends on the choice of the homotopy class of $\Om$ in the space of $J$-complex volume forms. 
In particular, when $H^1(M,\Z)=0,$ these are all homotopic and the notion of a grading becomes canonical.

In addition, %as noticed at the end of \cite[Ex. 2.9]{Sei00}, 
Lagrangians $L$ satisfying $\im(\Om)|_L=0,$ usually called \textbf{special} Lagrangians\footnote{In the case where $J$ is integrable and $\Omega$ is holomorphic, 
but we will use this notion in the possibly non-integrable case as well.}
have a canonical choice of grading, defined by
\begin{equation}\label{GradingForSpecLagrs}
	\widetilde{L}(x):=(T_x L, 0), \ \forall x \in L,
\end{equation}
which is well-defined as $\det_\Om^2(T_x L)\equiv 1.$

Finally, given two graded Lagrangians $\widetilde{L}_1,\widetilde{L}_2,$ one can canonically define the $\Z$-grading on the Floer cohomology $HF^*(\widetilde{L}_1,\widetilde{L}_2).$ Thus, as the grading $\widetilde{L}_i$ of each Lagrangian ${L}_i$ is defined up to a $\Z$-shift, so is $HF^*(\widetilde{L}_1,\widetilde{L}_2),$ as $HF^*(\widetilde{L}_1[k],\widetilde{L}_2[l])=HF^{*-k+l}(\widetilde{L}_1,\widetilde{L}_2).$\footnote{Where by $\widetilde{L}_i[k]$ we denote the grading shifted by $k\in\Z$ from the one previously chosen.} When $L_1=L_2=L,$ these shifts cancel out 
and $HF^*(\widetilde{L}_1,\widetilde{L}_1)$ is canonically graded.\\
%These choices yield the natural $\Z$-gradings of corresponding Lagrangian Floer cohomologies $HF^*(\tilde{L_1}[k],\tilde{L_2}[l]),$ which then depend on choices of gradings of starting Lagrangians ($HF^*(\widetilde{L}_1[k],\widetilde{L}_2[l])=HF^{*-k+l}(\widetilde{L}_1,\widetilde{L}_2)$) Nevertheless, choosing any grading $\tilde{L}[k]$ of a Lagrangian $L$ yields the graded isomorphism $HF^*(\tilde{L}[k],\tilde{L}[k])\iso H^*(L).$

Let us apply this in the setup of SHS manifolds. As proved in Lemma \ref{LemmaCSRsAreAlmostHK} any SHS manifold $(\M,\om_\C)$ is an almost \HK manifold $(\M,g,I,J,K)$ such 
that $\om_\C=\om_J+i \om_K$ and $\om_S(\cdot,\cdot):=-g(\cdot,S\cdot),$ for $S=I,J,K.$
Hence, for any almost complex structure $\Theta \in \{aJ+bK \mid a^2+b^2=1\}$ it has a natural $\Theta$-complex volume form
\begin{equation}\label{cplxVolForm}
	\Om_{\Theta}:= {1 \over {(n/2)!}}(\om_I - i \om_{I \Theta})^{n/2},
\end{equation}
which makes a notion of graded $\om_{\Theta}$-Lagrangians. 
%As $H^1(\M,\Z)=0$ (Corollary \ref{CoreIsHtpyEquiv}), this notion is canonical. 
In addition, minimal components have canonical choice of gradings, due to the following lemma:

\begin{lm} \label{LemmaSmoothCoreSpecialLagrs}
	Given a weight-1 SHS $(\M,\Fi),$ choosing a compatible almost \HK structure $(g,I,J,K)$, its smooth core components are special Lagrangians with respect to all $\Om_{\Theta}.$
	In particular, they are canonically graded.
%	Thus, each smooth core component $L$ satisfying $H^1(L,\Z)=0$ has a canonical grading. 
	%In particular, this holds for minimal components.
\end{lm}
\begin{proof}
	By Lemma \ref{CoreIsIsotropic}, smooth core components are holomorphic Lagrangian submanifolds with respect to $\om_\C,$ thus they are special with respect to all $\Om_{\Theta},$ by the
	argument written in \cite[Lem. 3.3]{SoVe19}.\footnote{Which goes back to \cite[p. 154]{HaLa82}.}
	Hence, their canonical grading is given by \eqref{GradingForSpecLagrs}. 
%{\FZ Such Lagrangians whose first cohomology vanishes have canonical grading \eqref{GradingForSpecLagrs}.
%	Let us show that a minimal component $\F_{\Fi}$ satisfies $H^1(\F_{\Fi},\Z)=0.$ 
%	Denoting $\{\F_\a\}_{\a\in A}$ the set of connected components of the fixed locus $\M^{\Fi},$ by Lemma \ref{LemmaMomentMapMorseBott} we have
%	$$b_k(M) = \sum_{\a \in A} b_{k-\mu(F_\a)}(\F_\a),$$
%	where $\mu(\F_\a)$ are Morse-Bott indices of the moment map of the $S^1$-part of $\varphi$, and the index of $\Fmin$ is equal to zero. Since $b_1(\M)=0$ by Corollary \ref{CoreIsHtpyEquiv}, that implies that $b_1(\Fmin)=0,$ thus $H^1(\Fmin,\R)=0.$ Moreover, since $H^1(\Fmin,\Z)$ has no torsion,\footnote{Due to universal coefficients theorem.} it vanishes as well.
%}
\end{proof}

Now we will compute the mutual Floer cohomologies of smooth core components with respect to these canonical gradings.

\begin{thm}\label{LagrFloerMinComps}
Given a weight-1 SHS manifold $\M,$ its smooth core components are exact Lagrangian submanifolds of $(\M,\om_{J,K}),$ hence $HF^*(L,L)\iso H^*(L),$ %,\Z/2)$ 
for each smooth core component $L.$ 
Given a minimal $\F_{\Fi}$ and a smooth core component $L,$ their intersection is clean and connected, hence
\begin{equation}\label{LagrFLoerOfTwoMinima}
	HF^*(\widetilde{\F}_{\Fi},\widetilde{L}) \iso H^{*-\mu}({\F}_{\Fi} \cap  L)%, \Z/2),%[dim_\L-\dim ], {\FZ FINISH}
\end{equation}
%Each pair $\Fminn,\Fminnn$ of minimal components whose actions $\Fi^1$ and $\Fi^2$ commute intersect cleanly, hence we have
%\begin{equation}\label{LagrFLoerOfTwoMinima}
%	HF^*(\widetilde{\F}_{\Fi^1},\widetilde{\F}_{\Fi^2}) \iso H^{*-\mu}(\Fminn \cap \Fminnn)%, \Z/2),%[dim_\L-\dim ], {\FZ FINISH}
%\end{equation}
for certain shift $\mu.$ %=\mu(\Fi^1,\Fi^2).$
Moreover, when the $S^1$-part of $\Fi$
%of $\Fi^1$ and $\Fi^2$ 
is isometric with respect to a compatible almost \HK metric we have that
%\begin{equation*}\label{gradingshiftcleanLagrMinimal}
%\mu=\codim_\C(\Fminn \cap \Fminnn \subset \Fminn),
$\mu=\codim_\C({\F}_{\Fi} \cap L \subset {\F}_{\Fi}),$
%\end{equation*}
hence the grading is symmetric, %$HF^*(\widetilde{\F}_{\Fi^1},\widetilde{\F}_{\Fi^2})\iso HF^*(\widetilde{\F}_{\Fi^2},\widetilde{\F}_{\Fi^1})$ 
$$HF^*(\widetilde{\F}_{\Fi},\widetilde{L})\iso HF^*(\widetilde{L},\widetilde{\F}_{\Fi})$$ as graded vector spaces.
In particular, when $\M$ is an SHK manifold, we have an isomorphism of graded vector spaces
\begin{equation}\label{SHKcaseFloerAlgebra}
	\bigoplus_{L_i,L_j \text{minimal}} HF^*(L_i,L_j) \iso \bigoplus_{L_i,L_j \text{minimal}} H^*(L_i \cap L_j)[-d_{i,j}],
\end{equation}
where $d_{i,j}=codim_\C ((L_i \cap L_j) \subset L_i),$ and the summands correspond to each other.
%in particular, different minimal components are non-Hamiltonian isotopic.\footnote{In fact, they are not even isotopic due to Corollary \ref{MinimalComponentsExactLagrTheorem}.}
\end{thm}
\begin{proof} By Lemma \ref{SmoothIsExact}, smooth core components are exact Lagrangian submanifolds, hence $HF^*(L,L)\iso H^*(L)$ for each smooth component $L.$ 
%However, here we notice that one can prove this in the other way, using Lemma \ref{LemmaSmoothCoreSpecialLagrs}. It tells us that each minimal component $\Fmin$ satisfies $H^1(\Fmin,\R)=0,$ hence $\Fmin$ is exact with respect to any primitive form of $\om_{J,K}.$
%By the same vanishing, the Floer cohomology $HF^*(\Fmin,\Fmin)$ is $\Z$-graded, and the exactness of $\Fmin$ implies the graded isomorphism $HF^*(\Fmin,\Fmin)\iso H^*(\Fmin,\Z/2)$, as claimed. 
Consider now a minimal $\F_{\Fi}$ and a smooth core component $L.$
%two different commuting weight-1 conical actions $\Fi^1, \Fi^2$ and their corresponding minimal components $\Fminn$, $\Fminnn$. 
First, we prove the following lemma:

\begin{lm} Intersection of $\F_{\Fi}$ and $L$ is clean and connected.
\end{lm}
\begin{proof}
Firstly, if the intersection of $\F_{\Fi}$ and $L$ is empty, we are done. Thus, let us assume that $S:=\F_{\Fi} \cap L \neq \emptyset.$
$L$ is a subvariety of $\M,$\footnote{Being an irreducible component of the core which is a subvariety of $\M.$} hence an $I$-holomorphic submanifold of $\M.$ 
Moreover, it is a \KH submanifold of $\M,$ with respect to a \KH structure from Lemma \ref{thereisaNonExactstructure}. 
Therefore, as the $\C^*$-action $\Fi$ restricts to $L,$ so does its $S^1$-part, and the restriction of the moment map $H:\M \fun \R$ (from the same lemma) 
on $L$ is a moment map of the restriction $S^1\dejstvo L.$ 
As $\F_{\Fi}$ is the absolute minimum of $H$ on $\M,$ so is $S$ on $L.$ 
Hence, as $L$ is connected, so is its minimum locus $S,$ by the same Betti number argument as in the proof of Proposition \ref{minimalcomp}, now applied to $L.$ 
(the Betti number relation for a \KH $S^1$-action on a compact \KH manifold is an older result, which goes back to Frankel, \cite[Sec. 4]{Frankel59}).
Moreover, it is smooth, being a fixed component of the smooth $S^1$-action on smooth $L.$ Indeed, it there was a strictly bigger $S^1$-fixed component containing $S,$ 
it would have the same, hence minimal value of $H,$ and then would be contained in $\F_{\Fi}.$ But that contradicts the definition of $S=\F_{\Fi} \cap L.$
Having $S$ to be smooth and connected, it is left to prove that it is a clean intersection,
$$\forall x \in S, \ T_x S= T_x \F_{\Fi} \cap T_x L.$$
The direction `$\subseteq$' is obvious as $S$ is a submanifold of both $\F_{\Fi}$ and $L$ so the inclusion between tangent spaces follows.
To prove the direction `$\supseteq$' notice that $v\in T_x \F_{\Fi}$ is equivalent to $(\Fi_t)_* v =v,$ for all $t\in \C^*$ and that precisely 
characterises vectors in $T_x L$ that belong to the fixed locus $T_x S.$
\end{proof}

As the two exact Lagrangians $\F_{\Fi}$ and $L$ intersect cleanly, and their intersection $S$ is connected, we have an (ungraded) isomorphism $HF(\F_{\Fi},L)\iso H(S).$
This is a well-known result that goes back to Po\'zniak \cite[Thm. 3.4.11]{Po99}. 
Moreover, picking a compatible almost \HK metric $g,$ by Lemma \ref{LemmaSmoothCoreSpecialLagrs} these Lagrangians can be canonically graded, 
$\widetilde{\F}_{\Fi}(x)=(T_x \F_\Fi,0),\   \widetilde{L}(x)=(T_x L,0).$ 
%$\widetilde{\F}_{\Fi^1},\widetilde{\F}_{\Fi^2}.$
Then, following \cite[Prop. 5.18]{KhSei02} we have a graded isomorphism %(see e.g. ):
$$HF^*(\widetilde{\F}_{\Fi},\widetilde{L})\iso H^{*-\mu}(S),$$
where $\mu=\mu(\lambda_1,\lambda_2)+{1\over 2}(\dim_\C \M-\dim_\R S ),$ %$\mu$ %=\mu_{\widetilde{\F}_{\Fi^1},\widetilde{\F}_{\Fi^2}}$ is %the shift which \textbf{depends} on the gradings,
%\begin{equation}\label{degreeshiftFloerCohCleanIntersect}
%\mu=\mu(\lambda_1,\lambda_2)+{1\over 2}(\dim_\C \M-\dim_\R \F_{\Fi^{12}} ),
%\end{equation}
and $\mu(\lambda_1,\lambda_2)$ is the Maslov index for paths of Lagrangians\footnote{Defined in \cite[Sec. 3]{RS93}, where reader can recall its properties which we are going to use below.}
$$\lambda_1,\lambda_2:[0,1]\fun \mathcal{L}_x, \ \lambda_1(0)=T_x \F_\Fi, \ \lambda_1(1)=T_x L,\ \lambda_2(t) \equiv T_x L.$$ 
Here, the path $\lambda_1$ is chosen to be compatible with the choice of gradings $\widetilde{\F}_{\Fi},\widetilde{L},$
i.e. it is the projection of a path 
$\widetilde{\lambda}_1:[0,1] \fun \widetilde{\mathcal{L}}_x$ satisfying  $\widetilde{\lambda}_1(0)=\widetilde{\F}_{\Fi}(x), \ \widetilde{\lambda}_1(1)=\widetilde{L}(x).$  
It is unique, up to homotopy that fixes endpoints.
Therefore, the index $\mu(\lambda_1,\lambda_2)$ neither depends on the choice of $\lambda_1,$ or on the chosen point $x\in S=\F_{\Fi} \cap L,$ as this intersection is clean and connected.

Assuming that the $S^1$-part of $\Fi$ is isometric with respect to a compatible almost \HK metric $g,$
we can compute the shift $\mu$ explicitly. 
For brevity, we will use the form $\om_J,$ thus complex volume form $\Om_J,$ noticing that the same proof goes for any
$\om_\Theta,\Om_\Theta,$ where $\Theta \in \{aJ+bK \mid a^2+b^2=1\}.$
Denote $$V:= T_x \M, \ N:= T_x S, \ L_1:=T_x \F_{\Fi}, \ L_2:=T_x L, \text{ and } U_i := N^{\perp_g} \subset L_i.$$ 
%Here, $g(\cdot, \cdot)=\om_J(\cdot,J \cdot)$ is the compatible Riemannian metric.
The following lemma is crucial, in it we use that one of the Lagrangians is minimal and the metric we use is almost \HKL.
\begin{lm}\label{JSendsL1toL2}
	$JU_1=U_2.$
\end{lm}
\begin{proof}
	As $L_i \subset V$ are $\om_J$-Lagrangian subspaces, for all $v,w\in L_i$ we have $g(v,J w)=-\om_J(v,w)=0.$	Thus, $JL_i = L_i^{\perp_g},$ and $V= L_i \oplus J L_i $ are $g$-orthogonal decompositions, for $i\in\{1,2\}$.
	Thus, as $L_1=N \oplus U_1$ is a $g$-orthogonal decomposition, so 
	is $V= N \oplus U_1 \oplus JN \oplus J U_1$ (here we are using that $J$ is $g$-orthogonal, $g(J\cdot,J\cdot)=g(\cdot,\cdot)$).
	Thus, as $\dim JU_1 = \dim U_1 = \dim U_2,$ in order to show $JU_1=U_2,$ it is enough to show that $(N \oplus U_1 \oplus JN)  \perp_g  U_2.$ By definition, $U_2\perp_g N,$ 
	and, as $JN\subset J L_2 = L_2^{\perp_g}$ is orthogonal to $L_2\supset U_2,$ it is orthogonal to $U_2$ as well.
	Thus, it is left to show that 
	\begin{equation}\label{orthogonalityofcomplements}
			U_2 \perp_g U_1.
	\end{equation}
	To do this, we will use that $\F_{\Fi^{1}}$ is a minimal component, hence a connected component of the fixed locus of a $\C^*$-action $\Fi^{1}.$
	Then, like in the proof of Proposition \ref{DifntActionsDiffntMinComp},
	we have the $g$-orthogonal 
	%{\FZ where $g$ is the $S^1$ invariant metric $g=\lambda \tilde{g}.$ But it is also $\tilde{g}$-orthogonal, as the scalar function $\lambda$ does not change orthogonality,
	%$g(v,w)=0 \iff \tilde{g}(v,w)=0$} 
	weight composition with respect to $\Fi^{1}:$
	$$V= T_x \M =H_0 \oplus H_1.$$
	We have that $H_0=T_x \F_{\Fi^{1}}= L_1 \supset U_1,$
	thus in order to show \eqref{orthogonalityofcomplements}, it is enough to show that $U_2 \subset H_1.$
	This follows from considering the restriction of the $\C^*$-action $\Fi$ to $L.$
	Then, we again have the weight decomposition
	$T_x L= H_0' \oplus H_1',$ where $H_0'=T_x S= N,$ and $H_1'$ is its orthogonal complement,  $U_2.$
	Thus, $U_2$ consists of weight-1 vectors, and we are done.
\end{proof}
Now, pick a $g$-orthogonal basis $(u_1,\dots,u_k)$ of $N$ and extend it to a $g$-orthogonal basis $(u_1,\dots,u_n)$ of $L_1.$ Thus, $U_1=\la u_{k+1},\dots,u_n\ra.$
Denoting by $v_i:=J u_i,$ previous lemma gives $L_2=\la u_1,\dots,u_k, v_{k+1},\dots,v_n \ra.$
Now, consider a path of half-dimensional subspaces in $V$:
\begin{equation}\label{pathOfSpecLagrs}
	\gamma_1(t)=\la u_1,\dots,u_k, e^{ i{\pi \over 2}t} u_{k+1}, \dots, e^{ i{\pi \over 2}t} u_{{n+k} \over 2}, e^{ -i{\pi \over 2}t} u_{{{n+k} \over 2} +1},\dots, e^{ -i{\pi \over 2}t} u_{n}\ra,
\end{equation}
where $e^{i \th} u_j := cos (\th) u_j + sin (\th) J u_j.$ 
The number ${n+k \over 2}$ is an integer, as both $n$ and $k$ are even.\footnote{Being real dimensions of $I-$\KH manifolds $\F_{\Fi^{1}}$ and $\F_{\Fi^{12}}.$}
From the $g$-orthogonality of $(u_i)_{i=1}^n,$ it follows that $\gamma_1 \subset \mathcal{L}_x.$ Clearly, $\gamma_1(0)=L_1$ and $\gamma_1(1)=L_2.$
The next lemma shows the importance of this path:
\begin{lm} $\gamma_1$ is an eligible choice for a path $\lambda_1.$ 
\end{lm}
\begin{proof}
	It is enough to show that 
	\begin{equation}\label{squarephaseisone}
		det_{\Om_J}^2(\gamma_1(t))\equiv 1,
	\end{equation}
	as then $\widetilde{\lambda}_1:=(\gamma_1(t),0)$ is a path in $\widetilde{\mathcal{L}}_x$ satisfying 
	$\widetilde{\lambda}_1(0)=\widetilde{\F}_{\Fi}, \ \widetilde{\lambda}_1(1)=\widetilde{L}$ 
	and $\gamma_1$ is its projection to $\mathcal{L},$ 
	hence an eligible choice for $\lambda_1.$
	In order to show \eqref{squarephaseisone}, consider a $J$-complex volume form 
	$\Om:= (u_1^*+i v_1^*)\wedge (u_2^*+i v_2^*)\wedge \dots \wedge (u_n^*+i v_n^*)$ on $V.$
	Denoting the basis of $\gamma_1(t)$ in \eqref{pathOfSpecLagrs} by $(u_j'(t))_{j=1}^n,$ it is immediate that
	$$\Om(u_1'(t),\dots,u_n'(t))=(e^{ i{\pi \over 2}t})^{{n-k} \over 2} (e^{- i{\pi \over 2}t})^{{n-k} \over 2} \Om(u_1,\dots,u_n)=\Om(u_1,\dots,u_n)=1,$$
	thus, $det_{\Om}^2(\gamma_1(t))\equiv 1.$
	
	On the other hand, $\Om$ and $\Om_J$ are both complex volume forms on $V,$ hence $\Om=\eta \Om_J$ for some $\eta\in\C^*.$
	Then, $det_{\Om}^2={\eta^2 \over |\eta|^2 } det_{\Om_J}^2.$
	As $\F_{\Fi}$ is special, we have $det_{\Om_J}^2(T_x \F_{\Fi})= det_{\Om_J}^2(L_1)=1.$ 
	On the other hand, $det_{\Om}^2(L_1)=det_{\Om}^2(\gamma_1(0))=1,$ thus ${\eta^2 \over |\eta|^2}=1$ and $det_{\Om}^2=det_{\Om_J}^2,$ which yields \eqref{squarephaseisone}.
\end{proof}

Now we can compute the Maslov index $\mu(\lambda_1,\lambda_2),$ which, %together with \eqref{degreeshiftFloerCohCleanIntersect}, 
finally gives the claimed shift $\mu ={1\over 2}(\dim_\C \M-\dim_\R S)=\codim_\C(\F_\Fi \cap L \subset \F_\Fi).$ 

\begin{lm}
	 $\mu(\lambda_1,\lambda_2)=0.$
\end{lm}
\begin{proof}

By definition, Maslov index of paths of Lagrangians in a symplectic vector space $V$ is computed using a symplectomorphism $\phi:(V,\om_J) \fun (\R^{2n},\om_{std}),$
$\mu(\lambda_1,\lambda_2):=\mu(\phi_*\lambda_1,\phi_*\lambda_2).$ Here, $\om_{std}$ is the standard symplectic structure on $\R^{2n}.$
By the naturality of the Maslov index for paths, this does not depend on the choice of $\phi.$
We will define the most natural such symplectomorphism:
$$\phi: V \fun \R^{2n}, \ \phi(u_i)=e_i,\ \phi(v_i)=f_i,$$
where $(e_i)_{i=1}^n,(f_i)_{i=1}^n$ are standard bases of $\R^n\times 0$ and $0 \times \R^n,$ respectively.
Denote $\nu_i:=\phi_*\lambda_i,$ for $i\in\{1,2\},$ and $K_i:=\phi_*L_i.$
Thus, $K_1=\la e_1,\dots,e_n \ra$ and $K_2=\la e_1,\dots,e_k,f_{k+1},\dots,f_n \ra.$
It is immediate that $\nu_1(t)=\psi(t)K_1=\psi(1-t)K_2,$ where
$$\psi:[0,1] \fun U(n)\subset Sp(2n), \ \psi(t)= diag (\underbrace{1,\dots,1}_k, \underbrace{e^{ i{\pi \over 2}t}, \dots, e^{ i {\pi \over 2}t}}_{{1\over 2} (n-k)},\underbrace{e^{ -i{\pi \over 2}t}, \dots, e^{ -i {\pi \over 2}t}}_{{1\over 2} (n-k)}).$$
Here, the complex coordinates used on $\R^{2n}\iso \C^n$ are standard, i.e. on the $i$-th place is the span $\la e_i,f_i \ra.$ 
%satisfies $\psi(0)L_1=L_1, \psi(1)L_1=L_2,$ and $\psi(t)L_1=\psi(1-t)L_2.$ The number ${{1\over 2} (n-k)}$ is an integer, as both $n$ and $k$ are even.\footnote{Being real dimensions of $I-$\KH manifolds $\F_{\Fi^{1}}$ and $\F_{\Fi^{12}}.$}
%Denote two paths $$\gamma_1(t):=\psi(t)L_1, \ \gamma_2(t): \equiv L_2$$ in the Lagrangian Grassmannian $\mathcal{L}:=\mathcal{L}(\R^{2n},\om_{std}).$
Using the elementary properties (given in \cite{RS93}) of Maslov index for pairs of Lagrangian paths, and for its version $\mu(A):=\mu(A(0\times \R^n),0\times \R^n)$ for a path $A=A(t)$ of symplectic matrices, 
we have a sequence of equalities:
%Notice that \
%\begin{equation}\label{MaslovIndexCalculation}
\begin{align*}
	\mu(\lambda_1,\lambda_2)=\mu(\nu_1,\nu_2)&=\mu(\psi(1-t)K_2,K_2) = \mu(\psi(1-t)\Fi(0 \times \R^n),\Fi(0 \times \R^n))\\
	&=\mu(\Fi \psi(1-t) (0 \times \R^n),\Fi(0 \times \R^n)) =\mu(\psi(1-t) 0\times \R^n, 0 \times \R^n )\\ 
	&= \mu(\psi(1-t))= -\mu(\psi(t)) \\
	&= -k \mu(1) - {1\over 2}(n-k)\mu((e^{i{\pi \over 2} t})_{t\in[0,1]})-{1\over 2}(n-k)\mu((e^{-i{\pi \over 2} t})_{t\in[0,1]})\\ 
	&= -{1\over 2}(n-k) ({1 \over 2}) -{1\over 2}(n-k) (-{1 \over 2})=0,
\end{align*}
%\end{equation}
where 
$\Fi:\R^{2n}\fun \R^{2n}, \ \Fi=diag(\underbrace{i,\dots,i}_k,\underbrace{1,\dots,1}_{n-k})$
is a symplectomorphism which satisfies $\Fi(0 \times \R^n)=(\R^k \times {0}) \times ({0} \times \R^{n-k})=L_2,$ and clearly commutes with $\psi(t).$
The Maslov index $\mu((e^{\pm i{\pi \over 2} t})_{t\in[0,1]})=\mu((e^{\pm i{\pi \over 2} t})_{t\in[0,{1 \over 2}]})+\mu((e^{\pm i{\pi \over 2} t})_{t\in[{1\over 2},1]}) 
=\mu((e^{\pm i{\pi \over 2} t})_{t\in[0,{1 \over 2}]})+0 = {1\over 2} \sign (tg(\pm {\pi \over 4})) - {1 \over 2} \sign (tg(0))= \pm {1 \over 2}$ is computed using the catenation, zero and localisation properties from \cite[Thm. 2.3]{RS93}.
\end{proof}
In particular, given a SHK $\M,$ recall\footnote{Recall Definition \ref{DefinitionHKCSR}.} that for them we consider only actions whose $S^1$-parts are isometric with respect to the \HK metric on $\M.$
Thus, there is a uniform grading for \textit{all} minimal components which gives the graded Floer cohomologies as in \eqref{SHKcaseFloerAlgebra}.
\end{proof}

\begin{rmk} \textbf{Comparison with the literature.}\label{ComparisonFloerAlgebraToGRTLiterature}
We want to remark that the gradings of Floer cohomologies obtained in Theorem \ref{LagrFloerMinComps} compares well with the literature on this subject,
considering the examples when $\M=\M^{m,n}$ is the resolution of a two-row Slodowy variety $\M_0^{m,n}.$ 

On the side of Symplectic Topology, the symmetric gradings for Lagrangian Floer cohomologies were used in the work on Symplectic Khovanov Homology \cite[Sec. 4.4]{AbS19} in order to obtain the purity of the obtained Floer algebra (called Symplectic arc algebra). %needed in order its formality. 
The Lagrangians $L_\rho$ used there are labelled by crossingless matchings $\rho,$ and they are topologically same as the Lagrangians of the core of $\M^{n,n}$ 
(the manifold $\mathcal{Y}_n$ they use is the deformation of a singular Slodowy variety $\M_0^{n,n}$ whereas we consider its resolution; however, these are known to be symplectomorphic). 
Moreover, the only minimal Lagrangian in $\M^{n,n}$ is precisely the Lagrangian which corresponds to their Lagrangian $L_{\rho_{plait}}$ obtained from so-called ``plait'' matching. 
They choose the gradings of other Lagrangians $L_\rho$ such that obtained Floer cohomologies $HF^*(L_{\rho_{plait}},L_{\rho})$ have symmetric gradings, whereas we 
prove this symmetry by using the existing \HK structure and the canonical gradings for special Lagrangians.

On the side of Geometric Representation Theory, in the work of Stroppel--Webster \cite{SW12} they consider two-row Springer fibres $\B^{m,n},$ which are precisely the cores of resolutions of two-row Slodowy varieties $\M^{m,n}.$ Considering the intersections of components of $\B^{m,n},$ they shift \cite[Eqn. 3.2]{SW12} the obtained singular cohomologies in order to get the graded isomorphism to generalised arc algebras.\footnote{Algebras that generalise the ones that Khovanov used construct the Khovanov Homology, \cite{Kh00}.} Their shift is precisely same as ours, when passing from Floer cohomology to the cohomology of intersection by isomorphism \eqref{LagrFLoerOfTwoMinima}.
\end{rmk}

\begin{rmk} We have seen in this section that smooth components of the cores of SHS manifolds are exact Lagrangians, hence well-defined objects in the closed Fukaya category. From that viewpoint, of interest would also be the core components which are not necessarily smooth but immersed, as their Floer theory is well-defined as well \cite{AkJo10}.
	However, it is hard to get the information whether some singular core component is immersed in general. Even the simplest singular core components, in the examples of cores being Springer fibres 
	\cite[Sec. 5]{Va79}, 
	\cite[Prop. 2.1]{FrMe10}, are rather defined by a set of equations than an (immersion) map.
\end{rmk}
% The only general information that we have on singularities of core components is in 2-column Springer case of type A and D, by 
% Perrin-Smirnov_Springer fiber components in the two columns case for types A and D are normal.
% There, they prove that the singularities of those components are normal, Cohen–Macaulay, and have rational singularities

%\subsubsection{\FZ Floer product is the convolution product}\label{FloerProductMinimalComps} 
%In this section we assume that a CSR is \HK manifold, such that \omC
%In this section we show that the Floer product
%\begin{equation}\label{LagrangianFloerProduct} 
%HF^*(L_2,L_3)\otimes HF^*(L_1,L_2)\fun HF^*(L_1,L_3)
%\end{equation}
%between three minimal components $\Fminn,\Fminnn,\F_{\Fi^3}$ of a weight-1 CSR $\M$ is simply the convolution product on cohomology, under the isomorphisms
%$HF^*(\F_{\Fi^i},\F_{\Fi^j})\iso H^*(\F_{\Fi^i} \cap \F_{\Fi^j}).$ from Proposition \ref{LagrFloerMinComps}. In particular, given a single minimal component $\F_{\Fi},$ the Floer product on $HF^*(\F_{\Fi},\F_{\Fi})\iso H^*(\F_{\Fi})$ becomes the ordinary cup product. 
%
%We prove this using the pearl model for Lagrangian Floer cohomology defined and developed by Biran-Cornea \cite{BiCo07,BiCo09a,BiCo09b}
%and the fact that the product (\ref{LagrangianFloerProduct}) counts only constant solutions.
\subsection{Higher Floer operations on smooth core components}\label{FloerProductsOnSHS}
Following Theorem \ref{LagrFloerMinComps}, an interesting avenue of further Floer-theoretic research on minimal, and more general, smooth core components could be towards computing their Floer product, or more generally, computing the higher operations of the Fukaya category between them. At the moment this is far from reach, but let us mention an encouraging result which says that the operations in the Fukaya category of a SHS $\M$ that involve smooth components of the core come only from considering constant solutions.
We say that a Riemann surface $\Sigma$ is of \textbf{type-$n$} if its boundary $\partial \Sigma$ has punctures, which decompose it into $n$ connected pieces, denoted by $\partial \Sigma_i, i=1,\dots,n.$ 
We use the almost \HK notation above equation \eqref{cplxVolForm}.
\begin{prop}\label{FloerSolnsAreConstant}
\textbf{(Holomorphic maps with boundary in the core are constant)}
Consider a weight-1 SHS $\M$ and pick \textit{any} choice of complex structure $\Theta \in \{aJ+bK \mid a^2+b^2=1\}.$ %and its compatible symplectic form $\om_{\Theta}.$ 
Given any smooth core components $L_1,\dots,L_n,$ and a Riemann surface $\Sigma$ of type-$n,$ any $\Theta$-holomorphic map $u:\Sigma\fun \M,\; u( \partial \Sigma_i) \subset L_i$ is constant.
\end{prop}
\begin{proof}
 Let us assume that we have a non-constant $S$-holomorphic map $$u:\Sigma\fun \M,\; u( \partial \Sigma) \subset \L,$$ 
%  where J is an almost complex structure compatible with a form $\om_{J,K}.$ 
By Gromov Lemma we have $\int_{\Sigma} u^*\om_{J,K}= {1\over 2}\int_{\Sigma} ||du||^2>0.$ 
On the other hand, like in the proof of Proposition \ref{canonicalLiouville}, we have $\om_{J,K}=d\th_{J,K}$ where $\th_{J,K}=i_{Z}\om_{J,K}$ and $Z$ is vector field of the $\R_+$-action, up to a scale. Thus, by Stokes theorem
$\int_{\Sigma} u^*\om_{J,K}=\int_{\partial \Sigma} u^*\th_{J,K},$ and the last integral vanishes due to 
$\th_{J,K}(\xi)=0$ for any vector $\xi\in T\L$ (see the proof of Proposition \ref{SmoothIsExact}). Thus we get a contradiction.
\end{proof}

\begin{rmk}
We remark that there is a similar result for holomorphic Lagrangians in \HK manifolds which is the main theorem of a recent paper by Solomon-Verbitsky \cite[Thm. 1.6]{SoVe19}, which gives the slightly weaker conclusion, as they claim the holomorphic maps to be constant for a \textit{generic} choice of symplectic form among $\om_{J,K},$ whereas our works for any.
However, they cover the case of all holomorphic Lagrangians, whereas we choose only the ones landing in the core, or equivalently (due to Lemma \ref{SmoothIsExact}), the Bohr--Sommerfeld ones.
%whereas ours work for any. 
%which they have, for choosing the symplectic form among the forms $\om_{J,K}.$} statement as Proposition \ref{FloerSolnsAreConstant} in the 
%Hoverstatement for arbitrary \HK manifolds and arbitrary holomorphic Lagrangians therein. 
%However, this is the main result of their paper, whereas our proof for the setup of SHS manifolds and the Lagrangians of the core goes rather easy, as we have just seen.
\end{rmk}

So, having Proposition \ref{FloerSolnsAreConstant} in mind, one should expect that the Floer product on three smooth core components $L_1,L_2,L_3,$
under the isomorphisms from Theorem \ref{LagrFloerMinComps}, becomes just the convolution product (with $d$ being the the appropriate shift)
\begin{equation*}
	H^*(L_2 \cap L_3) \times H^*(L_1 \cap L_2) \fun H^{*-d}(L_1 \cap L_3), \ (\a,\b) \mapsto p_{13}! (p_{23}^*\a \cup p_{12}^*\b),
\end{equation*}
possibly twisted by coefficients coming from our twisted Pin-structures on Lagrangians, cf. Remark \ref{OnCoefficients}. 
In cases when core components are spin manifolds there are no needs for these twists, and in particular when $\M=\M^{m,n}$ is a resolution of a Slodowy variety $\M_0^{m,n}$ of two-row type, 
the expected convolution algebra of intersections of core components compares well with two isomorphisms in the existing literature:
\begin{itemize}
 \item One from \cite{SW12}, between convolution algebra of the same intersections and generalised arc algebra $H^{m,n}.$
\item One from \cite[Prop. 7.8]{MaS19}, between a certain Floer algebra and generalised arc algebra $H^{m,n}.$ 
\end{itemize}
In the latter, they considered a deformation $\mathcal{Y}_{n,m}$ of $\M_0^{m,n}$ (which is symplectomorphic to $\M^{m,n}$) 
and a specific set of Lagrangians in $\mathcal{Y}_{n,m},$ which are topologically equivalent to the core components of $\M^{m,n},$ 
and their Floer cohomologies are equivalent (as graded vector spaces) to the ones of core components. %we get on the core of $\M^{m,n}$
Thus, having these examples and Proposition \ref{FloerSolnsAreConstant} in mind, it is reasonable to expect that the Floer product of 
smooth core components indeed correspond to (in general twisted) convolution product.
\section{Applications towards symplectic cohomology}\label{MinLagrasAndSH}

In this section we use the existence of minimal components in order to give some information on the symplectic cohomology $SH^*(\M):=SH^*(\M,\om_{J,K})$ of an SHS manifold. 
In general, symplectic cohomology is notoriously hard to compute explicitly, so we usually have to make do with partial information. We will obtain its degree-wise lower bounds.

Recall first by Lemma \ref{LemmaCSRsAreAlmostHK} that, given an SHS $\M,$ one has a compatible almost \HK structure $(g,I,J,K)$ on it. 
Thus, as in \eqref{cplxVolForm} one gets a complex volume form $\Om_{\Theta}$ for any complex structure $\Theta \in \{aJ+bK \mid a^2+b^2=1\},$ 
so this gives us a $\Z$-grading on the symplectic cohomology $SH^*(\M).$ % is canonically $\Z$-graded.
In principle, this grading depends on the choice of compatible almost \HK structure, but in the majority of examples (such as CSRs or Moduli space of Higgs bundles) one is already given a \HK structure, so the corresponding grading on symplectic cohomology is canonical. The existence of $\Z$-grading is important, as we will get a degree-wise information on its ranks; getting information on its global rank would be not much of use as symplectic cohomology for these manifolds is expected to be infinite-dimensional (cf. Remark \ref{NoteOnFiniteRAnksInEAchDimension}). 

Regarding the coefficients for Lagrangian Floer cohomologies and symplectic cohomology, we will in general work over field $\K$ of characteristic 2, 
but for most known examples we can actually use arbitrary fields $\K,$ cf. Remark \ref{OnCoefficients}.

We will make a crucial use of the commuting triangle due to Seidel \cite[Sec. (5a)]{Sei08}. 
Namely, given a closed exact Lagrangian submanifold $L\xhookrightarrow{i} M$ inside a Liouville manifold $M,$ there is a commuting triangle
\begin{equation} \label{SeidelTriangle}
	\begin{tikzcd}[column sep=small]
		H^*(M) \arrow{r}{c^*}  \arrow{rd}{i^*} 
		& SH^*(M) \arrow{d}{\mathcal{CO}^0} \\
		& HF^*(L,L)
	\end{tikzcd}
\end{equation}
consisting of ring homomorphisms such that the diagonal map, under the isomorphism $HF^*(L,L)\iso H^*(L),$ becomes the restriction map $i^*:H^*(M)\fun H^*(L)$ for singular cohomology. 

Here, the vertical map is the \textbf{closed-open string map} 
${\mathcal{CO}^0}:SH^*(M,\om) \fun HF^*(L,L).$ As we will not use its definition but merely its properties, we refer to the original paper \cite{Ab15} for a thorough exposition. 
Roughly speaking, this map counts half-cylinders satisfying the Floer equation and having a boundary on $L.$
The horizontal $c^*$ map is a version of the PSS map for symplectic cohomology. 
As Seidel points out, having a closed exact Lagrangian yields the non-vanishing of the symplectic cohomology, as the diagonal map $i^*$ sends unit to the unit, thus cannot vanish through $SH^*(M).$ 

\begin{cor} Given a weight-1 SHS $\M,$ its symplectic cohomology %over $\Z/2$-coefficients 
	is non-zero, $$SH^*(\M,\om_{J,K})\neq 0.$$
\end{cor}
\begin{proof}
By Proposition \ref{minimalcomp} there is at least one smooth core component of $\M,$ and by Lemma \ref{SmoothIsExact} it is a closed exact Lagrangian submanifold. 
The claim follows due to triangle \eqref{SeidelTriangle} applied to it.
\end{proof}

\begin{rmk}\label{WhenExactSHvanishesHiggsBranch}
	It should be noted that for SHS manifolds which do not admit a weight-1 conical action, symplectic cohomology may in fact vanish, which in turn prevents the existence of exact Lagrangians
	(by the argument above using the triangle \eqref{SeidelTriangle}).
	%that if a CSR $\M$ does not admit a weight-1 conical action, the symplectic cohomology $SH^*(\M,\om_{J,K})$ vanishes and hence $\M$ will not admit any smooth exact Lagrangians. 
	This is true for the case of Higgs branch spaces, as by Corollary \ref{SubcritStein} they are all subcritical Stein manifolds, so by Cieliebak \cite[p. 121]{Cie02} the symplectic cohomology vanishes. In particular, for Nakajima quiver varieties this happens precisely when quiver has a loop edge. %following the comment written at the very end of \cref{HiggsBranch}
\end{rmk}

By studying the closed-open string map, we can also obtain lower bounds on the rank of $SH^*(\M).$ For that matter, we will essentially use the results regarding the homology decomposition, obtained in Section \ref{SectionHomologyDecompositionOfCore}.
%we will need a result about the decomposition of singular homology of a variety in the presence of a $\C^*$-action by Carrell--Goresky \cite{CaGo83} %and the \BB decomposition (Theorem \ref{BBDecompositionGeneral}).
%beautiful result from on the decomposition of singular homology of a projective variety, applied to the setting of the Bia\l{}ynicki-Birula decomposition \cite{BB73} that arises from a holomorphic $\C^*$-action. 
%We remark that we will not be using the notation from that paper, rather we continue with the one we previously introduced in this thesis, for the reader's convenience.

\begin{prop}\label{SurjectiveCOmap}
The closed-open string map $$\mathcal{CO}^0: SH^*(\M)\fun HF^*(\Fmin,\Fmin)$$ of a minimal component $\Fmin$ is surjective. 
\end{prop}
\begin{proof}

Using diagram (\ref{SeidelTriangle}) for the minimal component $\Fmin$

\begin{equation} \label{SeidelTriangle2}
\begin{tikzcd}[column sep=small]
H^*(\M) \arrow{r}{c^*}  \arrow{rd}{i^*} 
& SH^*(\M) \arrow{d}{\mathcal{CO}^0} \\
& HF^*(\Fmin,\Fmin)
\end{tikzcd}
\end{equation}
we see that the surjectivity of ${\mathcal{CO}^0}$ %$$\mathcal{CO}^0: SH^*(\M,\om_{J,K})\fun HF^*(\Fmin,\Fmin)$$ 
would follow from the surjectivity of the map $H^*(\M)\fun HF^*(\Fmin,\Fmin),$ which under the isomorphism 
$HF^*(\Fmin,\Fmin)\iso H^*(\Fmin)$ becomes the restriction map $H^*(\M)\fun H^*(\Fmin).$ Recalling that the inclusion of the core $\L\subset \M$ is a homotopy equivalence (Lemma \ref{CoreIsADefRetr}), %\ref{CoreIsHtpyEquiv}), 
it is enough to show that the restriction map $H^*(\L)\fun H^*(\Fmin)$ is surjective.

But this follows directly from Proposition \ref{HomologyDecompositionOfTheCore}. Indeed, it gives us an isomorphism
\begin{equation}\label{sumsumsum}
\Phi_{\Fi}=\oplus_i \eta_i: \oplus_i H_*(\F_i)[-\mu_i] \fja{\iso} H_*(\L)
\end{equation}
where $\F_i$ are connected components of the fixed locus $\M^{\Fi},$ the $\mu_i$ are real dimensions of $(t\fun \infty)$-attracting bundles
$p_i:\L_i\fun \F_i$ and	$$\eta_i:  H_*(\F_i)[-\mu_i] \fun H_*(\L), \ [C]\mapsto [\ol{p_i^{-1}(C)}],$$
for a generic cycle $C.$
Notice that the minimal component $\F_{i_0}:=\F_{\Fi}$ satisfies $$\F_{i_0}=\L_{i_0}=\ol{\L_{i_0}}$$ (equation (\ref{TheEpicenterOfThesis}) in the proof of Proposition \ref{minimalcomp}). 
Thus, $\mu_{i_0}=0$ and $\ol{p_{i_0}^{-1}(C)}=C,$ for a cycle $C,$ thus
$$\eta_{i_0}=(i_{\Fi})_*:H_*(\F_{\Fi})\fun H_*(\L),$$ where $i_{\Fi}:\F_{\Fi}\hookrightarrow \L$ is the inclusion map. Hence by isomorphism (\ref{sumsumsum}), the map $(i_{\Fi})_*$ is injective. As we are working over field coefficients, %under the isomorphisms
the restriction map on cohomology $(i_{\Fi})^*:H^*(\L)\fun H^*(\F_{\Fi})$ is naturally isomorphic to the dual map $\Hom((i_{\Fi})_*,\K),$ hence is surjective. Thus by the argument above, the proposition is proved.
\end{proof}

As a corollary, we obtain the lower bounds on the ranks of symplectic cohomology by ordinary homology of a minimal component.
%Denote by $b_i(X)$ the $i$-th Betti number of a topological space $X.$

\begin{cor} Given an arbitrary minimal component $\Fmin$ of a weight-1 SHS $\M$, for all $k\in \N_0$, $$ {\rk}(SH^k(\M))\geq b_k(\Fmin).$$
\end{cor}
By looking carefully at the proof of Proposition \ref{SurjectiveCOmap}, notice that we have shown that a block of $H^*(\M),$ that is isomorphic to $H^*(\F_{\Fi})$ via the restriction map, injects in $SH^*(\M)$ with the $c^*$-map. Thus, fixing a weight-1 action $\Fi$ with the fixed locus $\F=\sqcup \F_i,$ we can inject the cohomologies of all $\F_i$ which lie in minimal components to $SH^*(\M)$ via the $c^*$-map. Notice that in order to prove this, we have to use both the homology decompositions for the action $\Fi$ and each of those minimal components,
%which exists as it is a holomorphic $\C^*$-action on a closed \KH manifold \cite{CaSo79},
and the homology decomposition for the core,
%with the homology decompositions for the {\FZZ}action $\Fi$ restricted to a minimal component and for the actions that produce minimal components, 
and to prove the compatibility of those decompositions. We do it in the following theorem.
% By considering all minimal components $\textrm{Min}(\M)=\{\F_\Fi \mid \Fi\in Con_1(\M)\}$ we get a stronger statement. %and $\mu_j$ denotes the Morse-Bott index of a fixed connected component $\F_j$ of a $\C^*$-action.

\begin{thm}\label{boundonSHBest}
	Let $(\M,\Fi)$ be a weight-1 SHS. Denote by $\F=\sqcup_i \F_i$ the fixed locus of $\Fi$ decomposed into connected components $\F_i,$ and 
	$\mu_i$ the real dimensions of corresponding attracting bundles $\L_i \fun \F_i.$ Then
%	the dual of the image $\displaystyle \bigoplus_{\{\a |\Labar=\Fmin\}} H_{*-\mu_\a}(\F_\a)\hookrightarrow H_*(\M)$ injects into $SH^*(\M)$ under the $c^*$ map.
%	Thus, we get an estimate for ranks of symplectic cohomology:
	$$ {\rk}( SH^k(\M))\;\geq \sum_{\{i |\,  \ol{\L_i} \in \mathrm{Min}(\M,\Fi)\}} b_{k-\mu_i}(\F_i),$$
%	$$ \text{rk}\,( SH^k(\M,\omega_{J,K}))\;\geq \sum_{\{\a \,|\,\Labar=\Fmin,\, \Fi\in\Con1\}} b_{k-\mu_\a}(\F_\a)$$
	for all $k\geq 0$. In particular, $\rk (SH^{dim_{\C}\M}(\M))\geq |\emph{Min}(\M,\Fi)|.$
\end{thm}
\begin{proof}

Firstly, for any $\Fi$-fixed component $\F_i$ by Proposition \ref{HomologyDecompositionOfTheCore} applied to action $\Fi$  
%for $X=\L$ and $Y$ a $\C^*$-equivariant compactification of $\M$ (as in Proposition \ref{SurjectiveCOmap}), 
we get an injective map
\begin{equation*}\label{etaInTheCore}
	\eta_i: H_*(\F_i)[-\mu_{i}] \fun H_*(\L), \ [C]\mapsto [\ol{p_i^{-1}(C)}], \text{ for a generic cycle } C.
\end{equation*}
where $p_i:\L_i\fun \F_i$ is the $(t\fun \infty)$-attracting bundle.
As $\Fi$ is a weight-1 conical action, by Corollary \ref{CoreDecompositonWeight1} the irreducible components of the core are precisely the closures $\ol{\L_i}$ of the attracting bundles $\L_i \fun \F_i.$
% are irreducible components of $\L.$ 
We will focus on the minimal ones. %$\ol{\L_i}=\F_{\phi},$ for some weight-1 action $\phi.$ 

So, fixing a weight-1 action $\phi,$ denote by $\F_{i(\phi)}$ the $\Fi$-fixed component satisfying $\ol{\L_{i(\phi)}}=\F_{\phi}.$
Then, $\Fi$ restricts to the minimal component $\F_{\phi},$
\footnote{Action $\Fi$ restrict to the core $\L,$ so as an algebraic action, it acts on the set of irreducible components. 
	Moreover, as $\C^*$ is connected, $\Fi$ acts trivially on the top homology of $\L,$ which is generated by the irreducible 
	components (Lemma \ref{LemaNonIsotopic}), hence each of them is $\Fi$-invariant.}
giving a holomorphic action on a closed \KH manifold (one can use the \KH structure from Lemma \ref{thereisaNonExactstructure}),
and in such setup there is a homology decomposition by Carrell-Sommese \cite[Thm. 1]{CaSo79}. Thus, we get an injective map
\begin{equation*}\label{etaRestricted}
\eta_i^{\phi}: H_*(\F_{i(\phi)})[-\mu_{i}^{\phi}] \fun H_*(\F_{\phi}), \ [C]\mapsto [\ol{p_i^{-1}(C)}^{\phi}], \text{ for a generic cycle } C.
\end{equation*}
	Here the closure $\ol{p_i^{-1}(C)}^{\phi}$ is taken in $\F_\phi,$ and $\mu_i^{\phi}$ is the dimension of the part of the bundle $\L_{i(\phi)} \fun \F_{i(\phi)}$ that lies in $\F_{\phi}.$ But, since $\F_{\phi}=\ol{\L_{i(\phi)}},$ the closures and shifts agree $\ol{p_i^{-1}(C)}^{\phi}=\ol{p_i^{-1}(C)},$ 
	$\mu_{i}^{\phi}=\mu_{i}.$
	Thus, we have the compatibility of maps
	$$(i_\phi)_* \eta_i^{\phi} =\eta_{i(\phi)} ,$$
	where $i_\phi: \F_{\phi} \hookrightarrow \L$ is the inclusion.
	
Thus, %given a  there is a $\Fi$-fixed connected component $\F_{i(\phi)}$ defined by $\ol{\L_{i(\phi)}}=\F_{\phi},$
we get that the image of the inclusion map $H_*(\F_{\phi}) \fja{(i_{\phi})_*} H_*(\L)$ contains $\eta_{i(\phi)}(H_*(\F_{i(\phi)})[-\mu_i]).$ Thus, summing over all $\phi\in \Con1$
we get that the image of the map
	\begin{equation}\label{themapimportant}
	\oplus_{\phi} H_*(\F_{\phi}) \fja{\oplus (i_{\phi})_*} H_*(\L)
	\end{equation}
contains the image of the map 
	$$\oplus_{\phi} \eta_{i(\phi)}: \oplus_{\phi} H_*(\F_{i(\phi)})[-\mu_{i(\phi)}]\fun H_*(\L),$$ 
which is injective,
by isomorphism (\ref{sumsumsum}). Denote by $U:=Im(\oplus_{\phi} \eta_{i(\phi)})$ and by $U'$ its arbitrary graded complement in $H_*(\L).$ 
	
Applying the functor $\Hom(\cdot,\Z/2)$ on the map (\ref{themapimportant}), we get the map
	\begin{equation}\label{equnequn}
	\Hom(H_*(\L),\Z/2) \fja{\oplus_\phi Hom((i_{\phi})_*,\Z/2)} \oplus_{\phi} \Hom(H_*(\F_{\phi}),\Z/2)
	\end{equation}
which we will show to be injective on the dual of $U,$ by applying the following linear-algebraic lemma:
\begin{lm}\label{LemaGojko} Given vector spaces $V$ and $W=U\oplus U'$ and a linear map $L:V\fun W$ such that $U\leq L(V),$
		the dual map $L^*:W^*\fun V^*$ is injective when restricted to $U^*:=\{\xi\in W^* \mid \xi|_{U'}=0\}.$
\end{lm}
\begin{proof}
We just have to show that if $L^*(\xi)=0$ for $\xi\in U^*$ then $\xi=0.$
By assumption, $0=L^*(\xi)(v)=\xi(L(v))$ for all $v\in V,$ so $\xi|_{L(V)}=0,$ so $\xi|_U=0.$ As $\xi\in U^*,$ 
we also have $\xi|_{U'}=0,$ hence $\xi=0$.
\end{proof}
Thus, by this lemma the map ${\oplus_\phi Hom((i_{\phi})_*,\Z/2)}$ injects $U^*$ into $\oplus_{\phi} \Hom(H_*(\F_{\phi}),\Z/2).$ As $U'$ was chosen to be a graded complement, $U^*$ is isomorphic to $U$ as a graded vector space. Now, using the Kronecker isomorphisms 
$\kappa: H^*(\L) \fun Hom(H_*(\L),\Z/2), \ \kappa:H^*(\F_{\phi})\fun \Hom(H_*(\F_{\phi}),\Z/2)$
we pass from (\ref{equnequn}) to the map
$$H^*(\L) \fja{\oplus_\phi i_{\phi}^*} \oplus_{\phi} H^*(\F_{\phi}),$$
which then injects  $\kappa^{-1}(U^*)$ into $\oplus_{\phi} H^*(\F_{\phi}).$ 
	
Finally, we connect this with symplectic cohomology. Considering the diagram (\ref{SeidelTriangle2}) for a minimal component $\F_{\phi},$ and using the isomorphism $r: H^*(\M) \fja{\iso} H^*(\L)$ (given by the restriction map), and $S_\phi: HF^*(\F_{\phi},\F_{\phi})\fja{\iso} H^*(\F_{\phi})$ we get the diagram
\begin{equation*} %\label{SeidelTriangleex3}
\begin{tikzcd}[column sep=small]
	H^*(\L) \arrow{r}{c}  \arrow{rd}{i_{\phi}^*} 
	& SH^*(\M) \arrow{d}{co_{\phi}} \\
	&  H^*(\F_{\phi}),
\end{tikzcd}
\end{equation*}
where $c=c^* \circ r^{-1}$ and $co_{\phi}:=S_\phi \circ \mathcal{CO}^0.$ Summing over all $\phi\in \Con1$ we get the diagram 
\begin{equation*} %\label{SeidelTriangleex4}
\begin{tikzcd}[column sep=small]
	H^*(\L) \arrow{r}{c}  \arrow{rd}{\oplus_\phi{i_{\phi}^*}} 
	& SH^*(\M) \arrow{d}{\oplus_{\phi} co_{\phi}} \\
	& \oplus_{\phi} H^*(\F_{\phi})
\end{tikzcd}
\end{equation*}	
Observe that, as $\kappa^{-1}(U^*)$ injects by the diagonal map $\oplus_\phi i_{\phi}^*$ of this diagram, it injects by the horizontal map $c$ as well.
The claim of the proposition then follows as
$\kappa^{-1}(U^*) \iso \oplus_{\phi} (H_*(\F_{i(\phi)})[-\mu_{i(\phi)}])$ (graded isomorphism),
and the summing goes through all $\phi\in\Con1,$ thus $\ol{\L_{i(\phi)}}$ goes through the set of all minimal components $\mathrm{Min}(\M,\Fi)$.
\end{proof}
We get an immediate corollary of this proposition:

\begin{cor}\label{AllMinimalEmbedHtoSH}
	If all components of $\L$ of an SHS $\M$ are minimal, the singular cohomology $H^*(\M)$ embeds into $SH^*(\M)$ via the $c^*$ map, hence for all $k\in \N_0$, $$\rk(SH^k(\M))\geq \rk(H^k(\M)).$$ 
\end{cor}
Let us see an instance of this corollary in the following example:

\begin{ex}
	Given the minimal resolution $X_{\Z/(n+1)}$ of a Du Val singularity of type $A_n,$ its core is topologically an $n$-wedge of 2-spheres, and as we have seen in Example \ref{DuVal_Weight1actions_TypeA}, these are all minimal components. 
	%{\FZ  write down explicit actions in earlier sections } %Examples \ref{DuVal_MinCompAreAll_QuiverVarSide} and \ref{DuVal_MinCompAreAll_SpringerSide}). 
	Hence, by Corollary \ref{AllMinimalEmbedHtoSH} we get that $$\text{rk}(SH^2(X_{\Z/(n+1)}))\geq H^2(X_{\Z/(n+1)})=n.$$
	It is known by \cite[Cor. 42]{EL17} that $\text{rk}(SH^2(X_{\Z/(n+1)}))=n$ and that it vanishes in higher degrees, 
	so in this example Theorem \ref{boundonSHBest} gives the actual rank on the highest non-zero degree of symplectic cohomology.
\end{ex}

%{\FZ PROBABLY ERASE \\
%Finally, notice that, by the proof of Proposition \ref{boundonSHBest}, we can actually show that the dual of the image
%$\oplus_{\phi} H_*(\F_{\phi}) \fja{\oplus (i_{\phi})_*} H_*(\L)$
%injects into $SH^*(\M)$ via the map
%$c^*:H^*(\L)\cong H^*(\M) \fun SH^*(\M).$ Indeed, one simply uses Lemma \ref{LemaGojko} for $W=H_*(\L), V=\oplus_{\phi} H_*(\F_{\phi}),$ and $U$ being the image of $V$ under the map
%$\oplus (i_{\phi})_*.$ Then proceed as in the Proposition.
%
%However, it is hard to say how much the blocks $H_*(\F_{\phi})$ overlap in their injective images in $H_*(\L).$
%A hope is that the map
%\begin{equation}\label{kikiki}
%\oplus_{\phi} H_*(\F_{\phi})\fun H_*(\cup \F_{\phi})
%\end{equation}
%is surjective, as minimal components together with their intersections have even-only homology (thus (\ref{kikiki}) is immediately surjective for up to three components). Having this surjectivity would prove the following conjecture:
%
%\begin{con}\label{ConjectureRkSH} Given a HKCSR $\M,$
%	Define the subvariety of its core $$\displaystyle \L_{min}:=\cup_{\varphi\in \Con1}\Fmin$$ that is made out of all minimal components. The dual of its image $H_*(\L_{min})\fun H_*(\L)\iso H_*(\M)$ embeds into $SH^*(\M)$ under the $c^*$ map. Thus, for all $k\in \N_0$,
%	$$rk(SH^k(\M))\geq rk(H_k(\L_{min})).$$ 
%\end{con}
%}

\begin{rmk}\label{NoteOnFiniteRAnksInEAchDimension}
	Note that, in principle, symplectic cohomology may be infinite-dimensional even degree-wise.\footnote{The simplest example is $T^*S^1,$ whose symplectic cohomology is supported in degrees 0 and 1 and has an infinite rank in both.} However, it is indeed finite-dimensional in each degree for certain examples of SHS manifolds:
	\begin{enumerate}[(1)]
		\item Minimal resolutions of Du Val singularities $X_\Gamma,$ due to \cite{EL17}.
		\item Cotangent bundles of generalised flag varieties $T^*(G/P),$ where $G$ is complex semisimple group and $P$ a parabolic subgroup. By Viterbo isomorphism \cite{Vi96, Ab15}, $SH^*(T^*(G/P))$ is isomorphic to the singular homology of the free loop space $\mathcal{L}(G/P)$ of flag variety, so we can use a general result \cite[Prop. 9, Ch. IV]{Ser51} by Serre that the homology of the free loop space of a simply connected space is finite-dimensional degree-wise. The flag variety $G/P$ is simply connected due to the decomposition into Schubert cells which are isomorphic to affine spaces \cite{Bo54}.
	\end{enumerate}
	Thus, we believe that for general SHS $\M$ the degree-wise ranks $rk(SH^k(\M)), \ k \in \Z$ should be finite-dimensional, so the lower bounds obtained in previous statements actually provide some non-trivial information.
\end{rmk}

%\input{delete}

%\appendix

%\include{appendix1}

%\include{appendix2}

\bibliography{FZ}

\providecommand{\bysame}{\leavevmode\hbox to3em{\hrulefill}\thinspace}
\providecommand{\MR}{\relax\ifhmode\unskip\space\fi MR }
\providecommand{\MRhref}[2]{%
  \href{http://www.ams.org/mathscinet-getitem?mr=#1}{#2}
}
\providecommand{\href}[2]{#2}
\begin{thebibliography}{BLPW16}

\bibitem[AB83]{AtB83}
M.~F. Atiyah and R.~Bott, \emph{{The Yang-Mills equations over Riemann
  surfaces}}, Philos. Trans. Roy. Soc. London Ser. A \textbf{308} (1983),
  no.~1505, 523--615.

\bibitem[Abo15]{Ab15}
M.~Abouzaid, \emph{{Symplectic cohomology and Viterbo's theorem}}, Free loop
  spaces in geometry and topology, IRMA Lect. Math. Theor. Phys \textbf{24}
  (2015), 271--485.

\bibitem[AJ10]{AkJo10}
M.~Akaho and D.~Joyce, \emph{{Immersed Lagrangian Floer theory}}, J.
  Differential Geom \textbf{86} (2010), no.~3, 381--500.

\bibitem[AS19]{AbS19}
M.~Abouzaid and I.~Smith, \emph{{Khovanov homology from Floer cohomology}}, J.
  Amer. Math. Soc \textbf{32} (2019), no.~1, 1--79.

\bibitem[BB73]{BB73}
A.~Bia{\l}ynicki-Birula, \emph{Some theorems on actions of algebraic groups},
  Ann. of Math \textbf{98} (1973), no.~3, 480--497.

\bibitem[BD00]{BD00}
R.~Bielawski and A.~Dancer, \emph{The geometry and topology of toric
  hyperk\"{a}hler manifolds}, Communications in Analysis and Geometry
  \textbf{8} (2000), no.~4, 727--759.

\bibitem[BFN18]{BFN16}
A.~Braverman, M.~Finkelberg, and H.~Nakajima, \emph{{Towards a mathematical
  definition of Coulomb branches of $3$-dimensional $\mathcal{N} = 4$ gauge
  theories, II}}, Adv. Theor. Math. Phys. \textbf{22} (2018), 1071--1147.

\bibitem[BGL15]{BGL15}
I.~Biswas, N.~L. Gammelgaard, and M.~Logares, \emph{Bohr--sommerfeld
  lagrangians of moduli spaces of higgs bundles}, Journal of Geometry and
  Physics \textbf{94} (2015), 179--184.

\bibitem[BLPW16]{BLPW16}
T.~Braden, A.~Licata, N.~Proudfoot, and B.~Webster, \emph{{Quantizations of
  conical symplectic resolutions II: category $\mathcal{O}$ and symplectic
  duality}}, Asterisque No \textbf{384} (2016), 75--179.

\bibitem[BM08]{fuctiorialResolution}
E.~Bierstone and P.~D. Milman, \emph{Functoriality in resolution of
  singularities}, Publications of the Research Institute for Mathematical
  Sciences \textbf{44} (2008), no.~2, 609--639.

\bibitem[Bor54]{Bo54}
A.~Borel, \emph{{K{\"a}hlerian Coset Spaces of Semisimple Lie Groups}},
  Proceedings of the National Academy of Sciences \textbf{40} (1954), no.~12,
  1147--1151.

\bibitem[BPW16]{BPW16}
T.~Braden, N.~Proudfoot, and B.~Webster, \emph{{Quantizations of conical
  symplectic resolutions I: local and global structure}}, Ast{\'e}risque No
  \textbf{384} (2016), 1--73.

\bibitem[Bre72]{Bre72}
G.~E. Bredon, \emph{Introduction to compact transformation groups}, Academic
  press, 1972.

\bibitem[BS21]{BeSch16}
G.~Bellamy and T.~Schedler, \emph{Symplectic resolutions of quiver varieties},
  Selecta Mathematica \textbf{27} (2021), no.~3, 1--50.

\bibitem[BZ08]{BaZi08}
L.~Barchini and R.~Zierau, \emph{{Certain components of Springer fibers and
  associated cycles for discrete series representations of $SU(p,q)$}},
  Represent. Theory \textbf{12} (2008), 403--434.

\bibitem[CG83]{CaGo83}
J.B. Carrell and R.M. Goresky, \emph{A decomposition theorem for the integral
  homology of a variety}, Invent. Math \textbf{73} (1983), no.~3, 367--381.

\bibitem[CG97]{CGi97}
N.~Chriss and V.~Ginzburg, \emph{{Representation Theory and Complex Geometry}},
  Birkh{\"a}user, Boston, 1997.

\bibitem[Cie02]{Cie02}
K.~Cieliebak, \emph{Handle attaching in symplectic homology and the chord
  conjecture}, J. Eur. Math. Soc \textbf{4} (2002), 115--142.

\bibitem[CS79]{CaSo79}
J.B. Carrell and A.J. Sommese, \emph{{Some topological aspects of
  $\mathbb{C}^*$ actions on compact Kaehler manifolds}}, Comment. Math. Helv
  \textbf{54} (1979), no.~4, 567--582.

\bibitem[EL17]{EL17}
T.~Etg\"{u} and Y.~Lekili, \emph{{Koszul duality patterns in Floer theory}},
  Geom. Topol. \textbf{21} (2017), no.~6, 3313--3389.

\bibitem[EL19]{EL19}
T.~Etg{\"u} and Y.~Lekili, \emph{Fukaya categories of plumbings and
  multiplicative preprojective algebras}, Quantum Topology \textbf{10} (2019),
  no.~4, 777--813.

\bibitem[FM10]{FrMe10}
L.~Fresse and A.~Melnikov, \emph{{On the singularity of the irreducible
  components of a Springer fiber in $\mathfrak{sl}_n$}}, Selecta Math
  \textbf{16} (2010), no.~3, 393--418.

\bibitem[FM11]{FrMe11}
L.~Fresse and A.~Melnikov, \emph{{Some characterizations of singular components
  of Springer fibers in the two-column case}}, Algebras and Representation
  Theory \textbf{14} (2011), no.~6, 1063--86.

\bibitem[FMSO15]{FrMeS-O15}
L.~Fresse, A.~Melnikov, and S.~Sakas-Obeid, \emph{{On the structure of smooth
  components of Springer fibers}}, Proceedings of the American Mathematical
  Society \textbf{143} (2015), no.~6, 2301--2315.

\bibitem[Fra59]{Frankel59}
T.~Frankel, \emph{{Fixed Points and Torsion on K{\"a}hler Manifolds}}, Annals
  of Mathematics \textbf{70} (1959), no.~1, 1--8.

\bibitem[Fre09]{Fr09b}
L.~Fresse, \emph{{Singular components of Springer fibers in the two-column
  case}}, Annales de l’institut Fourier \textbf{59} (2009), no.~6,
  2429--2444.

\bibitem[Fre11]{Fr11}
L.~Fresse, \emph{{On the singularity of some special components of Springer
  fibers}}, J. Lie Theory \textbf{21} (2011), no.~1, 205--242.

\bibitem[Fu03]{Fuu03}
B.~Fu, \emph{Symplectic resolutions for nilpotent orbits}, Inventiones
  mathematicae \textbf{151} (2003), no.~1, 167--186.

\bibitem[GH20]{GrHa20}
J.~F. Grimminger and A.~Hanany, \emph{Hasse diagrams for $3d \ \ \mathcal{N}=
  4$ quiver gauge theories -- inversion and the full moduli space}, Journal of
  High Energy Physics \textbf{2020} (2020), no.~9, 159.

\bibitem[Gin15]{Gi15}
V.~Ginzburg, \emph{On the geometry of symplectic resolutions}, PCMI, 2015.

\bibitem[GK04]{GiKalPoissonDeformation}
Victor Ginzburg and Dmitry Kaledin, \emph{Poisson deformations of symplectic
  quotient singularities}, Advances in Mathematics \textbf{186} (2004), no.~1,
  1--57.

\bibitem[Gro60]{Gro60}
A.~Grothendieck, \emph{{Techniques de construction et th{\'e}oremes
  d’existence en g{\'e}om{\'e}trie alg{\'e}brique. IV. Les sch{\'e}mas de
  Hilbert}}, S{\'e}minaire Bourbaki \textbf{6} (1960), no.~221, 249--276.

\bibitem[GZ11]{GrZi11}
W.~Graham and R.~Zierau, \emph{{Smooth components of Springer fibers}}, Annales
  de l'Institut Fourier \textbf{61} (2011), no.~5, 2139--2182. \MR{2961851}

\bibitem[Har77]{Ha77}
R.~Hartshorne, \emph{Algebraic geometry}, Grad. Texts in Math. 52, New York,
  Springer, 1977.

\bibitem[Hit87]{Hi87}
N.~J. Hitchin, \emph{{The self-duality equations on a Riemann surface}},
  Proceedings of the London Mathematical Society \textbf{3} (1987), no.~1,
  59--126.

\bibitem[HK20]{HaKa20}
A.~Hanany and R.~Kalveks, \emph{{Quiver theories and Hilbert series of
  classical Slodowy intersections}}, Nuclear Physics B \textbf{952} (2020),
  114939.

\bibitem[HL82]{HaLa82}
R.~Harvey and B.~Lawson, \emph{Calibrated geometries}, Acta Mathematica
  \textbf{148} (1982), 47--157.

\bibitem[HRV15]{HaR-V15}
T.~Hausel and F.~Rodriguez-Villegas, \emph{Cohomology of large semiprojective
  hyperk\"{a}hler varieties}, Ast\'{e}risque (2015), no.~370, 113--156.

\bibitem[HS18]{HaSp18}
A.~Hanany and M.~Sperling, \emph{{Resolutions of nilpotent orbit closures via
  Coulomb branches of 3-dimensional $\mathcal {N}=4$ theories}}, Journal of
  High Energy Physics \textbf{2018} (2018), no.~8, 1--36.

\bibitem[Hus94]{Hu94}
D.~Husemoller, \emph{{Fibre Bundles}}, third ed., Grad. Texts in Math. 20, New
  York, Springer, 1994.

\bibitem[Kal06]{Ka06}
D.~Kaledin, \emph{{Symplectic singularities from the Poisson point of view}},
  J. Reine Angew. Math \textbf{600} (2006), 135--156.

\bibitem[Kal08]{Ka08}
\bysame, \emph{Derived equivalences by quantization}, Geom. Funct. Anal
  \textbf{17} (2008), no.~6, 1968--2004.

\bibitem[Kal09]{Ka09}
\bysame, \emph{Geometry and topology of symplectic resolutions}, Algebraic
  geometry---{S}eattle 2005. {P}art 2, Proc. Sympos. Pure Math., vol.~80, Amer.
  Math. Soc., Providence, RI, 2009, pp.~595--628.

\bibitem[Kho00]{Kh00}
M.~Khovanov, \emph{A categorification of the {J}ones polynomial}, Duke Math. J.
  \textbf{101} (2000), no.~3, 359--426. \MR{1740682}

\bibitem[Kir84]{Ki84}
F.~C. Kirwan, \emph{Cohomology of quotients in symplectic and algebraic
  geometry}, Math. Notes 31, Princeton Univ. Press, Princeton, 1984.

\bibitem[KJ16]{Kir16}
A.~Kirillov~Jr, \emph{Quiver representations and quiver varieties}, vol. 174,
  American Mathematical Soc., 2016.

\bibitem[KN90]{KroNak90}
P.~B. Kronheimer and H.~Nakajima, \emph{{Yang-Mills instantons on ALE
  gravitational instantons}}, Math. Ann \textbf{288} (1990), 263--307.

\bibitem[Kro89]{Kro89}
P.~B. Kronheimer, \emph{{The construction of ALE spaces as a hyperk\"{a}hler
  quotients}}, J. Differential Geom \textbf{29} (1989), 665--683.

\bibitem[KS02]{KhSei02}
M.~Khovanov and P.~Seidel, \emph{{Quivers, Floer cohomology, and braid group
  actions}}, Journal of the American Mathematical Society \textbf{15} (2002),
  no.~1, 203--271.

\bibitem[Kuz07]{Kuz07}
A.~G. Kuznetsov, \emph{{Quiver varieties and Hilbert schemes}}, Moscow
  Mathematical Journal \textbf{7} (2007), no.~4, 673--697.

\bibitem[Maf05]{Maf05}
A.~Maffei, \emph{{Quiver varieties of type A}}, Commentarii Mathematici
  Helvetici \textbf{80} (2005), no.~1, 1--27.

\bibitem[MFK94]{GITbook}
David Mumford, John Fogarty, and Frances Kirwan, \emph{Geometric invariant
  theory}, vol.~34, Springer Science \& Business Media, 1994.

\bibitem[MO19]{MO12}
D.~Maulik and A.~Okounkov, \emph{Quantum groups and quantum cohomology},
  Ast{\'e}risque no \textbf{408} (2019), ix+209.

\bibitem[MS21]{MaS19}
C.~Y. Mak and I.~Smith, \emph{{Fukaya--Seidel categories of Hilbert schemes and
  parabolic category $\mathcal{O}$}}, Journal of the European Mathematical
  Society (2021).

\bibitem[Nak94]{Nak94a}
H.~Nakajima, \emph{{Instantons on ALE spaces, quiver varieties, and Kac-Moody
  algebras}}, Duke Mathematical Journal \textbf{76} (1994), no.~2, 365--416.

\bibitem[Nak98]{Nak98}
\bysame, \emph{{Quiver varieties and Kac-Moody algebras}}, Duke Mathematical
  Journal \textbf{91} (1998), no.~3, 515--560.

\bibitem[Nak99]{Nak99}
\bysame, \emph{{Lectures on Hilbert schemes of points on surfaces}}, no.~18,
  American Mathematical Soc., 1999.

\bibitem[Nak01]{Nak01}
\bysame, \emph{Quiver varieties and finite-dimensional representations of
  quantum affine algebras}, Journal of the American Mathematical Society
  (2001), 145--238.

\bibitem[Nak15]{Nak15}
\bysame, \emph{{Towards a mathematical definition of Coulomb branches of
  $3$-dimensional $\mathcal{N}= 4$ gauge theories, I}}, arXiv:1503.03676
  (2015).

\bibitem[Nak18]{Nak18}
\bysame, \emph{{Instantons on ALE spaces for classical groups}},
  arXiv:1801.06286 (2018).

\bibitem[Nam08]{Nam08}
Y.~Namikawa, \emph{{Flops and Poisson deformations of symplectic varieties}},
  Publ. Res. Inst. Math. Sci \textbf{44} (2008), no.~2, 259--314.

\bibitem[Nam11]{Nam11}
\bysame, \emph{Poisson deformations of affine symplectic varieties}, Duke
  Mathematical Journal \textbf{156} (2011), no.~1, 51--85.

\bibitem[Nam18]{Nam18}
Y.~Namikawa, \emph{A characterization of nilpotent orbit closures among
  symplectic singularities}, Mathematische Annalen \textbf{370} (2018), no.~1,
  811--818.

\bibitem[Nit91]{Nitsure91}
N.~Nitsure, \emph{Moduli space of semistable pairs on a curve}, Proceedings of
  the London Mathematical Society \textbf{3} (1991), no.~2, 275--300.

\bibitem[Po{\'{z}}99]{Po99}
M.~Po{\'{z}}niak, \emph{{Floer homology, Novikov rings and clean
  intersections}}, Northern California Symplectic Geometry Seminar, Amer. Math.
  Soc. Transl. Ser. , Adv. Math. Sci., 45, AMS \textbf{2} (1999), no.~196,
  119--181.

\bibitem[PR07]{PaRe06}
N.~G.~J. Pagnon and N.~Ressayre, \emph{{Adjacency of Young tableaux and the
  Springer fibers}}, Selecta Mathematica \textbf{12} (2007), no.~3, 517--540.

\bibitem[RS93]{RS93}
J.~Robbin and D.~Salamon, \emph{{The Maslov index for paths}}, Topology
  \textbf{32} (1993), 827--844.

\bibitem[R{\v{Z}}22]{RZ22}
A.~F. Ritter and F.~{\v{Z}}ivanović, \emph{{Symplectic cohomology of Conical
  Symplectic Resolutions}}, in preparation, 2022.

\bibitem[Sei00]{Sei00}
P.~Seidel, \emph{{Graded Lagrangian submanifolds}}, Bull. Soc. Math. France
  \textbf{128} (2000), no.~1, 103--149.

\bibitem[Sei08a]{Sei08}
\bysame, \emph{A biased view of symplectic cohomology}, Current Developments in
  Mathematics \textbf{2006} (2008), 211--253.

\bibitem[Sei08b]{SeiBook}
\bysame, \emph{{Fukaya categories and Picard-Lefschetz theory}}, Zurich
  Lectures in Advanced Mathematics, European Mathematical Society (EMS),
  Z{\"{u}}rich, 2008.

\bibitem[Ser51]{Ser51}
J.-P. Serre, \emph{{Homologie singuli{\`e}re des espaces fibr{\'e}s.
  Applications}}, Ann. of Math \textbf{54} (1951), 425--505.

\bibitem[Sim96]{Si96}
C.~Simpson, \emph{{The Hodge filtration on nonabelian cohomology}},
  arXiv:alg-geom/9604005 (1996).

\bibitem[Slo80]{SloBook}
P.~Slodowy, \emph{{Simple singularities and simple algebraic groups}}, Lecture
  Notes in Mathematics, vol. 815, Springer, Berlin, 1980.

\bibitem[Som75]{So75}
A.~J. Sommese, \emph{{Extension theorems for reductive group actions on compact
  Kaehler manifolds}}, Math. Ann \textbf{218} (1975), 107--116.

\bibitem[Spa76]{Spa76}
N.~Spaltenstein, \emph{{The Fixed Point Set of a Unipotent Transformation on
  the Flag Manifold}}, Indagationes Mathematicae (Proceedings) \textbf{79}
  (1976), no.~5, 452--456.

\bibitem[SS05]{SeiSm05}
P.~Seidel and I.~Smith, \emph{{The symplectic topology of Ramanujam's
  surface}}, Commentarii Mathematici Helvetici \textbf{80} (2005), no.~4,
  859--881.

\bibitem[ST80]{SeTh80}
H.~Seifert and W.~Threlfall, \emph{{Seifert and Threlfall: a textbook of
  topology and Seifert: Topology of 3-dimensional fibered spaces}}, Pure and
  Applied Mathematics, 89. Academic Press, Inc., New York-London, 1980.

\bibitem[Sum74]{Su74}
H.~Sumihiro, \emph{Equivariant completion}, J. Math. Kyoto Univ. \textbf{14}
  (1974), no.~1, 1--28.

\bibitem[SV19]{SoVe19}
J.~P. Solomon and M.~Verbitsky, \emph{{Locality in the Fukaya category of a
  hyperk{\"a}hler manifold}}, Compositio Mathematica \textbf{155} (2019),
  no.~10, 1924--1958.

\bibitem[SW12]{SW12}
C.~Stroppel and B.~Webster, \emph{{2-block Springer fibers: convolution
  algebras and coherent sheaves}}, Comment. Math. Helv. \textbf{87} (2012),
  no.~2, 477--520.

\bibitem[Val34]{DuVal34}
P.~Du Val, \emph{{On isolated singularities of surfaces which do not affect the
  conditions of adjunction (Part I.)}}, Mathematical Proceedings of the
  Cambridge Philosophical Society \textbf{30} (1934), no.~4, 453–459.

\bibitem[Var79]{Va79}
J.~A. Vargas, \emph{{Fixed points under the action of unipotent elements of
  $SL_n$ in the flag variety}}, Bol. Soc. Mat. Mexicana \textbf{24} (1979),
  1--14.

\bibitem[Vit96]{Vi96}
C.~Viterbo, \emph{{Functors and Computations in Floer homology with
  Applications Part II}}, arXiv:1805.01316 (1996).

\bibitem[\v{Z}22]{FZ2}
F.~\v{Z}ivanović, \emph{{Equivariance of Maffei's isomorphism and smooth
  components of Springer fibres}}, in preparation, 2022.

\end{thebibliography}
\bibliographystyle{amsalpha}
\end{document}